\documentclass[a4paper,11pt]{amsart}

\usepackage{a4wide,amsmath,amsfonts,amssymb,mathrsfs,amsthm,bbm}
\usepackage[hyperfootnotes=false]{hyperref}

\allowdisplaybreaks[2]

\numberwithin{equation}{section}

\newtheorem{theorem}{Theorem}[section]
\newtheorem{lemma}[theorem]{Lemma}
\newtheorem{corollary}[theorem]{Corollary}
\newtheorem{remark}[theorem]{Remark}
\newtheorem{proposition}[theorem]{Proposition}

\newtheorem{assumption}[theorem]{Assumption}

\renewcommand{\P}{\mathbb{P}}
\renewcommand{\d}{\mathrm{d}}
\renewcommand{\epsilon}{\varepsilon}
\newcommand{\dd}{\,\mathrm{d}}

\newcommand{\R}{\mathbb{R}}
\newcommand{\N}{\mathbb{N}}
\newcommand{\E}{\mathbb{E}}

\newcommand{\Q}{\mathbb{Q}}
\newcommand{\sgn}{\mathop{\mathrm{sgn}}}

\title[Pathwise uniqueness for singular stochastic Volterra equations]{Pathwise uniqueness for singular stochastic Volterra equations with H{\"o}lder coefficients}

\author[Pr{\"o}mel]{David J. Pr{\"o}mel}
\address{David J. Pr{\"o}mel, University of Mannheim, Germany}
\email{proemel@uni-mannheim.de}

\author[Scheffels]{David Scheffels}
\address{David Scheffels, University of Mannheim, Germany}
\email{dscheffe@mail.uni-mannheim.de}

\date{\today}

\begin{document}

\begin{abstract}
  Pathwise uniqueness is established for a class of one-dimensional stochastic Volterra equations driven by Brownian motion with singular kernels and H{\"o}lder continuous diffusion coefficients. Consequently, the existence of unique strong solutions is obtained for this class of stochastic Volterra equations.
\end{abstract}

\maketitle





\section{Introduction}

In this paper we study one-dimensional stochastic Volterra equations (SVEs) of the form
\begin{align}\label{eq:SVE intro}
  X_t&=x_0(t)+\int_0^t (t-s)^{-\alpha}\mu(s,X_s)\dd s+\int_0^t (t-s)^{-\alpha}\sigma(s,X_s)\dd B_s, \quad t\in [0,T],
\end{align}
where $\alpha\in[0,\frac{1}{2})$, $x_0\colon [0,T]\to\R$ is a continuous function, $\mu,\sigma\colon [0,T]\times \mathbb{R}\to \mathbb{R}$ are measurable functions and $(B_t)_{t\in [0,T]}$ is a standard Brownian motion. Although the stochastic integral in~\eqref{eq:SVE intro} is defined as a classical stochastic It{\^o} integral, a potential solution of this SVE is, in general, neither a semimartingale nor a Markov process. Assuming that $\mu$ is Lipschitz continuous and $\sigma$ is $\xi$-H{\"o}lder continuous for $\xi \in (\frac{1}{2(1-\alpha)},1]$, we show that pathwise uniqueness for the SVE~\eqref{eq:SVE intro} holds and, consequently, that there exists a unique strong solution.

Stochastic Volterra equations have been investigated in probability theory starting with the seminal works of Berger and Mizel~\cite{Berger1980a,Berger1980b} and serve as mathematical models allowing, in particular, to represent dynamical systems with memory effects such as population growth, spread of epidemics and turbulent flows. Recently, stochastic Volterra equations of the form~\eqref{eq:SVE intro} with non-Lipschitz continuous coefficients have demonstrated to fit remarkably well historical and implied volatilities of financial markets, see e.g. \cite{Bayer2016}, motivating the use of so-called rough volatility models in mathematical finance, see e.g. \cite{AbiJaberElEuch2019b,ElEuch2019}. Moreover, SVEs with non-Lipschitz continuous coefficients like~\eqref{eq:SVE intro} arise as scaling limits of branching processes in population genetics, see~\cite{Mytnik2015,AbiJaber2021b}.

The existence of strong solutions and pathwise uniqueness for stochastic Volterra equations with sufficiently regular kernels and Lipschitz continuous coefficients are well-known due to classical results such as \cite{Berger1980a,Berger1980b,Protter1985}, which have been generalized in various directions, e.g., allowing for anticipating and path-dependent coefficients, see \cite{Pardoux1990,Oksendal1993,Alos1997,Kalinin2021}. As long as the kernels of a one-dimensional SVE are sufficiently regular, i.e. excluding the singular kernel $(t-s)^{-\alpha}$ in~\eqref{eq:SVE intro}, the existence of unique strong solutions can be still obtained when the diffusion coefficients are only $1/2$-H{\"o}lder continuous, see \cite{AbiJaberElEuch2019b,Promel2022b}. The latter results are crucially based on the observation that solutions to SVEs with sufficiently regular kernels are semimartingales, allowing to rather directly implement approaches in the spirit of Yamada--Watanabe~\cite{Yamada1971}. Assuming a Lipschitz condition on the coefficients, the existence of unique strong solutions to SVEs with singular kernels were proven in \cite{Cochran1995,Coutin2001} and a slight extension beyond Lipschitz continuous coefficients can be found in \cite{Wang2008}.

Similarly to the case of ordinary stochastic differential equations (SDEs), the regularity assumptions on the coefficients and on the kernels of a stochastic Volterra equation can be significantly relaxed by considering the concept of weak solutions instead of strong solutions. While weak solutions to a certain class of one-dimensional SVEs were first treated by Mytnik and Salisbury in~\cite{Mytnik2015}, a comprehensive study of weak solutions to stochastic Volterra equations of convolutional type was recently developed by Abi~Jaber, Cuchiero, Larsson and Pulido~\cite{AbiJaber2021}, see also \cite{AbiJaber2019,AbiJaber2021b}. By introducing a local martingale problem associated to SVEs of convolutional type, Abi~Jaber et al.~\cite{AbiJaber2021} derived the existence of weak solutions to SVEs of convolutional type with sufficiently integrable kernels and continuous coefficients. Assuming additionally that the coefficients of the SVE lead to affine Volterra processes, weak uniqueness was obtained in \cite{Mytnik2015,AbiJaberElEuch2019,AbiJaber2021b,Cuchiero2020}. The concept of weak solutions to SVEs with general kernels was investigated in~\cite{Promel2022}.

A major challenge to prove pathwise uniqueness for the SVE~\eqref{eq:SVE intro} with its singular kernel~$(t-s)^{-\alpha}$ is the missing natural semimartingale representation of its potential solution. Assuming the drift coefficient~$\mu$ does not depend on the solution~$(X_t)_{t\in [0,T]}$ and the diffusion coefficient~$\sigma$ is $\xi$-H{\"o}lder continuous for $\xi \in (\frac{1}{2(1-\alpha)},1]$, Mytnik and Salisbury~\cite{Mytnik2015} established pathwise uniqueness for the SVE~\eqref{eq:SVE intro} by equivalently reformulating the SVE into a stochastic partial differential equation, which then allows to accomplish a proof of pathwise uniqueness in the spirit of Yamada--Watanabe relying on the methodology developed in \cite{Mytnik2006,Mytnik2011}. In the present paper, we generalize the results and method of Mytnik and Salisbury~\cite{Mytnik2015} to derive pathwise uniqueness for the stochastic Volterra equation~\eqref{eq:SVE intro} with general time-inhomogeneous coefficients. As classical transforms allowing to remove the drift of an SDE are not applicable to the SVE~\eqref{eq:SVE intro}, the general time-inhomogeneous coefficients~$\mu$ creates severe novel challenges. For the sake of readability, all proofs are presented in a self-contained manner although some intermediate steps can already be found in the work~\cite{Mytnik2015} of Mytnik and Salisbury.

The existence of a unique strong solution to the stochastic Volterra equation~\eqref{eq:SVE intro} follows by a general version of Yamada--Watanabe theorem (see~\cite{Yamada1971,Kurtz2014}) stating that the combination of pathwise uniqueness and the existence of weak solutions to the SVE~\eqref{eq:SVE intro} (as obtained in \cite{Promel2022}) guarantees the existence of a strong solution. Let us remark that strong existence and pathwise uniqueness play a crucial role in the context of large deviation and as key ingredients to fully justify some numerical schemes, see e.g. \cite{Dupuis1997,Mao1994}.


\vspace{0.15cm}

\noindent \textbf{Organization of the paper:} Section~\ref{sec: main results} presents the main results on the pathwise uniqueness and strong existence of solutions to stochastic Volterra equations. Section~\ref{sec:pathwUni} contains the main steps in the proof of pathwise uniqueness, while the remaining Sections~\ref{sec: transformation to SPDE}-\ref{sec_step5} provide the necessary auxiliary results to implement these main steps.

\vspace{0.15cm}

\noindent\textbf{Acknowledgments:} D. Scheffels gratefully acknowledges financial support by the Research Training Group ``Statistical Modeling of Complex Systems'' (RTG 1953) funded by the German Science Foundation (DFG).

\section{Main results}\label{sec: main results}

Let $(\Omega,\mathcal{F},(\mathcal{F}_t)_{t\in [0,T]},\mathbb{P})$ be a filtered probability space, which satisfies the usual conditions, $(B_t)_{t\in [0,T]}$ be a standard Brownian motion and $T\in (0,\infty)$.  We consider the one-dimensional stochastic Volterra equation (SVE) 
\begin{align}\label{eq:SVE}
  X_t&=x_0(t)+\int_0^t (t-s)^{-\alpha}\mu(s,X_s)\dd s+\int_0^t (t-s)^{-\alpha}\sigma(s,X_s)\dd B_s, \quad t\in [0,T],
\end{align}
where $\alpha\in[0,\frac{1}{2})$, $x_0\colon [0,T]\to\R$ is a deterministic continuous function and $\mu,\sigma\colon [0,T]\times \mathbb{R}\to \mathbb{R}$ are deterministic, measurable functions. Furthermore, $\int_0^t (t-s)^{-\alpha}\mu(s,X_s)\dd s$ is defined as a Riemann--Stieltjes integral and $\int_0^t (t-s)^{-\alpha}\sigma(s,X_s)\dd B_s$ as an It{\^o} integral.

\medskip

The regularity of the coefficients~$\mu$ and~$\sigma$ and of the initial condition~$x_0$ is determined in the following assumption.

\begin{assumption}\label{ass:coefficients}
  Let $\alpha \in [0,\frac{1}{2})$, let $x_0$ be deterministic and $\beta$-H{\"o}lder continuous for every $\beta\in (0,\frac{1}{2}-\alpha)$ and let $\mu,\sigma\colon [0,T]\times\R\to\R $ be measurable functions such that
  \begin{enumerate}
    \item[(i)] $\mu$ and $\sigma$ are of linear growth, i.e. there is a constant $C_{\mu,\sigma}>0$ such that
    \begin{equation*}
      |\mu(t,x)|+|\sigma(t,x)|\leq C_{\mu,\sigma}(1+|x|),
    \end{equation*}
    for all $t\in [0,T]$ and $x\in\R$.
    \item[(ii)] $\mu$ is Lipschitz continuous and $\sigma$ is H{\"o}lder continuous in the space variable uniformly in time of order $\xi$ for some $\xi\in [\frac{1}{2},1]$ such that
    \begin{equation*}
      \xi > \frac{1}{2(1-\alpha)},
    \end{equation*}
     where in the case of $\alpha=0$ even equality is allowed. Hence, there are constants $C_\mu,C_\sigma>0$ such that
    \begin{equation*}
      |\mu(t,x)-\mu(t,y)|\leq C_\mu|x-y|
      \quad\text{and}\quad
      |\sigma(t,x)-\sigma(t,y)|\leq C_\sigma|x-y|^{\xi}
    \end{equation*}
    hold for all $t\in [0,T]$ and $x,y\in \R$.
    \item[(iii)] For every $K>0$, there is some constant $C_K>0$ such that, for every $t\in[0,T]$ and every $x,y\in[-K,K]$,
    \begin{equation*}
      \bigg| \frac{\mu(t,x)-\mu(t,y)}{\sigma(t,x)-\sigma(t,y)}\bigg| \leq C_K,
    \end{equation*}
    where we use the convention $0/0:=1$.
  \end{enumerate}
\end{assumption} 

Assumption~\ref{ass:coefficients} is a standing assumption throughout the entire paper. Although not always explicitly stated all results are proven supposing Assumption~\ref{ass:coefficients}.

\begin{remark}
  Assumption~\ref{ass:coefficients}~(iii) is, for example, satisfied by any Lipschitz continuous functions~$\mu$ and $\sigma$ of the form $\sigma(t,x)=\mathop{\mathrm{sgn}}(x)|x|^\xi$ for $\xi\in[1/2,1]$. Note that, in interesting cases like the rough Heston model in mathematical finance, solutions to \eqref{eq:SVE} are non-negative (see \cite[Theorem~A.2]{AbiJaberElEuch2019}), so that the $\sgn$ in the definition of $\sigma$ does not influence the dynamics of the associated SVE. Then, for $|x|,|y|\leq K$, using the inequality $\big|\mathop{\mathrm{sgn}}(x)|x|^\xi-\mathop{\mathrm{sgn}}(y)|y|^\xi\big| \geq K^{-1}|x-y|$, we get
  \begin{equation*}
    \bigg| \frac{\mu(t,x)-\mu(t,y)}{\sigma(t,x)-\sigma(t,y)} \bigg|
    \leq C_\mu \frac{|x-y|}{\big|\sgn(x)|x|^\xi -\sgn(y)|y|^\xi\big|}
    \leq C_\mu \frac{|x-y|}{K^{-1}{|x-y|}}=C_\mu K<\infty.
  \end{equation*}
 
  Nevertheless, while Assumption~\ref{ass:coefficients}~(iii) is crucial for applying a Girsanov transformation in the proof of Theorem~\ref{thm:thm1} below, it is not a necessary condition. Indeed, if $\sigma$ does only depends on $t$, the Assumption~\ref{ass:coefficients}~(iii) cannot be satisfied for general Lipschitz continuous functions~$\mu$, but there exists a unique strong solution by classical results, see e.g. \cite{Wang2008}.
\end{remark}

Based on Assumption~\ref{ass:coefficients}, we obtain a unique strong solution of the stochastic Volterra equation~\eqref{eq:SVE}. Therefore, let us briefly recall the concepts of strong solutions and pathwise uniqueness. Let for $p\geq 1$, $L^p(\Omega\times [0,T])$ be the space of all real-valued, $p$-integrable functions on $\Omega\times [0,T]$. An $(\mathcal{F}_t)_{t\in[0,T]}$-progressively measurable stochastic process $(X_t)_{t\in [0,T]}$ in $L^p(\Omega\times [0,T])$, on the given probability space $(\Omega,\mathcal{F},(\mathcal{F}_t)_{t\in[0,T]},\mathbb{P})$, is called \textit{(strong) $L^p$-solution} to the SVE~\eqref{eq:SVE} if $ \int_0^t (|(t-s)^{-\alpha}\mu(s,X_s)|+|(t-s)^{-\alpha}\sigma(s,X_s)|^2 )\dd s<\infty$ for all $t\in[0,T]$ and the integral equation~\eqref{eq:SVE} holds a.s. We call a strong $L^1$-solution often just \textit{solution} to the SVE~\eqref{eq:SVE}. We say \textit{pathwise uniqueness} in $L^p(\Omega\times [0,T])$ holds for the SVE~\eqref{eq:SVE} if $\mathbb{P}(X_t=\tilde{X}_t, \,\forall t\in [0,T])=1$ for two $L^p$-solutions $(X_t)_{t\in[0,T]}$ and $(\tilde{X}_t)_{t\in[0,T]}$ to the SVE~\eqref{eq:SVE} defined on the same probability space $(\Omega,\mathcal{F},(\mathcal{F}_t)_{t\in[0,T]},\mathbb{P})$. Moreover, we say there exists a \textit{unique strong $L^p$-solution} $(X_t)_{t\in [0,T]}$ to the SVE~\eqref{eq:SVE} if $(X_t)_{t\in [0,T]}$ is a strong $L^p$-solution to the SVE~\eqref{eq:SVE} and pathwise uniqueness in $L^p$ holds for the SVE~\eqref{eq:SVE}. We say $(X_t)_{t\in [0,T]}$ is $\beta$-H{\"o}lder continuous for $\beta \in (0,1]$ if there exists a modification of $(X_t)_{t\in [0,T]}$ with sample paths that are almost all $\beta$-H{\"o}lder continuous.

\medskip

Note that the kernels $K_\mu(s,t)=K_\sigma(s,t)=(t-s)^{-\alpha}$ with $\alpha\in(0,1/2)$ fulfill the assumptions of Lemma~3.1 and Lemma~3.4 in \cite{Promel2022} for every
\begin{equation*}
  \epsilon\in\bigg(0,\frac{1}{\alpha}-2\bigg)
\end{equation*}
with 
\begin{equation*}
  \gamma=\frac{1}{2+\epsilon}-\alpha.
\end{equation*}
This means that, to use the results of \cite[Lemma~3.1 and Lemma~3.4]{Promel2022}, we need to consider $L^p$-solutions with
\begin{equation}\label{def_p_max}
  p>\max\bigg\lbrace \frac{1}{\gamma},1+\frac{2}{\epsilon} \bigg\rbrace = \max\bigg\{ \frac{
  2+\epsilon}{1-2\alpha-\epsilon\alpha},1+\frac{2}{\epsilon} \bigg\rbrace.
\end{equation}
The maximum in \eqref{def_p_max} is attained for $\epsilon^\star=\frac{1-2\alpha}{1+\alpha}$. Hence, inserting $\epsilon^\star$ into \eqref{def_p_max}, we consider in the following $L^p$-solutions and $L^p$-pathwise uniqueness for some 
\begin{equation}\label{def_p}
  p>3+\frac{6\alpha}{1-2\alpha}.
\end{equation}

\medskip

The following theorem states that pathwise uniqueness for the stochastic Volterra equation~\eqref{eq:SVE} holds, which is the main result of the present work.

\begin{theorem}\label{thm:main}
  Suppose Assumption~\ref{ass:coefficients} and let $p$ be given by \eqref{def_p}. Then, $L^p$-pathwise uniqueness holds for the stochastic Volterra equation~\eqref{eq:SVE}.
\end{theorem}

The proof of Theorem~\ref{thm:main} will be summarized in Section~\ref{sec:pathwUni} and the subsequent Sections~\ref{sec: transformation to SPDE}-\ref{sec_step5} provide the necessary auxiliary results. Relying on the pathwise uniqueness and the classical Yamada--Watanabe theorem, we get the existence of a unique strong solution.

\begin{corollary}
  Suppose Assumption~\ref{ass:coefficients} and let $p$ be given by \eqref{def_p}. Then, there exists a unique strong $L^p$-solution to the stochastic Volterra equation~\eqref{eq:SVE}.
\end{corollary}

\begin{proof}
  The $L^p$-pathwise uniqueness is provided by Theorem~\ref{thm:main}. The existence of a strong $L^p$-solution follows by the existence of a weak $L^p$-solution to the stochastic Volterra equation~\eqref{eq:SVE}, which is provided by \cite[Theorem~3.3]{Promel2022b}, which is applicable since the kernel $(t-s)^{-\alpha}$, $\alpha\in[0,\frac{1}{2})$, fulfills the required assumptions of \cite[Theorem~3.3]{Promel2022b}, cf. \cite[Remark~3.5]{Promel2022b}. Thanks to Yamada--Watanabe's theorem (see \cite[Corollary~1]{Yamada1971}, or \cite[Theorem~1.5]{Kurtz2014} for a generalized version), the existence of a weak $L^p$-solution and pathwise $L^p$-uniqueness imply the existence of a unique strong $L^p$-solution.
\end{proof}

Furthermore, we obtain the following regularity properties of solutions to the SVE~\eqref{eq:SVE}.

\begin{lemma}\label{lem:regularity}
  Suppose Assumption~\ref{ass:coefficients}, and let $(X_t)_{t\in[0,T]}$ be a strong $L^p$-solution to the stochastic Volterra equation~\eqref{eq:SVE} with $p$ given by \eqref{def_p}. Then, $\sup_{t\in[0,T]}\E[|X_t|^q]<\infty$ for any $q\geq 1$ and the sample paths of $(X_t)_{t\in[0,T]}$ are $\beta$-H{\"o}lder continuous for any $\beta\in (0,\frac{1}{2}-\alpha)$.
\end{lemma}

\begin{proof}
  The statements follow by \cite[Lemma~3.1 and Lemma~3.4]{Promel2022} since the kernel $(t-s)^{-\alpha}$ fulfills the regularity assumption of \cite[Lemma~3.1 and Lemma~3.4]{Promel2022} as shown in \cite[Remark~3.5]{Promel2022b}.
\end{proof}

For $k\in \N\cup\lbrace\infty\rbrace$, we write $C^k(\R)$, $C^k(\R_+)$ and $C^k([0,T]\times\R)$ for the spaces of continuous functions mapping from $\R$, $\R_+$ resp. $[0,T]\times\R$ to $\R$, that are $k$-times continuously differentiable. We use an index $0$ to indicate compact support, e.g.~$C_0^\infty(\R)$ denotes the space of smooth functions with compact support on $\R$. The space of square integrable functions $f\colon \R\to\R$ is denoted by $L^2(\R)$ and equipped with the usual scalar product $\langle \cdot, \cdot \rangle$. Moreover, a ball in $\R$ around~$x$ with radius $R>0$ is defined by $B(x,R):=\{y\in \R: |y-x|\leq R\}$ and we use the notation $A_{\eta}\lesssim B_{\eta}$ for a generic parameter~$\eta$, meaning that $A_{\eta}\le CB_{\eta}$ for some constant $C>0$ independent of $\eta$.

\section{Proof of pathwise uniqueness}\label{sec:pathwUni}

We prove Theorem~\ref{thm:main} by generalizing the well-known techniques of Yamada--Watanabe (cf. \cite[Theorem~1]{Yamada1971}) and the work of Mytnik and Salisbury~\cite{Mytnik2015}. One of the main challenges is the missing semimartingale property of a solution $(X_t)_{t\in[0,T]}$ to the SVE~\eqref{eq:SVE}. Therefore, we transform \eqref{eq:SVE} into a random field in Step~1, for which we can derive a semimartingale decomposition in \eqref{eq: main proof semimartingale}. Then, we implement an approach in the spirit of Yamada--Watanabe in Step~2-5 and conclude the pathwise uniqueness by using a Gr{\"o}nwall inequality for weak singularities in Step~6.

\begin{proof}[Proof of Theorem~\ref{thm:main}]
  Suppose there are two strong $L^p$-solutions $(X^1_t)_{t\in[0,T]}$ and $(X^2_t)_{t\in[0,T]}$ to the stochastic Volterra equation~\eqref{eq:SVE}.

  \textit{Step~1:} To induce a semimartingale structure, we introduce the random fields
  \begin{equation}
    X^i(t,x):=x_0(t)+\int_0^t p_{t-s}^\theta(x)\mu(s,X_s^i)\dd s+\int_0^tp_{t-s}^\theta(x)\sigma(s,X_s^i)\dd B_s,\label{def:randomfields}
  \end{equation}
  for $t\in [0,T]$, $x\in \R$ and $i=1,2$, where the densities $p_t^\theta\colon \R\to\R$ and $\theta:=1/2-\alpha$ are defined in \eqref{def_pt}. By Proposition~\ref{prop:step1}, we get that $X^i\in C([0,T]\times \R)$ and
  \begin{align}\label{eq: main proof semimartingale}
    \begin{split}
    \int_{\mathbb{R}}X^i(t,x)\Phi_t(x)\dd x
    &=\int_{\mathbb{R}}\bigg( x_0\Phi_0(x)+\int_0^t\Phi_s(x)\frac{\partial}{\partial s}x_0(s)\dd s\bigg)\dd x\\
    &\quad +\int_0^t\int_{\mathbb{R}}X^i(s,x)\bigg( \Delta_\theta\Phi_s(x)+\frac{\partial}{\partial s}\Phi_s(x) \bigg)\dd x\dd s\\
    &\quad+\int_0^t\mu(s,X^i(s,0))\Phi_s(0)\dd s
    +\int_0^t\sigma(s,X^i(s,0))\Phi_s(0)\dd B_s,
    \end{split}
  \end{align}
  for $t\in [0,T]$ and every $\Phi\in C^2_0([0,T]\times \mathbb{R})$, where the differential operator~$\Delta_\theta$ is defined in~\eqref{def:Deltatheta} and $\frac{\partial}{\partial s}x_0(s)$ is meant in the sense of distributions. Notice, due to \eqref{eq: main proof semimartingale}, the stochastic process $t\mapsto \int_{\mathbb{R}}X^i(t,x)\Phi_t(x)\dd x$ is a semimartingale and $X^i(t,0)=X^i_t$ for $t\in [0,T]$.

  \textit{Step~2:} We define suitable sequences $(\Phi_x^m)\subset C^2_0(\mathbb{R})$, for $x\in \R$, and $(\phi_n)\subset C^\infty(\R)$ of test functions, see \eqref{phi} and \eqref{def:phi_n} for the precise definitions, such that
  \begin{equation*}
    \Phi_x^m \to \delta_x\quad \text{as }m\to \infty,\quad \text{for every } x\in \R, \quad \text{and}\quad \phi_n\to |\cdot|\quad \text{as }n\to \infty.
  \end{equation*}
  Applying Proposition~\ref{prop:step2} (which is based on It{\^o}'s formula and \eqref{eq: main proof semimartingale}) and setting $\tilde{X}(t):=\tilde{X}(t,\cdot):=X^1(t,\cdot)-X^2(t,\cdot)$ for $t\in [0,T]$, we get
  \begin{align*}
    \phi_n(\langle \tilde{X}(t),\Phi_{x}^m \rangle)
    &=\int_0^t  \phi_n'(\langle \tilde{X}(s), \Phi_x^m \rangle)\langle \tilde{X}(s),\Delta_\theta \Phi_{x}^m \rangle \dd s \\
    &\quad+\int_0^t \phi_n'(\langle \tilde{X}(s),\Phi_x^m \rangle)\Phi_x^m(0) \big( \mu(s,X^1(s,0))-\mu(s,X^2(s,0)) \big)\dd s\\
    &\quad+\int_0^t  \phi_n'(\langle \tilde{X}(s), \Phi_x^m \rangle)\Phi_x^m(0) \big( \sigma(s,X^1(s,0))-\sigma(s,X^2(s,0))\big) \dd B_s\\
    &\quad+
    \frac{1}{2}\int_0^t  \psi_n(| \langle \tilde{X}(s),\Phi_x^m \rangle|)\Phi_x^m(0)^2 \big( \sigma(s,X^1(s,0))-\sigma(s,X^2(s,0)) \big)^2\dd s ,
  \end{align*}
  where $\langle \cdot, \cdot\rangle$ denotes the scalar product on $L^2(\R)$.

  \textit{Step~3:} To implement an approach in the spirit of Yamada--Watanabe, we need to introduce another suitable test function $\Psi \in  C([0,T]\times\R)$ (satisfying Assumption~\ref{ass:Psi} below). Denoting by $\dot{\Psi}:=\frac{\partial}{\partial s}\Psi$ the time derivative of $\Psi$, Proposition~\ref{prop:Psi_step3} leads to
  \begin{align*}
    &\langle \phi_n(\langle \tilde{X}(t),\Phi_{\cdot}^m \rangle),\Psi_t \rangle\\
    &\quad= \int_0^t \langle \phi_n'(\langle \tilde{X}(s),\Phi_\cdot^m \rangle)\langle \tilde{X}(s),\Delta_\theta \Phi_{\cdot}^m \rangle,\Psi_s \rangle \dd s\\
    &\quad\quad+\int_0^t \langle \phi_n'(\langle \tilde{X}(s),\Phi_\cdot^m \rangle)\Phi_\cdot^m(0),\Psi_s \rangle \big( \mu(s,X^1(s,0))-\mu(s,X^2(s,0)) \big)\dd s\\
    &\quad\quad +\int_0^t \langle \phi_n'(\langle \tilde{X}(s),\Phi_\cdot^m \rangle)\Phi_\cdot^m(0),\Psi_s \rangle \big( \sigma(s,X^1(s,0))-\sigma(s,X^2(s,0)) \big)\dd B_s\\
    &\quad\quad+\frac{1}{2}\int_0^t \langle \psi_n(| \langle \tilde{X}(s),\Phi_\cdot^m \rangle|)\Phi_\cdot^m(0)^2,\Psi_s \rangle \big ( \sigma(s,X^1(s,0))-\sigma(s,X^2(s,0)) \big ) \dd s\\
    &\quad\quad +\int_0^t \langle \phi_n(\langle \tilde{X}(s),\Phi_\cdot^m \rangle),\dot{\Psi}_s \rangle \dd s.
  \end{align*}

  \textit{Step~4:} Using the stopping time $T_{\xi,K}$ defined in \eqref{def:T_xi,K}, taking expectations and sending $n,m\to \infty$, Proposition~\ref{prop:step4} states that
  \begin{align*}
    &\E\big[\big\langle | \langle \tilde{X}(t\wedge T_{\xi,K})|,\Psi_{t\wedge T_{\xi,K}} \big\rangle ] \\
    &\quad\lesssim  \mathbb{E} \bigg [ \int_0^{t\wedge T_{\xi,K}}\int_{\mathbb{R}}|\tilde{X}(s,x)|\Delta_\theta \Psi_s(x)\dd x\dd s \bigg ]\\
    &\quad\quad +\int_0^{t\wedge T_{\xi,K}}\Psi_s(0)\mathbb{E}[|\tilde{X}(s,0)|]\dd s
    + \mathbb{E}\bigg[ \int_0^{t\wedge T_{\xi,K}}\int_{\mathbb{R}}|\tilde{X}(s,x)|\dot{\Psi}_s(x)\dd x\dd s \bigg].
  \end{align*}

  \textit{Step~5:} Since $T_{\xi,K}\to T$ as $K\to\infty$ a.s. by Corollary~\ref{Cor:StoppzeitggT}, applying Fatou's lemma yields
  \begin{align}
    \int_{\mathbb{R}}\mathbb{E}[ |\tilde{X}(t,x)|]\Psi_t(x)\dd x
    &\lesssim \int_0^{t}\int_{\mathbb{R}}\mathbb{E}[|\tilde{X}(s,x)|] | \Delta_\theta \Psi_s(x)+\dot{\Psi}_s(x)|\dd x\dd s \notag\\
    &\quad+ \int_0^t \Psi_s(0) \mathbb{E}[|\tilde{X}(s,0)|]\dd s.\label{eq:fastfertig}
  \end{align}
  Finally, we choose appropriate test functions $(\Psi_{N,M})_{N,M\in\N}$ (satisfying Assumption~\ref{ass:Psi}) to approximate the Dirac distribution around $0$ with $\Psi_{N,M}(t,\cdot)$. Thus,  choosing $\Psi_t(x)=\Psi_{N,M}(t,x)$ in \eqref{eq:fastfertig} and sending $N,M\to\infty$ yields, by Proposition~\ref{prop:step5}, that
  \begin{equation*}
    \mathbb{E}[ |\tilde{X}(t,0)|]\lesssim \int_0^t (t-s)^{-\alpha}\mathbb{E}[|\tilde{X}(s,0)| ]\dd s, \quad t\in [0,T].
  \end{equation*}

  \textit{Step~6:} Due to $\alpha\in (0,\frac{1}{2})$, Gr{\"o}nwall's inequality for weak singularities (see e.g. \cite[Lemma~A.2]{Kruse2014}) reveals
  \begin{equation*}
    \mathbb{E}[ |\tilde{X}(t,0)|]=0, \quad t\in [0,T],
  \end{equation*}
  and therefore $X^1_t=X^2_t=0$ a.s. By the continuity of $X^1$ and $X^2$ (see Lemma~\ref{lem:regularity}), we conclude the claimed pathwise uniqueness.
\end{proof}

\section{Step~1: Transformation into an SPDE}\label{sec: transformation to SPDE}

Recall, in general, a solution $(X_t)_{t\in[0,T]}$ of the SVE~\eqref{eq:SVE} will not be a semimartingale due to the $t$-dependence of the kernel. In this section we will transform the SVE~\eqref{eq:SVE} into a stochastic partial differential equation (SPDE) in distributional form, see \eqref{eq: main proof semimartingale}, which allows us to recover a semimartingale structure and, thus, to implement an approach in the spirit of Yamada--Watanabe.

\medskip

To that end, we consider the evolution equation
\begin{align}\label{eq: evol_eq}
  \begin{split}
  \frac{\partial u}{\partial t}(t,x)&=\Delta_\theta u(t,x),\quad t\in [0,T],\,x\in \R_+,\\
  u(0,x)&=\delta_0(x),
  \end{split}
\end{align}
where the differential operator $\Delta_\theta$ is defined by
\begin{equation}\label{def:Deltatheta}
  \Delta_\theta := \frac{2}{(2+\theta)^2} \frac{\partial}{\partial x}|x|^{-\theta}\frac{\partial}{\partial x}
\end{equation}
for some constant $\theta>0$. Note that we will later also consider the evolution equation \eqref{eq: evol_eq} on $t\in [0,T]$, $x\in\R$. It can be seen that the following densities solve \eqref{eq: evol_eq} if restricted to $x\in\R_+$:
\begin{equation}\label{def_pt}
  p_t^\theta(x):=c_\theta t^{-\frac{1}{2+\theta}}e^{-\frac{|x|^{2+\theta}}{2t}},\quad t\in [0,T],\,x\in\R_+,
\end{equation}
which we extend to $\R$ by setting
\begin{equation*}
  p_{t}^\theta(x):=p_{t}^\theta(|x|), \quad t\in [0,T],\, x\in \R.
\end{equation*}

Since $\int_0^\infty p_t^\theta(x)\dd x$ is independent of $t\in(0,T]$, one can verify that if we choose the constant
\begin{align}
  c_\theta := (2+\theta)2^{-\frac{1}{2+\theta}}\Gamma\bigg(\frac{1}{2+\theta}\bigg)^{-1},\label{def:c_theta}
\end{align}
where $\Gamma$ denotes the Gamma function, then $p_t^\theta\colon \R_+\to\R_+$ defines a probability density on $\R_+$. The reason, why we consider \eqref{eq: evol_eq}, is that by the choice of $\theta>0$ such that
\begin{equation}\label{choose_theta}
  \alpha=\frac{1}{2+\theta},
\end{equation}
we get that for $x=0$ the solution $p_{t-s}^\theta(0)$ represents the kernel in the SVE~\eqref{eq:SVE} up to a constant. Therefore, we obtain the following lemma.

\begin{lemma}\label{lem:equuiv}
  Every strong $L^p$-solution $(X_t)_{t\in [0,T]}$ of the SVE~\eqref{eq:SVE} defines an a.s.~continuous strong solution $(X(t,x))_{t\in[0,T],x\in\R}$ of
  \begin{align}\label{SPDE}
    X(t,x)=x_0(t)&+\int_0^t p_{t-s}^\theta(x)\mu(s,X(s,0))\dd s\\
    &+\int_0^tp_{t-s}^\theta(x)\sigma(s,X(s,0))\dd B_s,\quad t\in[0,T],x\in\R,\notag
  \end{align}
  with $\theta>0$ chosen such that \eqref{choose_theta} holds, i.e., on the probability space $(\Omega,\mathcal{F},(\mathcal{F}_t)_{t\in[0,T]},\P)$, there is a random field $(X(t,x))_{t\in[0,T],x\in\R}$ such that $X\in C([0,T]\times\R)$ a.s., $(X(t,x))_{t\in[0,T]}$ is $(\mathcal{F}_t)$-progressively measurable for $x\in\R$,
  \begin{equation*}
    \int_0^t\big(|p_{t-s}^\theta(x)\mu(s,X(s,0))|+|p^\theta_{t-s}(x)\sigma(s,X(s,0))|^2 \big)\dd s<\infty
  \end{equation*}
  and \eqref{SPDE} holds a.s. Conversely, every strong solution of \eqref{SPDE} defines a strong solution of the stochastic Volterra equation~\eqref{eq:SVE}.
\end{lemma}

\begin{proof}
  First, we assume that there is a solution to the SVE~\eqref{eq:SVE}. This implies a solution $Y$ to the SVE
  \begin{equation*}
    Y_t=x_0(t)+\int_0^t p_{t-s}^\theta(0)\mu(s,Y_s)\dd s+\int_0^t p_{t-s}^\theta(0)\sigma(s,Y_s)\dd B_s.
  \end{equation*}
  We define, for $t\in [0,T], x\in \mathbb{R}$,
  \begin{equation*}
    X(t,x):=x_0(t)+\int_0^t p_{t-s}^\theta(x)\mu(s,Y_s)\dd s+\int_0^tp_{t-s}^\theta(x)\sigma(s,Y_s)\dd B_s.
  \end{equation*}
  Then, by obtaining $X(t,0)=Y_t$, $X$ solves
  \begin{equation*}
    X(t,x)=x_0(t)+\int_0^t p_{t-s}^\theta(x)\mu(s,X(s,0))\dd s+\int_0^tp_{t-s}^\theta(x)\sigma(s,X(s,0))\dd B_s.
  \end{equation*}
  By the adaptedness of the It{\^o} integral and the Riemann--Stieltjes integral, $(X(t,x))_{t\in[0,T]}$ is $(\mathcal{F}_t)$-progressively measurable for every $x\in \R$. By the continuity of $p_t^\theta(x)$, $X(t,x)$ is continuous in $x$-direction. By the continuity of the initial condition $x_0$ and the integrals, it is also continuous in $t$-direction.
  
  Conversely, if $X=(X(t,x))_{t\in[0,T],x\in\R}$ solves \eqref{SPDE}, $Y_t:=X(t,0)$ is a solution of \eqref{eq:SVE}.
\end{proof}

Due to the transformation of the SVE~\eqref{eq:SVE} into the SPDE~\eqref{SPDE}, we shall study continuous solutions $X\in C([0,T]\times\mathbb{R})$ of the SPDE~\eqref{SPDE} instead of solutions to the SVE~\eqref{eq:SVE} directly. The next goal is to derive a regularity result for solutions of the SPDE~\eqref{SPDE}. For this purpose, we first investigate the densities $p_t^{\theta}$. We introduce some auxiliary lemmas, which are helpful for a better understanding of the densities~$p_t^\theta$, and skip the dependence on $\theta$ by writing
\begin{equation*}
  p_t(x):=ct^{-\alpha}e^{-\frac{|x|^{\frac{1}{\alpha}}}{2t}}\quad\text{for a fixed } \alpha\in (0, 1/2).
\end{equation*}

\begin{lemma}\label{lem:6.1}
  For any $x,y\in\mathbb{R}$, $t\in[0,T]$ and $\beta\in [0,1]$, one has
  \begin{equation*}
    |p_t(x)-p_t(y)|
    \lesssim t^{-\alpha}\bigg( \frac{|x-y|}{t}\bigg)^\beta  \max (|x|,|y|)^{(\frac{1}{\alpha}-1)\beta}.
  \end{equation*}
\end{lemma}

\begin{proof}
  First, let us fix $t\in[0,T]$ and consider the function $x\mapsto e^{-\frac{|x|^{1/\alpha}}{2t}}$. By applying the mean value theorem and assuming w.l.o.g. $|y|<|x|$, we obtain, for some $z\in [|y|,|x|]$,
  \begin{equation*}
    \frac{e^{-\frac{|x|^{\frac{1}{\alpha}}}{2t}} - e^{-\frac{|y|^{\frac{1}{\alpha}}}{2t}}}{|x|-|y|}
    = -\frac{z^{\frac{1}{\alpha}-1}}{2t\alpha}e^{-\frac{z^{1/\alpha}}{2t}},
  \end{equation*}
  which reveals that
  \begin{equation}\label{ineq13}
    \bigg| e^{-\frac{|x|^{\frac{1}{\alpha}}}{2t}}-e^{-\frac{|y|^{\frac{1}{\alpha}}}{2t}}\bigg| \leq \frac{|x-y|}{2t\alpha}|x|^{\frac{1}{\alpha}-1}.
  \end{equation}
  Using inequality \eqref{ineq13} and $\beta\in [0,1]$, we bound
  \begin{equation*}
    |p_t(x)-p_t(y)|
    \lesssim t^{-\alpha}\bigg| e^{-\frac{|x|^{\frac{1}{\alpha}}}{2t}}-e^{-\frac{|y|^{\frac{1}{\alpha}}}{2t}}\bigg|^\beta
    \lesssim t^{-\alpha}\bigg( \frac{|x-y|}{t} \bigg)^\beta \max(|x|,|y|)^{(\frac{1}{\alpha}-1)\beta}.
  \end{equation*}
\end{proof}

\begin{corollary}
  For any $x,y\in [-1,1]$, $t\in[0,T]$ and $\beta\in (0,1-\alpha)$, one has
  \begin{equation*}
    \int_0^t|p_s(x)-p_s(y)| \dd s\lesssim |x-y|^\beta.
  \end{equation*}
\end{corollary}

\begin{proof}
  By Lemma~\ref{lem:6.1}, we see that
  \begin{align*}
    \int_0^t|p_s(x)-p_s(y)| \dd s
    &\lesssim \int_0^t s^{-\alpha}\bigg(\frac{|x-y|}{s}\bigg)^\beta \max (|x|,|y|)^{(\frac{1}{\alpha}-1)\beta}\dd s\\
    &\lesssim |x-y|^\beta \int_0^t s^{-\alpha-\beta}\dd s\lesssim |x-y|^\beta.
  \end{align*}
\end{proof}

\begin{lemma}\label{lem:6.2}
  For any $0<t<t'\leq T$ and $x\in\mathbb{R}$, one has
  \begin{equation*}
    \int_0^t ( p_{t'-s}(x)-p_{t-s}(x))^2\dd s \lesssim |t'-t|^{1-2\alpha}.
  \end{equation*}
\end{lemma}

\begin{proof}
  We assume w.l.o.g. that $t'-t\leq t$ and use the linearity of the integral together with $|e^{-x}|\leq 1$ for non-negative $x$ to get
  \begin{align*}
    \int_0^t | p_{t'-s}(x)-p_{t-s}(x)|^{2}\dd s
    &\lesssim \int_{t-|t'-t|}^t | (t'-s)^{-\alpha}-(t-s)^{-\alpha}|^{2}\dd s\\
    &\quad+\int_0^{t-|t'-t|} |p_{t'-s}(x)-p_{t-s}(x) |^{2}\dd s\\
    &\lesssim \int_{t-|t'-t|}^t (t-s)^{-2\alpha}\dd s\\
    &\quad+\int_0^{t-|t'-t|} | (t-s)^{-\alpha}-(t'-s)^{-\alpha} |^{2}e^{-\frac{|x|^{\frac{1}{\alpha}}}{2(t-s)}}\dd s\\
    &\qquad+\int_0^{t-|t'-t|} (t'-s)^{-2\alpha}\bigg|e^{-\frac{|x|^{\frac{1}{\alpha}}}{2(t-s)}}-e^{-\frac{|x|^{\frac{1}{\alpha}}}{2(t'-s)}}\bigg|\dd s\\
    &=:I_1+I_2+I_3.
  \end{align*}
  For $I_1$, we directly compute
  \begin{equation*}
    I_1=\bigg[ \frac{-(t-s)^{1-2\alpha}}{1-2\alpha} \bigg]_{t-|t'-t|}^t \lesssim |t'-t|^{1-2\alpha}.
  \end{equation*}
  For $I_2$, we use $|a-b|^2\leq a^2-b^2$ for $a>b$ to bound
  \begin{align*}
    I_2
    &\leq \int_0^{t-|t'-t|} (t-s)^{-2\alpha}
    \dd s - \int_0^{t-|t'-t|} (t'-s)^{-2\alpha}\dd s\\
    &=\bigg[ \frac{-(t-s)^{1-2\alpha}}{1-2\alpha} \bigg]_0^{t-|t'-t|} - \bigg[ \frac{-(t'-s)^{1-2\alpha}}{1-2\alpha} \bigg]_{0}^{t-|t'-t|}\\
    &\lesssim |t'-t|^{1-2\alpha}.
  \end{align*}
  For $I_3$, we use the mean value theorem for the function $t\mapsto e^{-\frac{|x|^{\frac{1}{\alpha}}}{2(t-s)}}$, similarly as we did in~\eqref{ineq13}, to get the inequality
  \begin{equation*}
    \bigg| e^{-\frac{|x|^{\frac{1}{\alpha}}}{2(t-s)}}-e^{-\frac{|x|^{\frac{1}{\alpha}}}{2(t'-s)}}\bigg| \leq (t'-t)\frac{|x|^{\frac{1}{\alpha}}}{2(t-s)^2} e^{-\frac{|x|^{\frac{1}{\alpha}}}{2(t'-s)}}.
  \end{equation*}
  Using this and the inequality $e^{-x}\leq x^{-1}$ for all $x\geq 0$, such as $\frac{t'-t}{t-s}\leq 1$ and $\frac{t'-s}{t-s}\leq \frac{2(t-s)}{t-s}=2$ due to $s\leq t-|t'-t|$, we get
  \begin{align*}
    I_3
    &\leq (t'-t) \int_0^{t-|t'-t|} (t-s)^{-2\alpha}\bigg(\frac{|x|^{\frac{1}{\alpha}}}{2(t-s)^2} e^{-\frac{|x|^{\frac{1}{\alpha}}}{2(t'-s)}}\bigg)
    \dd s\\
    &\lesssim \int_0^{t-|t'-t|} (t-s)^{-2\alpha}\frac{(t'-t)(t'-s)}{(t-s)^2}\dd s\\
    &\lesssim \int_0^{t-|t'-t|} (t-s)^{-2\alpha} \dd s
    \lesssim |t'-t|^{1-2\alpha},
  \end{align*}
  which yields the statement.
\end{proof}

\begin{lemma}\label{lem:6.3}
  For any $x,y\in [-1,1]$, $t\in[0,T]$ and $\beta\in (0,\frac{1}{2}-\alpha)$, one has
  \begin{equation*}
    \int_0^t (p_{t-s}(x)-p_{t-s}(y))^2 \dd s\lesssim  \max \big(|x|,|y|\big)^{(\frac{1}{\alpha}-1)2\beta}|x-y|^{1-2\alpha}.
  \end{equation*}
\end{lemma}

\begin{proof}
  W.l.o.g. we may assume $t\geq |x-y|$ and split the integral into
  \begin{align*}
    \int_0^t ( p_{t-s}(x)-p_{t-s}(y))^{2} \dd s
    &\leq \int_0^{t-|x-y|} ( p_{t-s}(x)-p_{t-s}(y) )^{2} \dd s\\
    &\quad + \int_{t-|x-y|}^t ( p_{t-s}(x)-p_{t-s}(y))^{2} \dd s\\
    &=:I_1+I_2.
  \end{align*}
  
  For $I_1$, we apply Lemma~\ref{lem:6.1} with $\beta=1$ to get
  \begin{align*}
    I_1
    &\lesssim  \max (|x|,|y|)^{(\frac{1}{\alpha}-1)2} \int_0^{t-|x-y|} |x-y|^{2} (t-s)^{-2\alpha-2}\dd s\\
    &=  \max (|x|,|y|)^{(\frac{1}{\alpha}-1)2}|x-y|^{2}  \bigg[ \frac{-(t-s)^{1-2\alpha-2}}{1-2\alpha-2}\ \bigg]_0^{t-|x-y|}\\
    &\lesssim  \max (|x|,|y|)^{(\frac{1}{\alpha}-1)2}|x-y|^{2}  \big( t^{-2\alpha-1} +|x-y|^{-2\alpha-1}  \big)\\
    &\lesssim  \max (|x|,|y|)^{(\frac{1}{\alpha}-1)2\beta}|x-y|^{1-2\alpha}
  \end{align*}
  with $t\geq |x-y|$.
  
  For $I_2$, Lemma~\ref{lem:6.1} again, but with $\beta\in (0,1/2-\alpha)$ such that $2\alpha+2\beta<1$, yields
  \begin{align*}
    I_2
    &\lesssim \max (|x|,|y|)^{(\frac{1}{\alpha}-1)2\beta} |x-y|^{2\beta} \int_{t-|x-y|}^t (t-s)^{-2\alpha-2\beta}\dd s\\
    &\lesssim \max (|x|,|y|)^{(\frac{1}{\alpha}-1)2\beta} |x-y|^{2\beta} \bigg[ \frac{-(t-s)^{1-2\alpha-2\beta}}{1-2\alpha-2\beta} \bigg]_{t-|x-y|}^t\\
    &\lesssim \max (|x|,|y|)^{(\frac{1}{\alpha}-1)2\beta} |x-y|^{2\beta} |x-y|^{1-2\alpha-2\beta}\\
    &\lesssim \max (|x|,|y|)^{(\frac{1}{\alpha}-1)2\beta} |x-y|^{1-2\alpha}.
  \end{align*}
\end{proof}

With these auxiliary results at hand, we are ready to prove the following regularity result for solutions of the SPDE~\eqref{SPDE}.
  
\begin{proposition}\label{prop: regularity result}
  Suppose Assumption~\ref{ass:coefficients} and let $X\in C([0,T]\times\mathbb{R})$ be a strong solution of the SPDE~\eqref{SPDE}.
  \begin{itemize}
    \item[(i)] For any $p\in (0,\infty)$, one has
      \begin{equation*}
        \sup\limits_{t\in [0,T]} \sup\limits_{x\in\mathbb{R}}\mathbb{E}[ |X(t,x)|^p ]<\infty.
      \end{equation*}
    \item[(ii)] We define the random field $(Z(t,x))_{t\in[0,T],x\in\R}$ by
     \begin{align*}
       Z(t,x)&:=X(t,x)-x_0(t)\\
       &=\int_0^t p_{t-s}^\theta(x)\mu(s,X(s,0))\dd s + \int_0^t p_{t-s}^\theta(x)\sigma(s,X_s(s,0))\dd B_s.
     \end{align*}
     For any $0\leq t\leq t'\leq T$, $|x|,|y|\leq 1$ and $p\in [2,\infty )$, we get
     \begin{equation*}
       \mathbb{E}\big[| Z(t,x)-Z(t',y)|^p \big]\lesssim |t'-t|^{(\frac{1}{2}-\alpha)p}+|x-y|^{(\frac{1}{2}-\alpha)p}.
     \end{equation*}
  \end{itemize}
\end{proposition}

\begin{proof}
  (i) Let us assume that $p\geq 2$. For $p\in(0,2)$, the statement then follows by the orderedness of the $L^p$-spaces. From Lemma~\ref{lem:equuiv} we know that $Y_t:=X(t,0)$ is a solution of the SVE~\eqref{eq:SVE} and from Lemma~\ref{lem:regularity} we know that its moment are finite. Thus, applying H{\"o}lder's and the Burkholder--Davis--Gundy inequality, the linear growth condition on~$\mu$ and~$\sigma$ from Assumption~\ref{ass:coefficients}, and Lemma~\ref{lem:regularity}, we get
  \begin{align*}
    \mathbb{E}[ |X(t,x)|^p ]
    &\lesssim 1+ \mathbb{E}\bigg[ \bigg| \int_0^t p_{t-s}^\theta(x)\mu(s,Y_s)\dd s \bigg|^p \bigg]
    +\mathbb{E}\bigg[ \bigg| \int_0^t p_{t-s}^\theta(x)\sigma(s,Y_s)\dd B_s \bigg|^p \bigg]\\
    &\lesssim 1+\bigg( \int_0^t \big( p_{t-s}^\theta(x)\big)^2\dd s \bigg)^{\frac{p}{2}}
    +\bigg( \int_0^t\big( p_{t-s}^\theta(x) \big)^2 \dd s \bigg)^{\frac{p}{2}}\\
    &\lesssim 1+\bigg( \int_0^tc_\theta^2 (t-s)^{-2\alpha}e^{-2\frac{|x|^{2+\theta}}{2(t-s)}} \dd s \bigg)^{\frac{p}{2}}\\
    &\lesssim 1+\bigg( \int_0^t (t-s)^{-2\alpha}\dd s \bigg)^{\frac{p}{2}}<\infty.
  \end{align*}
  
  (ii) With 
  \begin{equation*}
    Z(t,x)=\int_0^t p_{t-s}^\theta(x)\mu(s,X(s,0))\dd s + \int_0^t p_{t-s}^\theta(x)\sigma(s,X_s(s,0))\dd B_s
  \end{equation*}
  and by splitting the integrals, we get
  \begin{align*}
    &|Z(t',x)-Z(t,y)|\\
    &\quad=\int_0^t \big( p_{t'-s}^\theta (x) - p_{t-s}^\theta (x) \big)\mu(s,X(s,0))\dd s + \int_0^t \big( p_{t-s}^\theta (x) - p_{t-s}^\theta (y) \big)\mu(s,X(s,0))\dd s\\
    &\qquad + \int_t^{t'}  p_{t'-s}^\theta (x) \mu(s,X(s,0))\dd s\\
    &\qquad+\int_0^t \big( p_{t'-s}^\theta (x) - p_{t-s}^\theta (x) \big)\sigma(s,X(s,0))\dd B_s + \int_0^t \big( p_{t-s}^\theta (x) - p_{t-s}^\theta (y) \big)\sigma(s,X(s,0))\dd B_s\\
    &\qquad + \int_t^{t'} p_{t'-s}^\theta (x) \sigma(s,X(s,0))\dd B_s\\
    &\quad=:D_1+D_2+D_3+S_1+S_2+S_3.
  \end{align*}
  We use Lemma~\ref{lem:6.2}, Lemma~\ref{lem:6.3}, H{\"o}lder's and the Burkholder--Davis--Gundy inequality, Fubini's theorem as well as (i) to get the following estimates:
  \begin{align*}
    &\mathbb{E}[|D_1|^p]\leq \bigg(\int_0^t \big( p_{t'-s}^\theta (x) - p_{t-s}^\theta (x) \big)^2 \dd s\bigg)^{\frac{p}{2}}\lesssim |t'-t|^{p(\frac{1}{2}-\alpha)},\\
    &\mathbb{E}[|S_1|^p]\leq \bigg(\int_0^t  \big( p_{t'-s}^\theta (x) - p_{t-s}^\theta (x) \big)^2 \dd s\bigg)^{\frac{p}{2}}\lesssim |t'-t|^{p(\frac{1}{2}-\alpha)},\\
    &\mathbb{E}[|D_2|^p]\leq \bigg(\int_0^t \big( p_{t-s}^\theta (x) - p_{t-s}^\theta (y) \big)^2 \dd s\bigg)^{\frac{p}{2}}\lesssim |x-y|^{p(\frac{1}{2}-\alpha)},\\
    &\mathbb{E}[|S_2|^p]\lesssim |x-y|^{p(\frac{1}{2}-\alpha)},\\
    &\mathbb{E}[|D_3|^p]\leq \bigg(\int_t^{t'} p_{t'-s}^\theta (x)^2 \dd s\bigg)^{\frac{p}{2}}\lesssim \bigg(\int_t^{t'} (t'-s)^{-2\alpha} \dd s\bigg)^{\frac{p}{2}}\lesssim |t'-t|^{p(\frac{1}{2}-\alpha)}, \\
    &\mathbb{E}[|S_3|^p]\lesssim |t'-t|^{p(\frac{1}{2}-\alpha)}.
  \end{align*}
  Hence, we obtain the desired statement.
\end{proof}

\subsection{Transformation to an SPDE in distributional form}

The next aim is to transform the SPDE~\eqref{SPDE} into an SPDE in distributional form. To that end, we consider the evolution equation \eqref{eq: evol_eq} on the whole $[0,T]\times\R$, i.e.
\begin{align}\label{eq: evol_eq 2}
  \begin{split}
  \frac{\partial u}{\partial t}(t,x)&=\Delta_\theta u(t,x),\quad t\in [0,T],\,x\in\R,\\
  u(0,x)&=\delta_0(x).
  \end{split}
\end{align}

We are interested in the fundamental solution $p^\theta \colon[0,T]\times\R\times\R\to\R$ of \eqref{eq: evol_eq 2}, in the sense that for any $g\colon\R\to\R$, $\big(\int_\R p_t^\theta(x,y)g(y)\dd y\big)_{t\in[0,T],x\in\R}$ is a solution of \eqref{eq: evol_eq 2} with initial condition $g$ instead of~$\delta_0$.

The semigroup $(S_t)_{t\in[0,T]}$ generated by $\Delta_\theta$ is then defined by $S_t\colon C_0^\infty(\R)\to C_0^\infty(\R)$ via
\begin{equation}\label{eq: def_semigroup}
  S_t\phi(x):=\int_{\mathbb{R}}p_t^\theta(x,y)\phi(y)\dd y, \quad \phi \in C_0^\infty(\R).
\end{equation}

First, we go back to the system \eqref{eq: evol_eq} where only $x\in\R_+$ is allowed and denote its fundamental solutions by
\begin{equation}\label{def:p_betrag}
  p^{|\cdot|} \colon[0,T]\times\R\times\R\to\R
\end{equation}
and skip the $\theta$-dependence for the sake of a better readability.

\medskip

To find explicit formulas for the $p^{|\cdot|}$, we need the following preparations:
\begin{itemize}
  \item A squared Bessel process $Z_t\geq 0$ of dimension $n\in\mathbb{R}$ is given by the stochastic differential equation
    \begin{equation*}
      \d Z_t=2\sqrt{Z_t}\dd B_t + n\dd t,\quad t\in[0,T].
    \end{equation*}
  \item The generator of a squared Bessel process of dimension $n$ is given by 
    \begin{equation}\label{generator}
      (Lf)(x)=n\frac{\partial}{\partial x}f(x) + 2x \frac{\partial^2}{\partial x^2}f(x),\quad x\in\R_+,
    \end{equation}
    for $f\in C_0^\infty(\R_+)$, see \cite[page~443]{Revuz1999}.
  \item The semigroup $(S_t)_{t\in [0,T]}$, defined in \eqref{eq: def_semigroup}, fulfills
    \begin{equation}\label{property}
      \frac{\partial}{\partial t}(S_tf)=\Delta_\theta (S_tf) 
    \end{equation}
    for all $f\in C_0^\infty(\R_+)$, since $p^\theta$ is the fundamental solution of \eqref{eq: evol_eq 2}. Analogue, the semigroup $(S^{|\cdot|}_t)_{t\in [0,T]}$ which we define as \eqref{eq: def_semigroup} but with $p^{|\cdot|}$ instead of $p$, fulfills
    \begin{equation*}
      \frac{\partial}{\partial t}(S^{|\cdot|}_tf)=\Delta_\theta (S^{|\cdot|}_tf) 
    \end{equation*}
    for all $f\in C_0^\infty(\R_+)$.
  \item Denote by $(\xi_t)_{t\in [0,T]}$ the Markov process generated by the semigroup~$\big(S_t^{|\cdot|}\big)_{t\in[0,T]}$, that is, it has the transition densities~$(p_t^{|\cdot|})_{t\in [0,T]}$. We define the semigroup $(T_t)_{t\in [0,T]}$ by
    \begin{equation*} 
      (T_tg)(x):= (S_t(g\circ \tilde{f}))(x)
      =\mathbb{E}_{x}[ g( \tilde{f}(\xi_t ) )]
    \end{equation*}
    for the fixed function $\tilde{f}(x):=x^{2+\theta}$ and for $g\in C_0^\infty(\R_+)$.
\end{itemize}

Our ultimate aim is to find bounds on the densities $p^\theta$. Therefore, we will use that we can find explicit formulas for the densities $p^{|\cdot|}$, and then bound
\begin{equation}\label{bound_now_p}
  p_t^{\theta}(x,y)\leq p_t^{\theta}(x,y) + p_t^{\theta}(x,-y) = p^{|\cdot|}_t(|x|,|y|),\quad \forall x,y\in \R.
\end{equation}

We derive the following Bessel property for the process $(\xi_t^{2+\theta})_{t\in [0,T]}$.

\begin{lemma}
  The process $(\xi_t^{2+\theta})_{t\in [0,T]}$ is a squared Bessel process of dimension~ $\frac{2}{2+\theta}<1$.
\end{lemma}

\begin{proof}
  We show that the generator of $\tilde{f}(\xi_t)$ is the same as the one of the squared Bessel process in \eqref{generator} with dimension $\frac{2}{2+\theta}$. Therefore, we use the semigroup $T_t$ and denote by $G$ its generator. For appropriate functions $g$ we get, by the definition of the generator and by~\eqref{property},
  \begin{align*}
    (Gg)(x)
    &=\frac{\partial}{\partial t}(T_tg)|_{t\to 0}(x)
    =\frac{\partial}{\partial t}( S_t(g\circ \tilde{f}))|_{t\to 0}(x)
    =\Delta_\theta S_0(g\circ \tilde{f})(x)\\
    &=\Delta_\theta (g\circ \tilde{f})(x).
  \end{align*}
  Note that the set $\lbrace t\in[0,T]\colon \xi_t=0 \rbrace$ has Lebesque measure zero. Therefore, we can explicitly calculate, for $x> 0$,
  \begin{align*}
    (Gg)(x)
    &=\frac{2}{(2+\theta)^2}\frac{\partial}{\partial x} \bigg( x^{-\theta}\frac{\partial}{\partial x}\big( g(x^{2+\theta})\big)\bigg)\\
    &=\frac{2}{(2+\theta)^2}\frac{\partial}{\partial x}(x^{-\theta} g'(x^{2+\theta})(2+\theta)x^{1+\theta}) \\
    &=\frac{2}{(2+\theta)}\frac{\partial}{\partial x}(x g'(x^{2+\theta})) \\
    &=\frac{2}{(2+\theta)}( g'(x^{2+\theta}) + x g''(x^{2+\theta})(2+\theta)x^{1+\theta} )\\
    &=\frac{2}{(2+\theta)}\frac{\partial g}{\partial x}(x^{2+\theta}) + 2 x^{2+\theta}\frac{\partial g^2}{\partial x^2}(x^{2+\theta}) \\
    &= (Lg)(u),
  \end{align*}
  where $L$ is the generator of a squared Bessel process of dimension $\frac{2}{2+\theta}$ and $u:=x^{2+\theta}$.
\end{proof}

Next, we derive explicit formulas for the transition densities of $(\xi_t)_{t\in[0,T]}$. Note that the transition densities for the squared Bessel process of dimension $n$ are for $t>0$ and $y>0$ given by (see e.g. \cite[Corollary~XI.1.4]{Revuz1999})
\begin{align}
  &q_t^n(x,y) =\frac{1}{2t}\bigg( \frac{y}{x}\bigg)^{\frac{\nu}{2}}e^{-\frac{x+y}{2t}}I_\nu\bigg(\frac{\sqrt{xy}}{t}\bigg)\quad\text{for }x>0 \quad\text{and}\label{densit}\\
  &q_t^n(0,y)=2^{-\nu}t^{-(\nu +1)}\Gamma(\nu +1)^{-1}y^{2\nu +1}e^{-\frac{y^2}{2t}},\label{densit0}
\end{align}
where $\nu:=\frac{n}{2}-1$ denotes the index of the Bessel process and $I_\nu$ is the modified Bessel function that is given by
\begin{equation}\label{def:Besselfunction}
  I_\nu(x):=\sum\limits_{k=0}^{\infty} \frac{( x/2 )^{2k+\nu}}{k!\Gamma(\nu+k+1)}
\end{equation}
for $\nu\geq -1$ and $x>0$.

\begin{lemma}\label{lem:transDens}
  The transition densities of the Markov process $(\xi_t)_{t\in [0,T]}$ are, for $t>0$, given by
  \begin{equation}\label{transDens}
    p_t^{|\cdot|}(x,y)=\frac{(2+\theta)}{2t}|xy|^{\frac{(1+\theta)}{2}} e^{-\frac{|x|^{2+\theta}+|y|^{2+\theta}}{2t}}I_\nu\bigg(\frac{|xy|^{1+\frac{\theta}{2}}}{t}\bigg)\quad\text{for }x,y>0,
  \end{equation}
  and for $x=0$, $y> 0$ with $p_t^{|\cdot|}(0,y)=p_t^\theta(y)$ defined in \eqref{def_pt}. Consequently, \eqref{transDens} are explicit formulas for the fundamental solutions $p^{|\cdot|}$ defined in \eqref{def:p_betrag}.
\end{lemma}

\begin{proof}
  Denote for fixed $\theta>0$ by $q_t$ the density function of the Bessel process $|\xi_t|^{2+\theta}$ with dimension~$\frac{2}{2+\theta}$, that is given by \eqref{densit} with $\nu=\frac{1}{2+\theta}-1$.
  
  Now, by noting that, for all $x,t,s>0$ and Borel sets $A\subset B(\mathbb{R}_+)$,
  \begin{align*}
    \mathbb{E}\Big[\mathbbm{1}_{A}(|\xi_{t+s}|^{2+\theta})|\xi_{t+s}|^{2+\theta}\,\Big|\,|\xi_{t}|^{2+\theta}=x \Big] = \mathbb{E}\Big[\mathbbm{1}_{A}(|\xi_{t+s}|^{2+\theta})|\xi_{t+s}|^{2+\theta} \,\Big|\,|\xi_{t}|=x^{\frac{1}{2+\theta}} \Big]
  \end{align*}
  holds, we get with the notation $B:=\lbrace b\in\mathbb{R}_+:\,b^{2+\theta}\in A \rbrace$ the relation
  \begin{align}\label{compare}
    \int_{A}q_t(x,y)y\dd y
    &= \int_{B}p_t^{|\cdot|}\Big(x^{\frac{1}{2+\theta}},y\Big)y^{2+\theta}\dd y\notag\\   
    &= \frac{1}{2+\theta}\int_{A}p_t^{|\cdot|}\Big(x^{\frac{1}{2+\theta}},z^{\frac{1}{2+\theta}}\Big)z\, z^{\frac{1}{2+\theta}-1}\dd z\notag\\
    &= \frac{1}{2+\theta}\int_{A}p_t^{|\cdot|}\Big(x^{\frac{1}{2+\theta}},y^{\frac{1}{2+\theta}}\Big)y^{\frac{1}{2+\theta}-1}y \dd y,   
  \end{align}
  where we substituted $z:=y^{2+\theta}$ and thus $\d y=\frac{1}{2+\theta}z^{\frac{1}{2+\theta}-1}\dd z$. Since \eqref{compare} must hold for all Borel sets~$A$, we can compare both sides of the equation to see with the notation
  \begin{equation*}    
    \hat{x}:=x^{\frac{1}{2+\theta}}
    \quad\text{and}\quad
    \hat{y}:=y^{\frac{1}{2+\theta}}
  \end{equation*}
  that, with $\nu=\frac{1}{2+\theta}-1=-(\frac{1+\theta}{2+\theta})$,
  \begin{align*}
    p_t^{|\cdot|}(\hat{x},\hat{y})
    &=(2+\theta)q_t\Big(\hat{x}^{2+\theta},\hat{y}^{2+\theta}\Big)y^{1-\frac{1}{2+\theta}}\\
    &=\frac{(2+\theta)}{2t}\bigg| \frac{\hat{y}}{\hat{x}}\bigg|^{\frac{(2+\theta)\nu}{2}}e^{-\frac{|\hat{x}|^{2+\theta}+|\hat{y}|^{2+\theta}}{2t}}I_\nu\bigg(\frac{|\hat{x}\hat{y}|^{1+\frac{\theta}{2}}}{t}\bigg)|\hat{y}|^{1+\theta}\\
    &=\frac{(2+\theta)}{2t}\bigg|\frac{\hat{y}}{\hat{x}}\bigg|^{-\frac{(1+\theta)}{2}}e^{-\frac{|\hat{x}|^{2+\theta}+|\hat{y}|^{2+\theta}}{2t}}I_\nu\bigg(\frac{|\hat{x}\hat{y}|^{1+\frac{\theta}{2}}}{t}\bigg)|\hat{y}|^{1+\theta}\\
    &=\frac{(2+\theta)}{2t}|\hat{x}\hat{y}|^{\frac{(1+\theta)}{2}}e^{-\frac{|\hat{x}|^{2+\theta}+|\hat{y}|^{2+\theta}}{2t}}I_\nu\bigg(\frac{|\hat{x}\hat{y}|^{1+\frac{\theta}{2}}}{t}\bigg).
  \end{align*}

  By a very similar calculation, \eqref{densit0} can be used to derive \eqref{def_pt} in the case of $x=0$:
  \begin{align*}
    \int_B q_t^\theta(0,y)y\dd y
    &=\int_A q_t^\theta(0,z^{1+\theta/2})z^{\theta/2}(1+\theta/2)z^{1+\theta/2}\dd z\\
    &=(1+\theta/2)2^{\frac{1+\theta}{2+\theta}}\Gamma(\nu+1)^{-1}\int_A t^{-(\nu+1)}z^{-\theta/2}e^{-\frac{|z|^{2+\theta}}{2t}}z^{\theta/2}z^{1+\theta/2}\dd z\\
    &=(2+\theta)2^{-\frac{1}{2+\theta}}\Gamma\bigg(\frac{1}{2+\theta}\bigg)^{-1}\int_A t^{-\frac{1}{2+\theta}}e^{-\frac{|z|^{2+\theta}}{2t}}z^{1+\theta/2}\dd z\\
    &=\int_A p_t^{|\cdot|}(0,z)z^{1+\theta/2}\dd z
  \end{align*}
  with $p_t^{|\cdot|}(0,z)=p_t^\theta(z)$ as in \eqref{def_pt} and choosing $c_\theta$ as in \eqref{def:c_theta}.  
\end{proof}

\begin{corollary}\label{cor:bound_ptheta}
  The fundamental solutions $p^\theta \colon[0,T]\times\R\times\R\to\R$ of \eqref{eq: evol_eq 2} fulfill for all $t\in[0,T]$,
  \begin{align*}
    p_t^{\theta}(x,y)\leq\frac{(2+\theta)}{2t}|xy|^{\frac{(1+\theta)}{2}} e^{-\frac{|x|^{2+\theta}+|y|^{2+\theta}}{2t}}I_\nu\bigg(\frac{|xy|^{1+\frac{\theta}{2}}}{t}\bigg)\quad\text{for }x,y\neq 0,
  \end{align*}
  and
  \begin{equation*}
    p_t^{\theta}(x,0)\leq c_\theta t^{-\frac{1}{2+\theta}}e^{-\frac{|x|^{2+\theta}}{2t}}\quad \text{for }x\neq 0.
  \end{equation*}
\end{corollary}

\begin{proof}
  This is a straight consequence of \eqref{bound_now_p} and Lemma~\ref{lem:transDens}.
\end{proof}

Having the bound from Corollary~\ref{cor:bound_ptheta}, we introduce a partial integration formula for the operator~$\Delta_\theta$ using the fundamental solutions $p_t^\theta$ of~\eqref{eq: evol_eq}.

\begin{lemma}\label{cor:1}
  For $\Delta_\theta=\frac{2}{(2+\theta)^2}\frac{\partial}{\partial x}|x|^{-\theta}\frac{\partial}{\partial x}$, the partial integration formula
  \begin{equation*}
    \int_{\mathbb{R}} p_t(x,y)\Delta_\theta \phi(x)\dd x=\int_{\mathbb{R}}\big(\Delta_\theta p_t(x,y)\big) \phi(x)\dd x,\quad t\in[0,T],y\in \R,
  \end{equation*}
  holds for any $\phi\in C_0^2(\R)$.
\end{lemma}

\begin{proof}
  Denoting $\phi_{2,t}(x):=|x|^{-\theta}\frac{\partial}{\partial x}\phi(x)$, then $\phi_{2,t}$ has also compact support and we get, by the classical partial integration formula,
  \begin{align*}
    &\int_{\mathbb{R}}p_t(x,y) \frac{\partial}{\partial x}|x|^{-\theta}\frac{\partial}{\partial x}\phi(x)\dd x
    =\int_{\mathbb{R}}p_t(x,y) \frac{\partial}{\partial x}\phi_{2,t}(x)\dd x\\
    &\quad =-\int_{\mathbb{R}}\frac{\partial}{\partial x}p_t(x,y) \phi_{2,t}(x)\dd x
    =-\int_{\mathbb{R}}\bigg( \frac{\partial}{\partial x}p_t(x,y) \bigg)|x|^{-\theta}\frac{\partial}{\partial x}\phi(x)\dd x.
  \end{align*}
  Then, again by partial integration, we get, as claimed,
  \begin{equation*}
    \int_{\mathbb{R}}p_t(x,y) \frac{\partial}{\partial x}|x|^{-\theta}\frac{\partial}{\partial x}\phi(x)\dd
    x=\int_{\mathbb{R}}\frac{\partial}{\partial x}\bigg( \bigg( \frac{\partial}{\partial x}p_t(x,y) \bigg)|x|^{-\theta} \bigg)\phi(x)\dd x.
  \end{equation*}
\end{proof}

With these auxiliary results at hand, we are in a position to do the transformation into an SPDE in distributional form. We consider test functions $\Phi\in C_0^2([0,T]\times\R)$, to which we can apply the operator $\Delta_\theta$ such that
\begin{equation*}
  \Delta_\theta\Phi_t(x)=\frac{\partial}{\partial x}|x|^{-\theta}\frac{\partial}{\partial x}\Phi_t(x)
\end{equation*}
is well-defined for all $t\in[0,T]$ and $x\in\R\setminus \lbrace 0\rbrace$.

\begin{lemma}\label{lem:dSPDE}
  Every strong solution $(X(t,x))_{t\in[0,T],x\in\R}$ of \eqref{SPDE} is a strong solution to the following SPDE in distributional form
  \begin{align}
    \begin{split}\label{dSPDE}
    &\int_{\mathbb{R}}X(t,x)\Phi_t(x)\dd x\\
    &\quad=\int_{\mathbb{R}}\bigg(x_0\Phi_0(x)+\int_0^t \Phi_s(x) \frac{\partial}{\partial s}x_0(s)\dd s\bigg)\dd x \\
    &\quad\quad+\int_0^t\int_{\mathbb{R}}X(s,x)\bigg( \Delta_\theta\Phi_s(x)+\frac{\partial}{\partial s}\Phi_s(x) \bigg)\dd x \dd s \\
    &\quad\quad +\int_0^t\mu(s,X(s,0))\Phi_s(0)\dd s
    +\int_0^t\sigma(s,X(s,0))\Phi_s(0)\dd B_s, \quad t\in[0,T],
    \end{split}  
  \end{align}
  for every test function $\Phi\in C^2_0([0,T]\times \mathbb{R})$.
\end{lemma}

\begin{proof}
  Let $X$ be a solution to \eqref{SPDE} and $\Phi$ be as in the statement. We first observe that
  \begin{align}\label{I}
    &\int_0^t \langle X(s,\cdot),\Delta_\theta \Phi_s \rangle \dd s\notag\\
    &\quad=\int_0^t\int_{\mathbb{R}} x_0(s) \Delta_\theta \Phi_s(x)\dd x\dd s+\int_0^t\int_{\mathbb{R}}\int_0^s p_{s-u}^\theta (x)\sigma(u,X(u,0))\dd B_u\,\Delta_\theta \Phi_s(x)\dd x\dd s\notag\\
    &\quad \quad +\int_0^t\int_{\mathbb{R}}\int_0^sp_{s-u}^\theta (x)\mu(u,X(u,0))\dd u\,\Delta_\theta \Phi_s(x)\dd x \dd s\notag\\
    &\quad=:I_1+I_2+I_3.
  \end{align}
  Use the fact that $p^\theta_s(x,\cdot)$ is a probability density to write $x_0(s)=\int_{\R}p_s^\theta(x,y)x_0(s) \dd y$ and use Fubini's theorem, the partial integration formula from Lemma~\ref{cor:1} and the fact that $p_t^\theta$ is a fundamental solution, to get
  \begin{align*}
    I_1&=\int_0^t\int_{\mathbb{R}}\int_{\mathbb{R}}p_s^\theta(x,y)x_0(s)\dd y\,\Delta_\theta \Phi_s(x)\dd x\dd s\\
    &=\int_0^t x_0(s)\int_{\mathbb{R}}\int_{\mathbb{R}}p_s^\theta(x,y)\Delta_\theta \Phi_s(x)\dd x \dd y\dd s\\
    &=\int_{\mathbb{R}}\int_{\mathbb{R}}\int_0^t x_0(s)\Big(\Delta_\theta p_s^\theta(x,y)\Big) \Phi_s(x)\dd s\dd y\dd x\\
    &=\int_{\mathbb{R}}\int_{\mathbb{R}}\int_0^t \bigg(\frac{\partial}{\partial s} p_s^\theta(x,y)\bigg)x_0(s) \Phi_s(x)\dd s\dd y\dd x.
  \end{align*}
  We denote the summands on the right-hand side of \eqref{SPDE} as $X_i(t,x)$ for $i=2,3$, that is, $X(t,x)=x_0+X_2(t,x)+X_3(t,x)$. Due to the $s$-dependence in $x_0(s)$ and $\Phi_s$, we apply the product rule to get
  \begin{align}\label{I1}
    I_1
    &=\int_{\mathbb{R}}\int_{\mathbb{R}}\int_0^t\frac{\partial}{\partial s}\Big((x_0(s) p_s^\theta(x,y) \Phi_s(x)\Big)\dd s\dd y\dd x\notag\\
    &\quad - \int_{\mathbb{R}}\int_{\mathbb{R}}\int_0^t p_s^\theta(x,y) \frac{\partial}{\partial s}\Big(x_0(s)\Phi_s(x)\Big)\dd s\dd y\dd x\notag\\
    &=\langle x_0(t),\Phi_t\rangle - \langle x_0(0),\Phi_0\rangle \notag \\
    &\quad- \int_0^t\int_{\mathbb{R}}x_0(s) \frac{\partial}{\partial s}\Phi_s(x)\dd x\dd s - \int_0^t\int_{\mathbb{R}}\Phi_s(x) \frac{\partial}{\partial s}x_0(s)\dd x\dd s.
  \end{align}
  Similarly, using the stochastic Fubini theorem, we get
  \begin{align}\label{I2}
    I_2&=\int_0^t\int_{\mathbb{R}}\int_0^s p_{s-u}^\theta (x)\sigma(u,X(u,0))\dd B_u \, \Delta_\theta \Phi_s(x)\dd x\dd s\notag\\
    &=\int_0^t\int_{\mathbb{R}}\int_u^t \bigg(\frac{\partial}{\partial s} p_{s-u}^\theta (x)\bigg) \Phi_s(x)\dd s\dd x\, \sigma(u,X(u,0))\dd B_u\notag\\
    &=\int_0^t\int_{\mathbb{R}}\int_u^t \frac{\partial}{\partial s}\bigg( p_{s-u}^\theta (x) \Phi_s(x)\bigg)\dd s\dd x\, \sigma(u,X(u,0))\dd B_u \notag\\
    &\quad- \int_0^t\int_{\mathbb{R}}\int_u^t  p_{s-u}^\theta (x) \bigg(\frac{\partial}{\partial s}\Phi_s(x)\bigg)\dd s\dd x \, \sigma(u,X(u,0))\dd B_u\notag\\
    &=\langle X_2(t,\cdot),\Phi_t\rangle -\int_0^t\int_{\mathbb{R}}p_0^\theta(x,0)\Phi_u(x)\dd x\,\sigma(u,X(u,0))\dd B_u\notag\\
    &\quad - \int_0^t\int_{\mathbb{R}}\int_0^s p_{s-u}^\theta (x) \sigma(u,X(u,0))\dd B_u\,\bigg(\frac{\partial}{\partial s}\Phi_s(x)\bigg)\dd x\dd s\notag\\
    &=\langle X_2(t,\cdot),\Phi_t\rangle -\int_0^t\Phi_u(0)\sigma(u,X(u,0))\dd B_u\notag\\
    &\quad - \int_0^t\int_{\mathbb{R}}X_2(s,x)\bigg(\frac{\partial}{\partial s}\Phi_s(x)\bigg)\dd x\dd s
  \end{align}
  and
  \begin{align}\label{I3}
    I_3&=\int_0^t\int_{\mathbb{R}}\int_0^sp_{s-u}^\theta (x)\mu(u,X(u,0))\dd u\,\Delta_\theta \Phi_s(x)\dd x\dd s\notag\\
    &=\int_0^t\int_{\mathbb{R}}\int_u^t\frac{\partial}{\partial s}\Big(
    p_{s-u}^\theta (x) \Phi_s(x)\Big)\dd s\dd x\,\mu(u,X(u,0))\dd u\notag\\
    &\quad-\int_0^t\int_{\mathbb{R}}\int_u^t
    p_{s-u}^\theta (x)\bigg(\frac{\partial}{\partial s}\Phi_s(x)\bigg)\dd s\dd x\,\mu(u,X(u,0))\dd u\notag\\
    &=\langle X_3(t,\cdot),\Phi_t\rangle - \int_0^t \Phi_u(0)\mu(u,X(u,0))\dd u\notag\\
    &\quad-\int_0^t\int_{\mathbb{R}}X_3(s,x)\bigg(\frac{\partial}{\partial s}\Phi_s(x)\bigg)\dd x\dd s.
  \end{align}
  Plugging \eqref{I1}, \eqref{I2} and \eqref{I3} into \eqref{I} and rearranging the terms yields
  \begin{align*}
    \langle X(t,\cdot),\Phi_t\rangle
    &=\int_{\mathbb{R}}\bigg(x_0(0)\Phi_0(x)+\int_0^t \Phi_s(x) \frac{\partial}{\partial s}x_0(s)\dd s\bigg) \dd x\\
    &\quad+\int_0^t\int_{\mathbb{R}}X(s,x)\bigg(\Delta_\theta\Phi_s(x)+\frac{\partial}{\partial s}\Phi_s(x) \bigg)\dd x\dd s\\
    &\quad+\int_0^t\mu(s,X(s,0))\Phi_s(0)\dd s
    +\int_0^t\sigma(s,X(s,0))\Phi_s(0)\dd B_s,
  \end{align*}
  for $t\in [0,T]$, which shows that \eqref{dSPDE} holds.
\end{proof}

We summarize the findings of Step~1 in the following proposition.

\begin{proposition}\label{prop:step1}
  Every strong $L^p$-solution $(X_t)_{t\in[0,T]}$ to the SVE~\eqref{eq:SVE} with $p$ given by \eqref{def_p} generates a strong solution $(X_t)_{t\in[0,T],x\in \R}$, as defined in \eqref{def:randomfields}, to the distributional SPDE~\eqref{dSPDE} with $X\in C([0,T]\times\R)$ a.s. Furthermore, $\sup_{t\in[0,T],x\in\R}\E[|X(t,x)|^q]<\infty$ for all $q\in(0,\infty)$ and, for $Z(t,x):=X(t,x)-x_0(t)$ and $q\in[2,\infty)$,
  \begin{equation*}
    \mathbb{E}[| Z(t,x)-Z(t',x')|^q ]
    \lesssim |t'-t|^{(\frac{1}{2}-\alpha)q}+|x-x'|^{(\frac{1}{2}-\alpha)q},
  \end{equation*}
  for all $t,t'\in[0,T]$ and $x,x'\in[-1,1]$.
\end{proposition}

\begin{proof}
  The implication of the solution to \eqref{dSPDE} by the one to \eqref{eq:SVE} is given by Lemma~\ref{lem:equuiv} and Lemma~\ref{lem:dSPDE}, the continuity by Lemma~\ref{lem:equuiv} and the remaining properties by Proposition~\ref{prop: regularity result}.
\end{proof}

\section{Step~2 and~3: Implementing Yamada--Watanabe's approach}

The next steps are to use the classical approximation of the absolute value function introduced by Yamada--Watanabe~\cite{Yamada1971}, allowing us to apply It{\^o}'s formula. Recall that, by Assumption~\ref{ass:coefficients}~(ii), $\sigma$ is $\xi$-H{\"o}lder continuous for some $\xi \in [\frac{1}{2},1]$. Hence, there exists a strictly increasing function $\rho\colon[0,\infty)\to [0,\infty)$ such that $\rho(0)=0$,
\begin{equation*}
  |\sigma(t,x)-\sigma(t,y)|\leq C_{\sigma} |x-y|^{\xi} \leq \rho(|x-y|)
  \quad \text{for } t\in [0,T] \text{ and }x,y \in \R
\end{equation*}
and
\begin{equation*}
  \int_0^{\epsilon} \frac{1}{\rho(x)^2}\dd x=\infty\quad  \text{for all } \epsilon>0.
\end{equation*}
Based on $\rho$, we define a sequence $(\phi_n)_{n\in\N}$ of functions mapping from $\R$ to $\R$ that approximates the absolute value in the following way: Let $(a_n)_{n\in\N}$ be a strictly decreasing sequence with $a_0=1$ such that $a_n\to 0$ as $n\to \infty$ and
\begin{equation}
  \int_{a_n}^{a_{n-1}}\frac{1}{\rho(x)^2}\dd x=n.\label{def:a_n_YW}
\end{equation}
Furthermore, we define a sequence of mollifiers: let $(\psi_n)_{n\in\N}\in C_0^{\infty}(\R)$ be smooth functions with compact support such that $\textup{supp}(\psi_n)\subset (a_n,a_{n-1})$,
\begin{equation}\label{prop}
  0\leq \psi_n(x)\leq \frac{2}{n\rho(x)^2}\leq \frac{2}{nx}, \quad x\in\R,
  \quad\text{and}\quad
  \int_{a_n}^{a_{n-1}}\psi_n(x)\dd x=1.
\end{equation}
We set
\begin{equation}\label{def:phi_n}
  \phi_n(x):=\int_0^{|x|} \bigg(\int_0^y \psi_n(z)\dd z \bigg)\dd y,\quad x\in \R.
\end{equation}
By \eqref{prop} and the compact support of $\psi_n$, it follows that $\phi_n(\cdot)\to |\cdot|$ uniformly as $n\to \infty$. Since every $\psi_n$ and, thus, every $\phi_n$ is zero in a neighborhood around zero, the functions~$\phi_n$ are smooth with
\begin{equation*}
  \|\phi_n'\|_\infty\leq 1,
  \quad
  \phi_n'(x)=\textup{sgn}(x)\int_0^{|x|}\psi_n(y)\dd y
  \quad\text{and}\quad
  \phi_n''(x)=\psi_n(|x|),
  \quad\text{for } x\in\R.
\end{equation*}

Let $X^1$ and $X^2$ be two strong solutions to the SPDE~\eqref{dSPDE} for a given Brownian motion $(B_t)_{t\in[0,T]}$ such that $X^1,X^2\in C([0,T]\times \R)$ a.s. We define $\tilde{X}:=X^1-X^2$ and consider, for some $\Phi_x^m\in C^2_0(\mathbb{R})$ for fixed $x\in\mathbb{R}$ and $m\in\R_+$ (we will later define $m$ depending on~$n$ and $\Phi_x^m$ is independent of~$t$):
\begin{equation*}
  \langle \tilde{X}_t,\Phi_x^m \rangle = \int_{\mathbb{R}}\tilde{X}(t,y)\Phi_x^m(y)\dd y,
\end{equation*}
where $\langle \cdot,\cdot \rangle$ denotes the scalar product on $L^2(\R)$.

\begin{proposition}\label{prop:step2}
  For a fixed $x\in\R$ and $m\in\R_+$, let $\Phi_x^m\in C^2_0(\R)$ be such that $\Delta_\theta \Phi_x^m$ is well-defined. Then, for $t\in [0,T]$, one has
  \begin{align}\label{afterIto}
    \phi_n(\langle \tilde{X}_t,\Phi_{x}^m \rangle)
    &=  \int_0^t  \phi_n'(\langle \tilde{X}_s, \Phi_x^m \rangle)\langle \tilde{X}_s,\Delta_\theta \Phi_{x}^m \rangle \dd s\notag\\
    &\quad+\int_0^t \phi_n'(\langle \tilde{X}_s,\Phi_x^m \rangle)\Phi_x^m(0) ( \mu(s,X^1(s,0))-\mu(s,X^2(s,0)) )\dd s\notag\\
    &\quad +\int_0^t  \phi_n'(\langle \tilde{X}_s, \Phi_x^m \rangle)\Phi_x^m(0) ( \sigma(s,X^1(s,0))-\sigma(s,X^2(s,0)) )\dd B_s\notag\\
    &\quad+\frac{1}{2}\int_0^t \psi_n(| \langle \tilde{X}_s,\Phi_x^m \rangle|)\Phi_x^m(0)^2 ( \sigma(s,X^1(s,0))-\sigma(s,X^2(s,0)) )^2\dd s.
  \end{align}
\end{proposition}

\begin{proof}
  By \eqref{dSPDE}, $(\langle \tilde{X}_t,\Phi_x^m \rangle)_{t\in[0,T]}$ is a semimartingale. Therefore, we are able to apply It{\^o}'s formula to $\phi_n$, which yields the result.
\end{proof}

Note that \eqref{afterIto} defines a function in~$x$. We want to integrate this against another non-negative test function with the following properties.

\begin{assumption}\label{ass:Psi}
  Let $\Psi\in C^2([0,T]\times \mathbb{R})$ be twice continuously differentiable such that
  \begin{itemize}
    \item[(i)] $\Psi_t(0)>0$ for all $t\in [0,T]$,
    \item[(ii)] $\Gamma(t):=\lbrace x\in \mathbb{R}\,:\,\exists s\leq t \text{ s.t. } |\Psi_s(x)|>0 \rbrace \subset B(0,J(t))$ for some $0<J(t)<\infty$,
    \item[(iii)] 
      \begin{equation*}
        \sup\limits_{s\leq t}\bigg| \int_{\mathbb{R}}|x|^{-\theta}\bigg( \frac{\partial \Psi_s(x)}{\partial x} \bigg)^2\dd x  \bigg|<\infty, \quad  t\in [0,T].
      \end{equation*}
  \end{itemize}
\end{assumption}

We will later choose an explicit function $\Psi$ and show that it fulfills Assumption~\ref{ass:Psi}. Then, we get the following equality, where the extra term $I_5^{m,n}$ arises due to the $t$-dependence of~$\Psi$.

\begin{proposition}\label{prop:Psi_step3}
  For $\Psi$ fulfilling Assumption~\ref{ass:Psi}, we have
  \begin{align}\label{I1I2I5}
    &\langle \phi_n(\langle \tilde{X}_t,\Phi_{\cdot}^m \rangle),\Psi_t \rangle\notag\\
    &\quad= \int_0^t \langle \phi_n'(\langle \tilde{X}_s,\Phi_\cdot^m \rangle)\langle \tilde{X}_s,\Delta_\theta \Phi_{\cdot}^m \rangle,\Psi_s \rangle \dd s\notag\\
    &\qquad+\int_0^t \langle \phi_n'(\langle \tilde{X}_s,\Phi_\cdot^m \rangle)\Phi_\cdot^m(0),\Psi_s \rangle ( \mu(s,X^1(s,0))-\mu(s,X^2(s,0)) )\dd s\notag\\
    &\qquad+ \int_0^t \langle \phi_n'(\langle \tilde{X}_s,\Phi_\cdot^m \rangle)\Phi_\cdot^m(0),\Psi_s \rangle ( \sigma(s,X^1(s,0))-\sigma(s,X^2(s,0)) )\dd B_s\notag\\
    &\qquad+\frac{1}{2}\int_0^t \langle \psi_n(| \langle \tilde{X}_s,\Phi_\cdot^m \rangle|)\Phi_\cdot^m(0)^2,\Psi_s \rangle ( \sigma(s,X^1(s,0))-\sigma(s,X^2(s,0)))^2\dd s\notag\\
    &\qquad+\int_0^t \langle \phi_n(\langle \tilde{X}_s,\Phi_\cdot^m \rangle),\dot{\Psi}_s \rangle \dd s\notag\\
    &\quad=:I_1^{m,n}(t)+I_2^{m,n}(t)+I_3^{m,n}(t)+I_4^{m,n}(t)+I_5^{m,n}(t),
  \end{align}
  for $t\in [0,T]$, where $\dot{\Psi}_s(x):=\frac{\partial}{\partial s}\Psi_s(x)$.
\end{proposition}

\begin{proof}
  We discretize $\Psi_t(x)$ in its time variable, then let the grid size go to zero and show that the resulting term converges to \eqref{I1I2I5}. Therefore, let $t_i=i2^{-k}$, $i=0,1,\dots,\lfloor t2^k\rfloor+1=:K_t^k$, where $\lfloor\cdot\rfloor$ denotes rounding down to the next integer, such that $t_{\lfloor t2^k\rfloor}\leq t< t_{K_t^k}$, and denote
  \begin{equation}\label{def_psii}
    \Psi_{t}^k(x):=2^k\int_{t_{i-1}}^{t_i}\Psi_s(x)\dd s,\quad t\in [t_{i-1},t_i), x\in\R.
  \end{equation}
  Then, we can build the telescope sum
  \begin{align}\label{eq331}
    \langle \phi_n(\langle \tilde{X}_t,\Phi_{\cdot}^m \rangle),\Psi_t \rangle
    &=\sum\limits_{i=1}^{K_t^k}\langle \phi_n(\langle \tilde{X}_{t_i},\Phi_{\cdot}^m \rangle),\Psi_{t_i}^k \rangle - \langle \phi_n(\langle \tilde{X}_{t_{i-1}},\Phi_{\cdot}^m \rangle),\Psi_{t_{i-1}}^k \rangle\notag\\
    &\quad-\langle \phi_n(\langle \tilde{X}_{t_{K_t^k}},\Phi_{\cdot}^m \rangle),\Psi_{t_{K_t^k}}^k \rangle + \langle \phi_n(\langle \tilde{X}_{t},\Phi_{\cdot}^m \rangle),\Psi_{t} \rangle.
  \end{align}
  By the continuity of $\tilde{X}$, $\Psi$ and $\phi_n$, the sum of the last two terms approaches zero as $t_{K_t^k}\to t$ and thus as $k\to\infty$.
  
  For the terms in the summation, we use the continuity of $\tilde{X}$ and the notation $f(t_{i}-):=\lim\limits_{s<t_{i},s\to t_{i}}f(s)$, to get the equality
  \begin{align*}
    &\langle \phi_n(\langle \tilde{X}_{t_{i}},\Phi_{\cdot}^m \rangle),\Psi_{t_{i}}^k \rangle
    =\langle \phi_n(\langle \tilde{X}_{t_{i}-},\Phi_{\cdot}^m \rangle),\Psi_{t_{i}-}^k \rangle + \langle \phi_n(\langle \tilde{X}_{t_{i}},\Phi_{\cdot}^m \rangle),\Psi_{t_{i}}^k-\Psi_{t_{i-1}}^k \rangle.
  \end{align*}
  By plugging this into \eqref{eq331}, we get
  \begin{align*}  
    \langle \phi_n(\langle \tilde{X}_t,\Phi_{\cdot}^m \rangle),\Psi_t \rangle &=\sum\limits_{i=1}^{K_t^k}\langle \phi_n(\langle \tilde{X}_{t_{i}-},\Phi_{\cdot}^m \rangle),\Psi_{t_{i}-}^k \rangle - \langle \phi_n(\langle \tilde{X}_{t_{i-1}},\Phi_{\cdot}^m \rangle),\Psi_{t_{i-1}}^k \rangle\notag\\
    &\quad +\sum\limits_{i=1}^{K_t^k}\langle \phi_n(\langle \tilde{X}_{t_{i}},\Phi_{\cdot}^m \rangle),\Psi_{t_{i}}^k-\Psi_{t_{i-1}}^k \rangle=:A_t^k+C_t^k.
  \end{align*}
  For $A_t^k$, we get, by applying It{\^o}'s formula, that
  \begin{align*}
    A_t^k&=\sum\limits_{i=1}^{K_t}\langle \phi_n(\langle \tilde{X}_{t_{i}},\Phi_{\cdot}^m \rangle),\Psi_{t_{i-1}}^k \rangle - \langle \phi_n(\langle \tilde{X}_{t_{i-1}},\Phi_{\cdot}^m \rangle),\Psi_{t_{i-1}}^k \rangle\\
    &\to I_1^{m,n}(t)+I_2^{m,n}(t)+I_3^{m,n}(t)+I_4^{m,n}(t)
    \quad\text{as }k\to\infty,
  \end{align*}
  by the continuity of $\Psi$.

  Thus, it remains to show that $C_t^k$ converges to $I_5^{m,n}(t)$. To that end, we use the construction~\eqref{def_psii} and Fubini's theorem to conclude that
  \begin{align*}
    C_t^k
    &=\sum\limits_{i=1}^{K_t^k}\bigg\langle \phi_n(\langle \tilde{X}_{t_{i}},\Phi_{\cdot}^m \rangle),2^k\int_{t_{i-1}}^{t_i} (\Psi_{s}-\Psi_{s-2^{-k}})\dd s \bigg\rangle\\
    &=\sum\limits_{i=1}^{K_t^k}\bigg\langle \phi_n(\langle \tilde{X}_{t_{i}},\Phi_{\cdot}^m \rangle),2^k\int_{t_{i-1}}^{t_i} \int_{s-2^{-k}}^s \dot{\Psi}_r\dd r \dd s \bigg\rangle\\
    &=2^k\sum\limits_{i=1}^{K_t^k}\int_{t_{i-1}}^{t_i} \int_{s-2^{-k}}^s\langle \phi_n(\langle \tilde{X}_{t_{i}},\Phi_{\cdot}^m \rangle), \dot{\Psi}_r \rangle \dd r \dd s\\
    &=2^k\sum\limits_{i=1}^{K_t^k}\int_{t_{i-1}}^{t_i} \int_{s-2^{-k}}^s\langle \phi_n(\langle \tilde{X}_{t_{i}},\Phi_{\cdot}^m \rangle), \dot{\Psi}_r \rangle - \langle \phi_n(\langle \tilde{X}_{r},\Phi_{\cdot}^m \rangle), \dot{\Psi}_r \rangle \dd r \dd s\\
    &\quad + 2^k\sum\limits_{i=1}^{K_t^k}\int_{t_{i-1}}^{t_i} \int_{s-2^{-k}}^s \langle \phi_n(\langle \tilde{X}_{r},\Phi_{\cdot}^m \rangle), \dot{\Psi}_r \rangle \dd r \dd s.
  \end{align*}
  The first summand can be bounded by 
  \begin{align*}
    \int_0^t \sup\limits_{u\leq t,|u-r|\leq 2^{-k}}\big| \langle \phi_n(\langle \tilde{X}_{u},\Phi_{\cdot}^m \rangle), \dot{\Psi}_r \rangle - \langle \phi_n(\langle \tilde{X}_{r},\Phi_{\cdot}^m \rangle), \dot{\Psi}_r \rangle \big|\dd r,
  \end{align*}
  which converges to zero a.s. as $k\to\infty$ by the continuity and boundedness of $\tilde{X}$. Furthermore, we get, by
  \begin{equation*}
    2^k\int_{s-2^{-k}}^s \langle \phi_n(\langle \tilde{X}_{r},\Phi_{\cdot}^m \rangle), \dot{\Psi}_r \rangle \dd r\to \langle \phi_n(\langle \tilde{X}_{s},\Phi_{\cdot}^m \rangle), \dot{\Psi}_s \rangle
    \quad \text{as } k\to\infty
  \end{equation*}
  and the dominated convergence theorem, that
  \begin{equation*}
    C_t^k \to \int_{0}^{t} \langle \phi_n(\langle \tilde{X}_{s},\Phi_{\cdot}^m \rangle), \dot{\Psi}_s \rangle \dd s
    \quad \text{as } k\to\infty,
  \end{equation*}
  which proves the proposition.
\end{proof}

We will bound the expectation of the terms $I_1^{m,n}$ to $I_5^{m,n}$ as $m,n\to\infty$ in Section~\ref{sec:step4}.

\section{Step~4: Passing to the limit}\label{sec:step4}

Before we can pass to the limit in \eqref{I1I2I5}, we need to choose a sequence $(\Phi_x^{m,n})_{n\in\N}$ of smooth functions $\Phi_x^{m,n}\in C_0^\infty(\R)$ for some $x\in\R$ and for $m\in\R_+$, which approximates the Dirac distribution~$\delta_x$ explicitly. We will choose some $m=m^{(n)}$ dependent on the index~$n$ of the Yamada--Watanabe approximation and, for notational simplicity, will skip the $m$-dependence and shortly write~$(\Phi_x^n)_{n\in\N}$.

\subsection{Explicit choice of the test function}

We want to approximate with $\Phi_x^{n}$ a Dirac distribution centered around $x\in\R$. Therefore, we choose it to coincide with the sum of two Gaussian kernels with mean $x$ and $y$, respectively, and standard deviation $m^{{-1}}$, when $x$ and $y$ are close. The reason for this construction is that we want to keep the mass of $\Phi$ in $B(0,\frac{1}{m^{(n)}})$ constant as $n\to\infty$. For this purpose, we define
\begin{equation*}
  \tilde{\Phi}_x^{m}(y):=\frac{1}{\sqrt{2\pi m^{-2}}}e^{-\frac{(y-x)^2}{2m^{-2}}}
\end{equation*}
and, to construct the compact support, let $\tilde{\psi}^{m,n}_x$ be smooth functions for $n\in \N$ and fixed $x\in\mathbb{R}$ with
\begin{align*}
  \tilde{\psi}^{m,n}_x(y):= \left\{
  \begin{array}{ll}
    1, & \text{if }y\in B(x,\frac{1}{m}) \\
    0, &  \textrm{if }y\in \mathbb{R}\setminus B(x,\frac{1}{m}+b_n) \\
  \end{array}
  \right.
\end{align*}
and $0\leq \tilde{\psi}^{m,n}_x(y)\leq 1$ for $y$ elsewhere such that $\tilde{\psi}^{m,n}_x$ is smooth. Here, let $(b_n)_{n\in\mathbb{N}}$ be a sequence such that $b_n>0$ and
\begin{equation*}
  \mu_n\bigg(B\bigg(x,\frac{1}{m}+b_n\bigg)\setminus B\bigg(x,\frac{1}{m}\bigg)\bigg)=\frac{a_n}{2},
\end{equation*}
where $\mu_n(A):=\int_A \tilde{\Phi}_x^{m}(y)\dd y$ denotes the measure in terms of the above normal distribution and $a_n:=e^{-\frac{n(n+1)}{2}}$ comes from the Yamada--Watanabe sequence. It is always possible to find such a $b_n>0$ since the mass of $\tilde{\Phi}_x^{m}$ in $B(x,\frac{1}{m})$ is $\approx 0.6827$, which is independent of~$n$, and $\frac{a_n}{2}<0.3$ for all~$n\in \N$.

Then, we define
\begin{equation}\label{phi}
  \Phi_x^{n}(y):=c\Big(\tilde{\psi}_x^{m,n}(y)\tilde{\Phi}_x^{m}(y)+\tilde{\Phi}_y^{m}(x)\tilde{\psi}_y^{m,n}(x)\Big),
\end{equation}
with $c:=1/(2m_\sigma)$, where $m_\sigma\approx 0.6827$ denotes the mass of a normal distribution $\mathcal{N}(\mu,\sigma^2)$ inside the interval $[\mu-\sigma,\mu+\sigma]$. With that choice of $c$, $\Phi_x^n$ approximates the Dirac distribution~$\delta_x$ around~$x$ as $n\to\infty$. Note that $\Phi_x^{n}(y)$ is identical in terms of $x$ and $y$. Furthermore, $\Phi_x^{n}$ owes the following properties that we will need later. To that end, let us introduce the following stopping time for $K>0$:
\begin{equation}\label{def_stopptimeK}
  T_K:=\inf\limits_{t\in [0,T]} \bigg\{ \sup\limits_{x\in [-\frac{1}{2},\frac{1}{2}]}( |X^1(t,x)|+|X^2(t,x)|)>K \bigg\},
\end{equation}
where we use the convention $\inf\emptyset:=\infty$. Note that, by the continuity of $X^1$ and $X^2$, $T_K\to \infty$ a.s. as $K\to\infty$.

\begin{proposition}\label{prop:my_prop6}
  For fixed $x\in\R$, $\Phi_x^{n}$, as defined in \eqref{phi}, fulfills:
  \begin{itemize}
    \item[(i)] $\Delta_{\theta,x}\Phi_x^{n}(y)=\Delta_{\theta,y}\Phi_x^{n}(y)$ for all $x,y\in\R$, where $\Delta_{\theta,x}$ denotes $\Delta_\theta$ acting on~$x$;
    \item[(ii)] $\int_{\mathbb{R}}\Phi_x^{n}(0)^2\dd x \lesssim m^{(n)}$ for all $n\in\N$;
    \item[(iii)] $\int_{\mathbb{R}}\Phi_x^{n}(0)\dd x\leq 2$ for all $n\in\N$;
    \item[(iv)] for all $(s,x)\in[0,T]\times\R$,
      \begin{equation*}
        \langle\tilde{X}_s,\Phi_x^{n} \rangle \to \tilde{X}(s,x) \quad\text{and}\quad
        \phi_n'(\langle \tilde{X}_s,\Phi_x^{n} \rangle )\langle \tilde{X}_s,\Phi_x^{n} \rangle \to |\tilde{X}(s,x)|,\quad \text{as } n\to\infty;
      \end{equation*}
    \item[(v)] given $s\in [0,T_K]$, there exists a constant $C_K>0$ that is independent from $n$, such that, if
      \begin{equation*}
        \bigg| \int_{\mathbb{R}} \tilde{X}(s,y)\Phi_x^{n}(y)\dd y\bigg | \leq a_{n-1}
      \end{equation*}
      holds, then there is some $\hat{x}\in B(x,\frac{1}{m})$ such that $|\tilde{X}(s,\hat{x})|\leq C_Ka_{n-1}$.
  \end{itemize}
\end{proposition}

\begin{proof}
  (i) This statement is clear since $\Phi_x^{n}$ is identical in $x$ and $y$.

  (ii) We denote $c:=\frac{1}{\sqrt{2\pi}}$ to get
  \begin{equation*}
    \int_{\mathbb{R}}\Phi_x^{n}(0)^2\dd x
    \leq \int_{\mathbb{R}}\bigg( cm e^{-\frac{|x|^{2}}{2m^{-2}}}\bigg)^2\dd x
    \leq c m \int_{\mathbb{R}} c m e^{-\frac{|x|^{2}}{2m^{-2}}} \dd x= c m.
  \end{equation*}

  (iii) $\int_{\mathbb{R}}\Phi_x^{n}(0)\dd x \leq 2\int_{\mathbb{R}}\tilde{\Phi}_x^{m}(0)\dd x=2$.

  (iv) From the construction of $\Phi_x^{n}$ we get that
  \begin{equation*}
    \int_{\mathbb{R}}\tilde{X}(s,y)\Phi_x^{n}(y)\dd y \to \int_{\mathbb{R}}\tilde{X}(s,y)\delta_x(y)\dd y =\tilde{X}(s,x)\quad \text{as } n\to\infty.
  \end{equation*}
  Furthermore, we know that $\phi_n'(x)x \to |x|$ as $n\to\infty$ uniformly in $x\in\R$ and thus the second statement follows.

  (v) Let us write
  \begin{equation}\label{star_e}
    \int_{\mathbb{R}} \tilde{X}(s,y)\Phi_x^{n}(y)\dd y
    =\int_{B(x,\frac{1}{m})} \tilde{X}(s,y)\Phi_x^{n}(y)\dd y +\int_{\mathbb{R}\setminus B(x,\frac{1}{m})} \tilde{X}(s,y)\Phi_x^{n}(y)\dd y.
  \end{equation}
  By the construction of $\tilde{\psi}_x^{m,n}$ we know that $\Phi_x^{n}$ vanishes outside the ball $B(x,\frac{1}{m}+b_n)$, and, by the choice of~$b_n$, we know that the mass of $\Phi_x^{n}$ in $B(x,\frac{1}{m}+b_n)\setminus B(x,\frac{1}{m})$ is $a_{n-1}/2$. Since we have that $s\leq T_K$, we can bound
  \begin{equation*}
    \bigg| \int_{\mathbb{R}\setminus B(x,\frac{1}{m})} \tilde{X}(s,y)\Phi_x^{n}(y)\dd y\bigg|
    \leq 2K \int_{\mathbb{R}\setminus B(x,\frac{1}{m})} \Phi_x^{n}(y)\dd y
    \leq Ka_{n-1}.
  \end{equation*}
  Thus, by assumption and \eqref{star_e}, we have that
  \begin{equation*}
    \bigg| \int_{B(x,\frac{1}{m})} \tilde{X}(s,y)\Phi_x^{n}(y)\dd y \bigg| \leq (K+1)a_{n-1},
  \end{equation*}
  and, since $\Phi_x^{n}$ is the sum of two Gaussian densities with standard deviation $\frac{1}{m}$, we know that its mass inside the ball is $\approx 2\cdot 0.6827$ and can conclude, using the continuity of~$\tilde{X}$, that
  \begin{equation*}
    (K+1)a_{n-1}
    \geq \int_{B(x,\frac{1}{m})} \Phi_x^{n}(y)\dd y \inf\limits_{y\in B(x,\frac{1}{m})}|\tilde{X}(s,y)|
    \geq 1.3 \inf\limits_{y\in B(x,\frac{1}{m})}|\tilde{X}(s,y)|,
  \end{equation*}
  and thus, the statement holds with $C_K=(K+1)/1.3$.
\end{proof}

\subsection{Bounding the Yamada--Watanabe terms}

We start with the summands $I_1^{m,n}$, $I_2^{m,n}$, $I_3^{m,n}$ and $I_5^{m,n}$ in \eqref{I1I2I5} and will analyze $I_4^{m,n}$ later. To that end, we need the following elementary estimate.

\begin{lemma}\label{lem:calculus}
  If $f\in C_0^2(\mathbb{R})$ is non-negative and not identically zero, then
  \begin{equation*}
    \sup\limits_{x\in\mathbb{R}\,:\,f(x)>0}\lbrace ( f'(x))^2f(x)^{-1} \rbrace\leq 2\|f''(x)\|_\infty.
  \end{equation*}
\end{lemma}

\begin{proof}
  Choose some $x\in\mathbb{R}$ with $f(x)>0$ and assume w.l.o.g. that $f'(x)>0$. Let
  \begin{align*}
    x_1:=\sup\lbrace x'<x\colon f'(x')=0 \rbrace,
  \end{align*}
  which exists due to the compact support of $f$. By the extended mean value theorem (see \cite[Theorem~4.6]{Apostol1967}), applied to $f$ and $(f')^2$, there exists an $x_2\in (x_1,x)$ such that
  \begin{equation*}
    ( f'(x)^2-f'(x_1)^2)f'(x_2)=( f(x)-f(x_1))\frac{\partial (f')^2}{\partial x}(x_2).
  \end{equation*}
  By the choice of $x_1$, we know that $f'(x_2)>0$, and thus with $f'(x_1)=0$,
  \begin{equation*}
    f'(x)^2=( f(x)-f(x_1))2f''(x_2).
  \end{equation*}
  Since $f$ is strictly increasing on $(x_1,x)$ and non-negative, we conclude
  \begin{equation*}
    \frac{f'(x)^2}{f(x)}\leq \frac{f'(x)^2}{f(x)-f(x_1)}= 2f''(x_2)\leq 2\|f''\|_\infty. 
  \end{equation*}
\end{proof}

We want to take expectations on both sides of \eqref{I1I2I5} and then send $m,n\to\infty$.

\begin{lemma}\label{lem:i1i2i3i5i}
  For any stopping time $\mathcal{T}$ and fixed $t\in [0,T]$ we have:
  \begin{itemize}
    \item[(i)] $\lim\limits_{m,n\to\infty} \mathbb{E}[ I_1^{m,n}(t\wedge \mathcal{T}) ]\leq \mathbb{E}\big[ \int_0^{t\wedge \mathcal{T}}\int_{\mathbb{R}}|\tilde{X}(s,x)|\Delta_\theta \Psi_s(x)\dd x\dd s \big]$;
    \item[(ii)]
    $ \lim_{m,n\to\infty} \mathbb{E}[ I_2^{m,n}(t\wedge \mathcal{T}) ]\lesssim \int_0^{t\wedge \mathcal{T}}\Psi_s(0)\,\mathbb{E}[|\tilde{X}(s,0)|]\dd s$;
    \item[(iii)]
    $\mathbb{E}[ I_3^{m,n}(t\wedge \mathcal{T})]=0$
    for all $m,n\in\N$;
    \item[(iv)] $\lim\limits_{m,n\to\infty} \mathbb{E}[ I_5^{m,n}(t\wedge \mathcal{T})]= \mathbb{E}\big[ \int_0^{t\wedge \mathcal{T}}\int_{\mathbb{R}}|\tilde{X}(s,x)|\dot{\Psi}_s(x)\dd x\dd s \big]$.
  \end{itemize}
\end{lemma}

\begin{proof}
  (i) We need to rewrite $I_1^{m,n}$. We use the property of $\Phi_x^n$ from Proposition~\ref{prop:my_prop6}~(i) and the product rule to get
  \begin{align*}
    I_1^{m,n}(t)
    &=\int_0^t \int_{\mathbb{R}} \phi_n'(\langle \tilde{X}_s,\Phi_x^n \rangle)\int_{\mathbb{R}} \tilde{X}(s,y)\Delta_{y,\theta} \Phi_{x}^n(y)\dd y\, \Psi_s(x) \dd x \dd s\\
    &=\int_0^t \int_{\mathbb{R}} \phi_n'(\langle \tilde{X}_s,\Phi_x^n \rangle)\Delta_{x,\theta}(\langle \tilde{X}_s, \Phi_{x}^n\rangle) \Psi_s(x) \dd x \dd s\\
    &=2\alpha^2\int_0^t \int_{\mathbb{R}} \phi_n'(\langle \tilde{X}_s,\Phi_x^n \rangle) \Big(\frac{\partial}{\partial x}|x|^{-\theta}\frac{\partial}{\partial x}\langle \tilde{X}_s, \Phi_{x}^n\rangle\Big) \Psi_s(x) \dd x \dd s\\
    &\quad +2\alpha^2\int_0^t \int_{\mathbb{R}} \phi_n'(\langle \tilde{X}_s,\Phi_x^n \rangle) |x|^{-\theta}\Big(\frac{\partial^2}{\partial x^2}\langle \tilde{X}_s, \Phi_{x}^n\rangle\Big) \Psi_s(x) \dd x \dd s.
  \end{align*}
  Now, we use integration by parts for both summands and the compact support of $\Psi_s$ for every $s\in[0,T]$ to get
  \begin{align}\label{part1}
    I_1^{m,n}(t)
    &=-2\alpha^2\int_0^t \int_{\mathbb{R}}\psi_n(\langle \tilde{X}_s,\Phi_x^n \rangle) |x|^{-\theta}\bigg(\frac{\partial}{\partial x}\langle \tilde{X}_s \Phi_{x}^n\rangle\bigg)^2 \Psi_s(x) \dd x \dd s\notag\\
    &\quad-2\alpha^2\int_0^t \int_{\mathbb{R}} \phi_n'(\langle \tilde{X}_s,\Phi_x^n \rangle) |x|^{-\theta}\frac{\partial}{\partial x}\langle \tilde{X}_s \Phi_{x}^n\rangle \frac{\partial}{\partial x} \Psi_s(x) \dd x \dd s.
  \end{align}
  By a very similar partial integration we see that
  \begin{align}\label{part2}
    &\int_0^t \int_{\mathbb{R}} \phi_n'(\langle \tilde{X}_s,\Phi_x^n \rangle)\langle \tilde{X}_s,\Phi_x^n \rangle \Delta_\theta \Psi_s(x)\dd x\dd s\notag\\
    &\quad=-2\alpha^2 \int_0^t \int_{\mathbb{R}} \psi_n(\langle \tilde{X}_s,\Phi_x^n \rangle) \frac{\partial}{\partial x}\langle \tilde{X}_s,\Phi_x^n \rangle  \langle \tilde{X}_s,\Phi_x^n \rangle |x|^{-\theta}\frac{\partial}{\partial x} \Psi_s(x)\dd x\dd s\notag\\
    &\quad\quad-2\alpha^2 \int_0^t \int_{\mathbb{R}} \phi_n'(\langle \tilde{X}_s,\Phi_x^n \rangle)\frac{\partial}{\partial x}\langle \tilde{X}_s,\Phi_x^n \rangle |x|^{-\theta}\frac{\partial}{\partial x} \Psi_s(x)\dd x\dd s.
  \end{align}
  By identifying that the second term in \eqref{part1} coincides with the second term in \eqref{part2}, we can plug in the latter one into the first one to get
  \begin{align}\label{I1I2I3mn}
     I_1^{m,n}(t)&=-2\alpha^2\int_0^t \int_{\mathbb{R}}\psi_n(\langle \tilde{X}_s,\Phi_x^n \rangle) |x|^{-\theta}\bigg(\frac{\partial}{\partial x}\langle \tilde{X}_s \Phi_{x}^n\rangle\bigg)^2 \Psi_s(x) \dd x  \dd s\notag\\
     &\quad+2\alpha^2\int_0^t \int_{\mathbb{R}} \psi_n(\langle \tilde{X}_s,\Phi_x^n \rangle) \frac{\partial}{\partial x}\langle \tilde{X}_s,\Phi_x^n \rangle  \langle \tilde{X}_s,\Phi_x^n \rangle |x|^{-\theta}\frac{\partial}{\partial x} \Psi_s(x)\dd x\dd s\notag\\    
     &\quad+\int_0^t \int_{\mathbb{R}} \phi_n'(\langle \tilde{X}_s,\Phi_x^n \rangle)\langle \tilde{X}_s,\Phi_x^n \rangle \Delta_\theta \Psi_s(x)\dd x\dd s\notag\\   
     &=\int_0^t \big(I_{1,1}^{m,n}(s) + I_{1,2}^{m,n}(s) + I_{1,3}^{m,n}(s)\big)\dd s.
  \end{align}
  In order to deal with the various parts of $I_1^{m,n}$, we start with treating $I_{1,1}^{m,n}$ and $I_{1,2}^{m,n}$. Since we want to show that these parts are less than or equal to $0$, we define for fixed $s\in [0,t]$:
  \begin{align*}
    A^s&:=\bigg\{ x\in\mathbb{R}\,:\,\bigg( \frac{\partial}{\partial x}\langle \tilde{X}_s,\Phi_x^n \rangle \bigg)^2 \Psi_s(x)\leq \langle \tilde{X}_s,\Phi_x^n \rangle \frac{\partial}{\partial x} \langle \tilde{X}_s,\Phi_x^n \rangle \frac{\partial}{\partial x}\Psi_s(x) \bigg\}\\
    &\quad\quad\cap \lbrace x\in\mathbb{R}\,: \,\Psi_s(x)>0 \}\\
    &=A^{+,s} \cup A^{-,s} \cup A^{0,s},
  \end{align*}
  with
  \begin{align*}
    &A^{+,s}:=A^s\cap \bigg\lbrace \frac{\partial}{\partial x}\langle \tilde{X}_s,\Phi_x^n \rangle >0  \bigg\rbrace,\quad
    A^{-,s}:=A^s\cap \bigg \lbrace \frac{\partial}{\partial x}\langle \tilde{X}_s,\Phi_x^n \rangle <0  \bigg\rbrace\quad\text{and}\\
    &A^{0,s}:=A^s\cap \bigg \lbrace \frac{\partial}{\partial x}\langle \tilde{X}_s,\Phi_x^n \rangle =0  \bigg\rbrace.
  \end{align*}
  By Assumption~\ref{ass:Psi}~(i) and (iii), we can find an $\epsilon>0$ such that
  \begin{equation}\label{epsilon}
    B(0,\epsilon)\subset \Gamma(t)\quad \text{and}\quad \inf\limits_{s\leq t, x\in B(0,\epsilon)}\Psi_s(x)>0.
  \end{equation}
  On $A^{+,s}$ we have, by the definition of $A^s$, that
  \begin{equation*}
    0<\bigg( \frac{\partial}{\partial x}\langle \tilde{X}_s,\Phi_x^n \rangle \bigg)\Psi_s(x)\leq \langle \tilde{X}_s,\Phi_x^n \rangle\frac{\partial}{\partial x}\Psi_s(x),
  \end{equation*}
  and, therefore, we can bound the $A^{+,s}$-part of $I_{1,2}^{m,n}$ for any $t\in [0,T]$ by
  \begin{align*}
    &\int_0^t \int_{A^{+,s}} \psi_n(\langle \tilde{X}_s,\Phi_x^n \rangle) \frac{\partial}{\partial x}\langle \tilde{X}_s,\Phi_x^n \rangle  \langle \tilde{X}_s,\Phi_x^n \rangle |x|^{-\theta}\frac{\partial}{\partial x} \Psi_s(x)\dd x\dd s\\
    &\quad\leq\int_0^t \int_{A^{+,s}} \psi_n(\langle \tilde{X}_s,\Phi_x^n \rangle)|x|^{-\theta} \langle \tilde{X}_s,\Phi_x^n \rangle^2 \frac{(\frac{\partial}{\partial x} \Psi_s(x))^2}{\Psi_s(x)} \dd x\dd s\\
    &\quad\leq\int_0^t \int_{A^{+,s}}\frac{2}{n} \mathbbm{1}_{\lbrace a_{n-1}\leq |\langle \tilde{X}_s,\Phi_x^n \rangle| \leq a_n\rbrace} |x|^{-\theta} \langle \tilde{X}_s,\Phi_x^n \rangle \frac{(\frac{\partial}{\partial x} \Psi_s(x))^2}{\Psi_s(x)} \dd x\dd s\\
    &\quad\leq \frac{2a_n}{n}\int_0^t \int_{\mathbb{R}} \mathbbm{1}_{\lbrace \Psi_s(x)>0\rbrace} |x|^{-\theta} \frac{(\frac{\partial}{\partial x} \Psi_s(x))^2}{\Psi_s(x)} \dd x\dd s.
  \end{align*}
  Next, we split the integral by using $\epsilon$ from \eqref{epsilon} to be able to apply Assumption~\ref{ass:Psi} and Lemma~\ref{lem:calculus} and get
  \begin{align*}
    &\int_0^t \int_{A^{+,s}} \psi_n(\langle \tilde{X}_s,\Phi_x^n \rangle) \frac{\partial}{\partial x}\langle \tilde{X}_s,\Phi_x^n \rangle  \langle \tilde{X}_s,\Phi_x^n \rangle |x|^{-\theta}\frac{\partial}{\partial x} \Psi_s(x)\dd x\dd s\\
    &\quad\leq \frac{2a_n}{n}\int_0^t \bigg( \int_{B(0,\epsilon)} |x|^{-\theta} \frac{(\frac{\partial}{\partial x} \Psi_s(x))^2}{\Psi_s(x)} \dd x +2\|D^2\Psi_s\|_\infty \int_{\Gamma(t)\setminus B(0,\epsilon)}|x|^{-\theta}\dd x\bigg) \dd s\\
    &\quad=:\frac{2a_n}{n}C(\Psi,t).
  \end{align*}
  Note that $\epsilon>0$ is fixed and thus the $\epsilon$-dependence of $C(\Psi,t)$ does not matter.
  
  On the set $A^{-,s}$,
  \begin{equation}\label{Aminus}
    0>\bigg( \frac{\partial}{\partial x}\langle \tilde{X}_s,\Phi_x^n \rangle \bigg)\Psi_s(x)
    \geq \langle \tilde{X}_s,\Phi_x^n \rangle\frac{\partial}{\partial x}\Psi_s(x),
  \end{equation}
  holds and, since both terms in \eqref{Aminus} are negative, we can use the same calculation as above to get
  \begin{equation*}
    \int_0^t \int_{A^{+,s}} \psi_n(\langle \tilde{X}_s,\Phi_x^n \rangle) \frac{\partial}{\partial x}\langle \tilde{X}_s,\Phi_x^n \rangle  \langle \tilde{X}_s,\Phi_x^n \rangle |x|^{-\theta}\frac{\partial}{\partial x} \Psi_s(x)\dd x\dd s
    \leq \frac{2a_n}{n}C(\Psi,t).
  \end{equation*}
  Finally, on the set $A^{0,s}$,
  \begin{equation*}
    \int_0^t \int_{A^{+,s}} \psi_n(\langle \tilde{X}_s,\Phi_x^n \rangle) \frac{\partial}{\partial x}\langle \tilde{X}_s,\Phi_x^n \rangle  \langle \tilde{X}_s,\Phi_x^n \rangle |x|^{-\theta}\frac{\partial}{\partial x} \Psi_s(x)\dd x\dd s=0
  \end{equation*}
  and thus
  \begin{equation*}
    \mathbb{E}[I_{1,1}^{m,n}(t\wedge \mathcal{T})+I_{1,2}^{m,n}(t\wedge \mathcal{T})]
    \leq 4\alpha^2C(\Psi,t)\frac{a_n}{n} \to 0\quad\text{as } n\to\infty.
  \end{equation*}
  The remaining term in \eqref{I1I2I3mn}, we have to deal with, is
  \begin{equation*}
    I_{1,3}^{m,n}=\int_0^t \int_{\mathbb{R}} \phi_n'(\langle \tilde{X}_s,\Phi_x^n \rangle)\langle \tilde{X}_s,\Phi_x^n \rangle \Delta_\theta \Psi_s(x)\dd x\dd s.
  \end{equation*}
  Therefore, we apply Proposition~\ref{prop:my_prop6}~(iv) to get the pointwise convergence
  \begin{equation*}
    \phi_n'(\langle \tilde{X}_s,\Phi_x^n \rangle)\langle \tilde{X}_s,\Phi_x^n \rangle\to \tilde{X}(s,x) \quad\text{as } m,n\to\infty.
  \end{equation*} 
  To complete our proof, we only need to show uniform integrability of $|\phi_n'(\langle \tilde{X}_s,\Phi_x^n \rangle)\langle \tilde{X}_s,\Phi_x^n \rangle|$ in terms of $m,n\in\mathbb{N}$ on $([0,T]\times B(0,J(t))\times \Omega)$, since $\Psi$ vanishes outside $B(0,J(t))$. First, by the inequality $|\phi_n'|\leq 1$, we can bound
  \begin{align*}
    |\phi_n'(\langle \tilde{X}_s,\Phi_x^n \rangle)\langle \tilde{X}_s,\Phi_x^n \rangle|\leq \langle |\tilde{X}_s|,\Phi_x^n \rangle.
  \end{align*}
  Inserting the function $\Phi^{n}$ from \eqref{phi}, taking the mean and using Proposition~\ref{prop: regularity result}~(i), we can bound
  \begin{align}\label{unif_iint}
    \E[|\langle |\tilde{X}_s|,\Phi_x^n \rangle|]
    \leq\mathbb{E}\bigg[ \int_{\mathbb{R}}|\tilde{X}(s,y)|2\tilde{\Phi}_x^{m}(y) \dd y \bigg]
    \leq 2\sup\limits_{y\in\mathbb{R}} \mathbb{E}[|\tilde{X}(s,y)|] \int_{\mathbb{R}}\tilde{\Phi}_x^{m^{(n)}}(y)\dd y<\infty,
  \end{align}
  thus the claimed integrability holds and we get
  \begin{equation*}
    \lim\limits_{m,n\to\infty} \mathbb{E}[ I_{1,3}^{m,n}(t\wedge \mathcal{T})]
    \leq \mathbb{E}\bigg[ \int_0^{t\wedge \mathcal{T}}\int_{\mathbb{R}}|\tilde{X}(s,x)|\Delta_\theta \Psi_s(x)\dd x\dd s \bigg]
  \end{equation*}
  and, altogether, we have shown the statement.
      
  (ii) Again the inequality $|\phi_n'|\leq 1$ and the Lipschitz continuity of $\mu$ yield
  \begin{equation*}
    \mathbb{E}[ I_2^{m,n}(t\wedge \mathcal{T})]
    \lesssim \int_0^{t\wedge \mathcal{T}}\bigg(\int_{\R}\Phi_x^n(0)\Psi_s(x)\dd x\bigg)\mathbb{E}[|\tilde{X}(s,0)|]\dd s.
  \end{equation*}
  Sending $m,n\to\infty$ gives the statement as $\Phi^n_x(0)\to\delta_0(x)$.
      
  (iii) We set $g_{m,n}(s):=\langle \phi_n'(\langle \tilde{X}_s,\Phi_\cdot^n \rangle)\Phi_\cdot^n(0),\Psi_s \rangle$. Then, by $|\phi_n'|\leq 1$, one has
  \begin{equation*}
    |g_{m,n}(s)|
     =\bigg| \int_{\mathbb{R}} \phi_n'(\langle \tilde{X}_s,\Phi_x^n \rangle)\Phi_x^n(0)\Psi_s(x) \dd x \bigg|
     \leq \|\Psi\|_\infty \int_{\mathbb{R}} 2\tilde{\Phi}_0^{m}(x) \dd x
     =2\|\Psi\|_\infty
  \end{equation*}
  by the construction of $\Phi^{n}$ in~\eqref{phi}. Thus, $I_3^{m,n}(t\wedge \mathcal{T})$ is a continuous local martingale with quadratic variation
  \begin{align*}
    \langle I_3^{m,n} \rangle_{t\wedge \mathcal{T}}
    &\leq 4\|\Psi\|_\infty^2 \int_0^{t\wedge \mathcal{T}}( \sigma(s,X^1(s,0))-\sigma(s,X^2(s,0)))^2\dd s\\
    &\lesssim \int_0^{t\wedge \mathcal{T}}(|X^1(s,0)|+|X^2(s,0)|+2)^2\dd s
  \end{align*}
  by the growth condition on $\sigma$ and, consequently, by Proposition~\ref{prop: regularity result},
  \begin{equation*}
    \E[\langle I_3^{m,n} \rangle_{t\wedge \mathcal{T}}]<\infty,
  \end{equation*}
  such that $I_3^{m,n}(t\wedge \mathcal{T})$ is a square integrable martingale with mean~$0$.
  
  (iv) We want to calculate the limit as $n,m\to \infty$ of the term
  \begin{equation*}
    \mathbb{E}[I_5^{m,n}(t\wedge \mathcal{T})]
    =\mathbb{E}\bigg[ \int_0^{t\wedge \mathcal{T}} \langle \phi_n(\langle \tilde{X}_s,\Phi_\cdot^n \rangle),\dot{\Psi}_s \rangle \dd s \bigg].
  \end{equation*}
  Therefore, the same argumentation as in (i) with the uniform integrability in \eqref{unif_iint} and the boundedness of $|\dot{\Psi}_s|$ as a continuous function with compact support yield
  \begin{equation*}
    \lim\limits_{m,n\to\infty}\mathbb{E}[ I_5^{m,n}(t\wedge \mathcal{T}) ]
    = \mathbb{E}\bigg[ \int_0^{t\wedge \mathcal{T}}\int_{\mathbb{R}}|\tilde{X}(s,x)|\dot{\Psi}_s(x)\dd x\dd s \bigg].
  \end{equation*}
\end{proof}

\subsection{Key argument: Bounding the quadratic variation term}

What is left to bound in line \eqref{I1I2I5}, is the expectation of the quadratic variation term $I_4^{m,n}$. The main ingredient to be able to do this, will be the following Theorem~\ref{thm:thm1}.

Let us first introduce some definitions that we need to formulate the Theorem~\ref{thm:thm1}. Recall the definition of $T_K$ in~\eqref{def_stopptimeK}. Moreover, we define a semimetric on $[0,T]\times \mathbb{R}$ by
\begin{equation*}
  d((t,x),(t',x')):=|t-t'|^\alpha+|x-x'|,\quad t,t'\in [0,T],x,x'\in \mathbb{R},
\end{equation*}
and, for $K>0$, $N\in\N$ and $\zeta\in (0,1)$, the set
\begin{align}\label{def_Z}
  Z_{K,N,\zeta}
  &:=
  \left \{ (t,x)\in [0,T]\times [-1/2,1/2]: \begin{array}{l}
  t\leq T_K, \,|x|\leq 2^{-N\alpha-1},\\
  |t-\hat{t}|\leq 2^{-N}\,|x-\hat{x}|\leq 2^{-N\alpha},   \\
  \text{for some }(\hat{t},\hat{x})\in [0,T_K]\times[-1/2,1/2]\\
  \text{satisfying }|\tilde{X}(\hat{t},\hat{x})|\leq 2^{-N\zeta}
  \end{array}
  \right \}.
\end{align}

The following theorem improves the regularity of $\tilde{X}(t,x)$ when $|x|$ is small. For two measures $\mathbb{Q}_1$ and $\mathbb{Q}_2$ on some measurable space $(\tilde{\Omega},\tilde{\mathscr{F}})$, we call $\mathbb{Q}_1$ absolutely continuous with respect to $\mathbb{Q}_2$, denoted by $\mathbb{Q}_1\ll\mathbb{Q}_2$, if $\mathscr{N}_{1}\supseteq\mathscr{N}_{2}$, where $\mathscr{N}_i\in\tilde{\mathscr{F}}$ denotes the zero sets of $\mathbb{Q}_i$ in $(\tilde{\Omega},\tilde{\mathscr{F}})$.

\begin{theorem}\label{thm:thm1}
  Suppose Assumption~\ref{ass:coefficients} and let $\tilde{X}:=X^1-X^2$, where $X^i$ is a solution of the SPDE~\eqref{SPDE} with $X^i\in C([0,T]\times \R)$ a.s. for $i=1,2$. Let $\zeta \in (0,1)$ satisfy:
  \begin{align}\label{5.1}
    \exists N_\zeta &= N_\zeta(K,\omega)\in \mathbb{N}\text{ a.s. such that, for any }N\geq N_\zeta\text{ and any }(t,x)\in Z_{K,N,\zeta}:\notag\\
    &\left. \begin{array}{c} |t'-t|\leq 2^{-N},t'\leq T_K\\|y-x|\leq 2^{-N\alpha} \end{array} \right\} \quad \Rightarrow \quad |\tilde{X}(t,x)-\tilde{X}(t',y)|\leq 2^{-N\zeta}.
  \end{align}
  Let $\frac{1}{2}-\alpha<\zeta^1 < ( \zeta\xi+\frac{1}{2}-\alpha) \wedge 1$. Then, there is an $N_{\zeta^1}(K,\omega,\zeta)\in\mathbb{N}$ a.s. such that, for any $N\geq N_{\zeta^1}$ and any $(t,x)\in Z_{K,N,\zeta^1}$:
  \begin{align}\label{5.2}
    &\left. \begin{array}{c} |t'-t|\leq 2^{-N},t'\leq T_K\\|y-x|\leq 2^{-N\alpha} \end{array} \right\} \quad \Rightarrow \quad |\tilde{X}(t,x)-\tilde{X}(t',y)|\leq 2^{-N\zeta^1}.
  \end{align}
  Moreover, there is some measure $\Q^{X,K}$ on $(\Omega,\mathscr{F})$ such that $\mathbb{Q}^{X,K}\ll\P$ on $(\Omega,\mathscr{F})$ and $\P\ll\mathbb{Q}^{X,K}$ on $(\Omega,\mathscr{F}^K)$, where $\mathscr{F}^K:=\lbrace A\cap \lbrace T_K\geq T \rbrace\colon A\in\mathscr{F}\rbrace \subseteq\mathscr{F}$ is the $\sigma$-algebra restricted to $\lbrace T_K\geq T \rbrace$, and there are constants $R>1$ and $\delta,C,c_2>0$ depending on $\zeta$ and $\zeta^1$ (not on $K$) and $N(K)\in\mathbb{N}$ such that
  \begin{equation}\label{thm_probbound}
    \mathbb{Q}^{X,K}(N_{\zeta^1}\geq N)\leq C\bigg( \mathbb{Q}^{X,K}\bigg(N_\zeta\geq \frac{N}{R}\bigg)+ Ke^{-c_22^{N\delta}} \bigg)
  \end{equation}
  for $N\geq N(K)$.
\end{theorem}

\begin{proof}[Proof of Theorem~\ref{thm:thm1}]
  From the assumptions of Theorem~\ref{thm:thm1} and Assumption~\ref{ass:coefficients}, we are given the variables $\alpha\in [0,\frac{1}{2})$, $\zeta\in (0,1)$, $\xi \in (\frac{1}{2(1-\alpha)},1]$ and $\zeta_1<(\zeta\xi+\frac{1}{2}-\alpha)\wedge 1$. Moreover, fix arbitrary $(t,x),(t',y)\in [0,T_K]\times [-\frac{1}{2},\frac{1}{2}]$ such that w.l.o.g. $t\leq t'$ and given some $N\geq N_\zeta$,
  \begin{equation}\label{tundtstrich}
    |t-t'|\leq \epsilon:=2^{-N},\quad
    |x|\leq 2^{-N\alpha}\quad \text{and}\quad
    |x-y|\leq 2^{-N\alpha}.
  \end{equation}
  We define small numbers $\delta,\delta',\delta_1,\delta_2>0$ in the following way. We choose $\delta\in (0,\frac{1}{2}-\alpha)$ such that
  \begin{equation*}
    \zeta_1<\bigg( \bigg(\zeta\xi+\frac{1}{2}-\alpha\bigg)\wedge 1\bigg)-\alpha\delta<1.
  \end{equation*}
  Fixing $\delta'\in (0,\delta)$, we choose $\delta_1\in (0,\delta')$ sufficiently small that
  \begin{equation}\label{def:delta1}
    \zeta_1<\bigg( \bigg(\zeta\xi+\frac{1}{2}-\alpha\bigg)\wedge 1\bigg)-\alpha\delta+\alpha\delta_1<1.
  \end{equation}
  Furthermore, we define $\delta_2>0$ sufficiently small such that
  \begin{equation}\label{delta2}
    \delta'-\delta_2>\delta_1,
  \end{equation}
  and we set
  \begin{equation} \label{def:p}
    p:=\bigg( \bigg(\zeta\xi+\frac{1}{2}-\alpha\bigg)\wedge 1\bigg)-\alpha \bigg(\frac{1}{2}-\alpha\bigg)+\alpha\delta_1
  \end{equation}
  and
  \begin{align}\label{p_hut}
    \hat{p}&:=p+\alpha(\delta'-\delta_2-\delta_1)=\bigg(\bigg(\zeta\xi+\frac{1}{2}-\alpha\bigg)\wedge  1\bigg)-\alpha\bigg(\frac{1}{2}-\alpha\bigg)+\alpha(\delta'-\delta_2).
  \end{align}
  By \eqref{delta2}, we see that $\hat{p}>p$.

  Moreover, we introduce
  \begin{equation}\label{def_Dxytt}
    D^{x,y,t,t'}(s):=|p_{t-s}(x)-p_{t'-s}(y)|^2|\tilde{X}(s,0)|^{2\xi}
    \quad \text{and}\quad
    D^{x,t'}(s):=p_{t'-s}(x)^2|\tilde{X}(s,0)|^{2\xi}.
  \end{equation}

  Our goal is to bound the following expression, where we will explicitly determine the measure $\mathbb{Q}$ as in the statement of the theorem and the random variable $N_1:=N_1(\omega)$ (in \eqref{N1}), later:
  \begin{align}
    &\mathbb{Q}\bigg( |\tilde{X}(t,x)-\tilde{X}(t,y)|\geq |x-y|^{\frac{1}{2}-\alpha-\delta}\epsilon^p,(t,x)\in Z_{K,N,\zeta},N\geq N_1\bigg)\notag\\
    &+\mathbb{Q}\bigg( |\tilde{X}(t',x)-\tilde{X}(t,x)|\geq |t'-t|^{\alpha ( \frac{1}{2}-\alpha-\delta )}\epsilon^p,\,(t,x)\in Z_{K,N,\zeta},N\geq N_1\bigg)\notag\\
    &\quad\leq \mathbb{Q}\bigg( |\tilde{X}(t,x)-\tilde{X}(t,y)|\geq |x-y|^{\frac{1}{2}-\alpha-\delta}\epsilon^p,(t,x)\in Z_{K,N,\zeta},N\geq N_1,\notag\\
    &\quad\quad\quad \quad \int_0^t D^{x,y,t,t}(s)\dd s \leq |x-y|^{1-2\alpha-2\delta'}\epsilon^{2p} \bigg)\notag\\
    &\quad\quad+ \mathbb{Q}\bigg( |\tilde{X}(t',x)-\tilde{X}(t,x)|\geq |t'-t|^{\alpha ( \frac{1}{2}-\alpha-\delta)}\epsilon^p,(t,x)\in Z_{K,N,\zeta},N\geq N_1,\notag\\
    &\quad\quad\quad\quad \int_t^{t'}D^{x,t'}(s)\dd s+ \int_0^t D^{x,x,t,t'}(s)\dd s \leq (t'-t)^{2\alpha ( \frac{1}{2}-\alpha-\delta')}\epsilon^{2p} \bigg)\notag\\
    &\quad\quad+ \mathbb{Q}\bigg( \int_0^t D^{x,y,t,t}(s)\dd s > |x-y|^{1-2\alpha-2\delta'}\epsilon^{2p},(t,x)\in Z_{K,N,\zeta},N\geq N_1 \bigg)\notag\\
    &\quad\quad+ \mathbb{Q}\bigg( \int_t^{t'}D^{x,t'}(s)\dd s+ \int_0^t D^{x,x,t,t'}(s)\dd s >(t'-t)^{2\alpha ( \frac{1}{2}-\alpha-\delta' )}\epsilon^{2p} ,\notag\\
    &\quad\quad\quad\quad\quad(t,x)\in Z_{K,N,\zeta},N\geq N_1\bigg)\notag\\
    &\quad=:Q_1+Q_2+Q_3+Q_4. \label{q1q2q3q4}
  \end{align}
  We will proceed in three steps to prove the theorem:
  \begin{itemize}
    \item[Step (i):] explicitly choosing a measure $\mathbb{Q}^{X,K}$ as in the statement of the theorem, such that $Q_1$ and $Q_2$ in \eqref{q1q2q3q4} fulfill $Q_1+Q_2\leq ce^{-c'|t'-t|^{-2\alpha \delta''}}$ for some $c,c'>0$,
    \item[Step (ii):] showing that $Q_3=Q_4=0$ holds w.r.t.~$\P$ (and hence also w.r.t. $\mathbb{Q}^{X,K}$, since $\mathbb{Q}^{X,K}~\ll~\P$), if we choose the random variable $N_1:=cN_\zeta$ for some large enough deterministic constant $c>0$,
    \item[Step (iii):] completing the proof, using Step (i) and Step (ii).
  \end{itemize}

  \textit{Step (i):}
  Consider first the term $Q_1$. Note that on the measurable space $(\Omega,\mathscr{F}^K)$, where the restricted $\sigma$-algebra $\mathscr{F}^K$ on $\lbrace T_K\geq T\rbrace$ is defined in the statement of the theorem, Assumption~\ref{ass:coefficients}~(iii) yields the existence of some constant $C_K>0$ such that
  \begin{align*}
    \bigg| \frac{\mu(s,X^1(s,0))-\mu(s,X^2(s,0))}{\sigma(s,X^1(s,0))-\sigma(s,X^2(s,0))}\bigg|\leq C_K<\infty,
  \end{align*}
  for all $s\in[0,T]$ $\P$-a.s. on $(\Omega,\mathscr{F}^K)$ and, thus, we can apply Girsanov's theorem (see~\cite[Theorem~3.5.1]{Karatzas1988}) with the adapted process $(L_t)_{t\in[0,T]}$ defined by
  \begin{equation*}
    L_t:=-\int_0^t \frac{\mu(s,X^1(s,0))-\mu(s,X^2(s,0))}{\sigma(s,X^1(s,0))-\sigma(s,X^2(s,0))}\dd B_s,
  \end{equation*}
  whose stochastic exponential process $\mathscr{E}(L_t)$ is a martingale due to Novikov's condition (see \cite[Proposition~3.5.12]{Karatzas1988}). We define $\mathbb{Q}^{X,K}$ via the Radon--Nikodym derivative~$\mathscr{E}(L_T)$ of the measure~$\mathbb{Q}^{X,K}$ with respect to~$\P$, under which the process $(\tilde{B}^{X,K}_t)_{t\in[0,T]}$ is a Brownian motion, where $\tilde{B}^{X,K}_t=B_t-\langle B,L \rangle_t=B_t+A_t $ with $A_t:=\int_0^t \frac{\mu(s,X^1(s,0))-\mu(s,X^2(s,0))}{\sigma(s,X^1(s,0))-\sigma(s,X^2(s,0))}\dd s$ on $[0,T_K]$.

  To avoid measurability problems we re-define $\mathbb{Q}^{X,K}$ as a measure on $(\Omega,\mathscr{F})$ by setting
  \begin{equation*}
    \mathbb{Q}^{X,K}(A):=\mathbb{Q}^{X,K}(A\cap \lbrace T_K\geq T\rbrace )
  \end{equation*}
  for $A\in\mathscr{F}$. Girsanov's theorem implies that $\mathbb{Q}^{X,K}\ll \P$ on $(\Omega,\mathscr{F})$ and $\P\ll\mathbb{Q}^{X,K}$ on $(\Omega,\mathscr{F}^K)$.
  With this notation, we see that
  \begin{align*}
    &\tilde{X}(t,x)-\tilde{X}(t,y)\\
    &\quad=\int_0^t p_{t-s}^\theta(x)\Big( \sigma(s,X^1(s,0))-\sigma(s,X^2(s,0))\Big)\dd ( B_s+ A_s)\\
    &\quad\quad - \int_0^t p_{t-s}^\theta(y)\Big( \sigma(s,X^1(s,0))-\sigma(s,X^2(s,0))\Big)\dd ( B_s+ A_s)\\
    &\quad=\int_0^t \big( p_{t-s}^\theta(x)-p_{t-s}^\theta(y)\big)\Big( \sigma(s,X^1(s,0))-\sigma(s,X^2(s,0))\Big)\dd \tilde{B}_s^{X,K}.
  \end{align*}
  For fixed $t\in[0,T]$ and $x,y\in [-\frac{1}{2},\frac{1}{2}]$, the process
  \begin{equation*}
    S_{\tilde{t}}^{x,y}= \int_0^{\tilde{t}} ( p_{t-s}^\theta(x)-p_{t-s}^\theta(y))(\sigma(s,X^1(s,0))-\sigma(s,X^2(s,0)))\dd \tilde{B}^{X,K}_s,\quad\tilde{t}\in [0,t],
  \end{equation*}
  is a local $\mathbb{Q}^{X,K}$-martingale with quadratic variation
  \begin{align*}
    \langle S^{x,y}\rangle_{\tilde{t}}
    &=\int_0^{\tilde{t}} ( p_{t-s}^\theta(x)-p_{t-s}^\theta(y))^2 (\sigma(s,X^1(s,0))-\sigma(s,X^2(s,0)))^2\dd s\notag\\
    &\leq C_\sigma^2 \int_0^{\tilde{t}} (p_{t-s}^\theta(x)-p_{t-s}^\theta(y))^2|\tilde{X}(s,0)|^{2\xi}\dd s\notag\\
    &=C_\sigma^2 \int_0^{\tilde{t}}D^{x,y,t,t}(s)\dd s.
  \end{align*}

  Thus, working under $\mathbb{Q}^{X,K}$ in \eqref{q1q2q3q4}, we can bound the term~$Q_1$ as follows:
  \begin{align*}
    Q_1&\leq \mathbb{Q}^{X,K}\bigg(| S_t^{x,y} |\geq |x-y|^{\frac{1}{2}-\alpha-\delta}\epsilon^p,\,\int_0^t D^{x,y,t,t}(s)\dd s \leq |x-y|^{1-2\alpha-2\delta'}\epsilon^{2p} \bigg)\\
    &\leq \mathbb{Q}^{X,K}\big( |S_{t}^{x,y}|\geq |x-y|^{\frac{1}{2}-\alpha-\delta}\epsilon^p,\, \langle S^{x,y}\rangle_{t}\leq C_\sigma^2 |x-y|^{1-2\alpha-2\delta'}\epsilon^{2p}\big)
  \end{align*}
  by the definition of $D^{x,y,t,t}$.

  Next, we apply the Dambis--Dubins--Schwarz theorem, which states that the local $\Q^{X,K}$-martingale $S_{\tilde{t}}^{x,y}$ can be embedded into a $\Q^{X,K}$-Brownian motion $(\tilde{W}_{\tilde{t}})_{\tilde{t}\in[0,t]}$ such that $S_{\tilde{t}}^{x,y}=\tilde{W}_{\langle S^{x,y} \rangle_{\tilde{t}}}$ holds for all $\tilde{t}\in [0,t]$. Thus, with $z:=C_\sigma^2 |x-~y|^{1-2\alpha-2\delta'}\epsilon^{2p}$ we obtain
  \begin{align*}
    Q_1
    &\leq \mathbb{Q}^{X,K}\bigg( | \tilde{W}_{\langle S^{x,y}\rangle_t}| \geq |x-y|^{\frac{1}{2}-\alpha-\delta}\epsilon^p, \, \langle S^{x,y}\rangle_{t}\leq z \bigg)\\
    &\leq \mathbb{Q}^{X,K}\bigg(\sup\limits_{0\leq s\leq z}| \tilde{W}_{s}| \geq |x-y|^{\frac{1}{2}-\alpha-\delta}\epsilon^p \bigg),
  \end{align*}
  since from the first event follows always the second one. Thus, with the notation $\tilde{W}^*(t):=\sup\limits_{0\leq s\leq t}|\tilde{W}_{s}|$, the scaling property of Brownian motion and the reflection principle, we get
  \begin{align*}
    Q_1
    &\leq \mathbb{Q}^{X,K}\big( \tilde{W}^*(C_\sigma^2 |x-y|^{1-2\alpha-2\delta'}\epsilon^{2p})\geq  |x-y|^{\frac{1}{2}-\alpha-\delta}\epsilon^p \big)\\
    &= \mathbb{Q}^{X,K}\big( \tilde{W}^*(1)C_\sigma |x-y|^{\frac{1}{2}-\alpha-\delta'}\epsilon^{p} \geq |x-y|^{\frac{1}{2}-\alpha-\delta}\epsilon^p \big)\\
    &= 2\mathbb{Q}^{X,K}\big( \tilde{W}(1) \geq C_\sigma^{-1} |x-y|^{-\delta''} \big)
  \end{align*}
  with $\delta'':=\delta-\delta'>0$ and, applying the concentration inequality $\mathbb{Q}^{X,K}(N>a)\leq e^{-\frac{a^2}{2}}$ for standard normal distributed~$N$, we get
  \begin{equation}\label{eq2}
    Q_1
    \leq 2e^{-\frac{1}{2C_\sigma^2}|x-y|^{-2\delta''}}
    =:ce^{-c'|x-y|^{-2\delta''}},
  \end{equation}
  for some constants $c,c'>0$. With a very similar argumentation, we can use the probability measure~$\Q^{X,K}$ and proceed as above to derive the bound
  \begin{equation*}
    Q_2\leq ce^{-c'|t'-t|^{-2\alpha \delta''}},
  \end{equation*}
  where $c$ and $c'$ are the same constants as in \eqref{eq2}.

  \medskip
  \textit{Step (ii):}
  We want to show that the terms $Q_3$ and $Q_4$ in \eqref{q1q2q3q4} vanish $\P$-a.s., if we choose $N_1$ large enough. Therefore, we consider $(t,x)\in Z_{K,N,\zeta}$ and $(t',y)$ as in \eqref{tundtstrich} and begin by showing the following bound on $|\tilde{X}(s,0)|$ for $s\leq t'$:
  \begin{equation}\label{bound_X(s,0)}
    |\tilde{X}(s,0)| \leq \left\{
    \begin{array}{ll}
    3\epsilon^\zeta & \text{if } s\in [t-\epsilon,t'], \\
    (4+K)2^{\zeta N_\zeta}(t-s)^\zeta & \text{if } s\in [0,t-\epsilon]. \\
    \end{array}\right.
  \end{equation}

  To see \eqref{bound_X(s,0)}, we choose for $(t,x)\in Z_{K,N,\zeta}$ some $(\hat{t},\hat{x})$ as in the definition of $Z_{K,N,\zeta}$ in \eqref{def_Z} such that
  \begin{equation*}
    |t-\hat{t}|\leq \epsilon=2^{-N},\quad
    |x-\hat{x}|\leq \epsilon^\alpha \quad\text{and}\quad
    |\tilde{X}(\hat{t},\hat{x})|\leq 2^{-N\zeta}=\epsilon^\zeta.
  \end{equation*}
  Then, for $s\in [t-\epsilon,t']$, we see that $|t-s|\leq \epsilon$ by \eqref{tundtstrich}. Thus, by \eqref{5.1}, we obtain that
  \begin{align*}
    |\tilde{X}(s,0)|&\leq |\tilde{X}(\hat{t},\hat{x})|+|\tilde{X}(\hat{t},\hat{x})-\tilde{X}(t,x)|+|\tilde{X}(t,x)-\tilde{X}(s,0)|\\
    &\leq 3\cdot2^{-N\zeta}=3\epsilon^\zeta.
  \end{align*}
  For $s\in [t-2^{-N_\zeta},t-\epsilon]$, we can choose some $\tilde{N}\geq N_\zeta$ such that $2^{-(\tilde{N}+1)}\leq t-s\leq 2^{-\tilde{N}}$ due to $t-\epsilon\geq s$, i.e. $t-s\geq 2^{-N}$. Thus, we get
  \begin{align*}
    |\tilde{X}(s,0)|&\leq |\tilde{X}(\hat{t},\hat{x})|+|\tilde{X}(\hat{t},\hat{x})-\tilde{X}(t,x)|+|\tilde{X}(t,x)-\tilde{X}(s,0)|\\
    &\leq 2^{-N\zeta}+2^{-N\zeta}+2^{-\tilde{N}\zeta}\leq 2\cdot (t-s)^\zeta +2^\zeta 2^{-(\tilde{N}+1)\zeta}\\
    &\leq 4(t-s)^\zeta.
  \end{align*}
  Last, for $s\in [0,t-2^{-N_\zeta}]$ with $s\leq T_K$, i.e. $\tilde{X}$ is bounded by $K>0$, and $t-s\geq 2^{-N_\zeta}$, we can bound
  \begin{equation*}
    |\tilde{X}(s,0)|\
    \leq K\leq K(t-s)^{-\zeta}(t-s)^\zeta\\
    \leq K2^{N_\zeta \zeta}(t-s)^\zeta,
  \end{equation*}
  which shows the bound~\eqref{bound_X(s,0)}.

  For $Q_3$, using \eqref{bound_X(s,0)} and the definition of $D^{x,y,t,t'}$ in \eqref{def_Dxytt}, we can bound the term inside~$Q_3$ by
  \begin{align}
    \int_0^t D^{x,y,t,t}(s)\dd s
    &\leq 3^{2\xi}\int_{t-\epsilon}^t ( p_{t-s}(x)-p_{t-s}(y))^2  \epsilon^{2\zeta\xi}\dd s\notag\\
    &\quad+ (4+K)^{2\xi}2^{2\xi\zeta N_\zeta}\int_0^{t-\epsilon}( p_{t-s}(x)-p_{t-s}(y))^2 (t-s)^{2\zeta\xi}\dd s\notag\\
    &=:D_1(t)+D_2(t).\label{D1+D2}
  \end{align} 
  Now, by Lemma~\ref{lem:6.3} with $\beta=\frac{1}{2}-\alpha-\delta'$ and $\max(|x|,|y|)\leq 2\epsilon^\alpha$, we can bound
  \begin{align}
    D_1(t)
    &\lesssim \epsilon^{2\zeta\xi}|x-y|^{1-2\alpha}\max(|x|,|y|)^{(\frac{1}{\alpha}-1)2\beta}\notag\\
    &\lesssim \epsilon^{2\zeta\xi+2\delta'}|x-y|^{1-2\alpha-2\delta'}\epsilon^{(1-\alpha)2\beta}\notag\\
    &=\epsilon^{2(\frac{1}{2}-\alpha(\frac{3}{2}-\alpha)+\alpha\delta'+\xi\zeta)}|x-y|^{1-2\alpha-2\delta'}\notag\\
    &\leq \epsilon^{2\hat{p}}|x-y|^{1-2\alpha-2\delta'}\label{estimateD1}
  \end{align}
  by the definition of $\hat{p}$ in \eqref{p_hut}. For $D_2(t)$, we use Lemma~\ref{lem:6.1} with $\beta=1$ to bound
  \begin{align}
    D_2(t)&\lesssim 2^{2\xi\zeta N_\zeta}\int_0^{t-\epsilon} |x-y|^2 (t-s)^{2\zeta\xi-2\alpha-2}\epsilon^{2(1-\alpha)} \dd s\notag\\
    &= 2^{2\xi\zeta N_\zeta}|x-y|^{1-2\alpha-2\delta'}|x-y|^{1+2\alpha+2\delta'} \epsilon^{2(1-\alpha)} \bigg[ \frac{(t-s)^{-2\alpha-1+2\xi\zeta}}{-2\alpha-1+2\xi\zeta} \bigg]_0^{t-\epsilon}\notag\\
    &\lesssim 2^{2\xi\zeta N_\zeta}|x-y|^{1-2\alpha-2\delta'}\epsilon^{\alpha (1+2\alpha+2\delta')} \epsilon^{2(1-\alpha)} \epsilon^{((-2\alpha-1+2\xi\zeta)\wedge 0)-2\alpha\delta_2}\notag\\
    &= 2^{2\xi\zeta N_\zeta}|x-y|^{1-2\alpha-2\delta'}\epsilon^{2\hat{p}}.\label{estimateD2}
  \end{align}
  Hence, by inserting \eqref{estimateD1} and \eqref{estimateD2} into \eqref{D1+D2}, we obtain
  \begin{equation}\label{bound_Q3}
    \int_0^t D^{x,y,t,t}(s)\dd s\lesssim  2^{2\xi\zeta N_\zeta}|x-y|^{1-2\alpha-2\delta'}\epsilon^{2\hat{p}}.
  \end{equation}

  For $Q_4$, we can use \eqref{bound_X(s,0)} to bound the first summand in the definition of $Q_4$ by
  \begin{align}
    \int_t^{t'}D^{x,t'}(s)\dd s
    &=\int_{t}^{t'}p_{t'-s}(x)^2|\tilde{X}(s,0)|^{2\xi}\dd s\notag\\
    &\lesssim \int_{t}^{t'}(t'-s)^{-2\alpha}\epsilon^{2\zeta\xi}\dd s\notag\\
    &\lesssim \epsilon^{2\zeta\xi}|t'-t|^{1-2\alpha}\notag\\
    &\lesssim \epsilon^{2\xi\zeta}\epsilon^{2(\frac{1}{2}-\alpha-\alpha(\frac{1}{2}-\alpha)+\alpha\delta')}|t'-t|^{2\alpha(\frac{1}{2}-\alpha-\delta')}\notag\\
    &\lesssim \epsilon^{2\hat{p}}|t'-t|^{2\alpha(\frac{1}{2}-\alpha-\delta')},\label{bound_Dxt}
  \end{align}
  where we used that $|t-t'|\leq \epsilon$ and $\hat{p}<\frac{1}{2}-\alpha-\alpha(\frac{1}{2}-\alpha)+\alpha\delta'$.
  We split the second summand similar as before:
  \begin{equation}\label{D3+D4}
    \int_0^t D^{x,x,t,t'}(s)\dd s=\int_{t-\epsilon}^t D^{x,x,t,t'}(s)\dd s+\int_0^{t-\epsilon} D^{x,x,t,t'}(s)\dd s=:D_3(t)+D_4(t).
  \end{equation}
  By Lemma~\ref{lem:6.2}, we estimate
  \begin{align}
    D_3(t)&=\int_{t-\epsilon}^t|p_{t-s}(x)-p_{t'-s}(x)|^2|\tilde{X}(s,0)|^{2\xi}\dd s\notag\\
    &\lesssim \epsilon^{2\xi\zeta}|t'-t|^{1-2\alpha}\notag\\
    &\lesssim \epsilon^{2\hat{p}}|t'-t|^{2\alpha(\frac{1}{2}-\alpha-\delta')},\label{estimateD3}
  \end{align}
  where the last estimate follows as in \eqref{bound_Dxt}.

  For $D_4(t)$, using the inequality $(a+b)^2\leq 2(a^2+b^2)$, we obtain
  \begin{align}\label{D41+D42}
    D_4(t)&=\int_0^{t-\epsilon}|p_{t-s}(x)-p_{t'-s}(x)|^2|\tilde{X}(s,0)|^{2\xi}\dd s\notag\\
    &\leq 2(4+K)^{2\xi}2^{2\xi\zeta N_\zeta}\int_0^{t-\epsilon} \bigg| ( (t-s)^{-\alpha}-(t'-s)^{-\alpha})e^{-\frac{|x|^{1/\alpha}}{t-s}} \bigg |^2(t-s)^{2\xi\zeta}\dd s\notag\\
    &\quad+2(4+K)^{2\xi}2^{2\xi\zeta N_\zeta}\int_0^{t-\epsilon}\bigg|  (t'-s)^{-\alpha} \bigg( e^{-\frac{|x|^{1/\alpha}}{t-s}}-e^{-\frac{|x|^{1/\alpha}}{t'-s}}\bigg) \bigg|^2(t-s)^{2\xi\zeta}\dd s\notag\\
    &=:D_{4,1}+D_{4,2}.
  \end{align}
  For $D_{4,1}$, we use the inequality 
  \begin{equation}\label{inequality_ue}
    ((t-s)^{-\alpha}-(t'-s)^{-\alpha})e^{-\frac{|x|^{1/\alpha}}{t-s}}\leq (t-s)^{-\alpha-1}(t'-t).
  \end{equation}
  To see this, note that
  \begin{equation*}
    e^{-\frac{|x|^{1/\alpha}}{t-s}}
    \leq \bigg( \frac{t-s}{t'-s} \bigg)^\alpha e^{-\frac{|x|^{1/\alpha}}{t-s}} +\frac{t'-t}{t-s},
  \end{equation*}
  which holds since
  \begin{align}\label{holdsasin}
    \bigg(\frac{t-s}{t'-s} \bigg)^\alpha +\frac{t'-t}{t-s}
    &\geq  \frac{t-s}{t'-s}  +\frac{t'-t}{t-s}\notag\\
    &= \frac{t-s}{t'-s} +\frac{t'-s}{t-s}-1\geq 1
  \end{align}
  as $x\mapsto \frac{1}{x} +x\geq 2$ on $[0,1]$. Thus, using \eqref{inequality_ue}, we get
  \begin{align}
    D_{4,1}&\lesssim 2^{2\xi\zeta N_\zeta}\int_0^{t-\epsilon}(t-s)^{-2\alpha-2}(t'-t)^2(t-s)^{2\xi\zeta}\dd s\notag\\
    &\lesssim 2^{2\xi\zeta N_\zeta}(t'-t)^2 \epsilon^{((-2\alpha-1+\xi\zeta)\wedge 0)-2\alpha\delta_2}\notag\\
    &\lesssim 2^{2\xi\zeta N_\zeta}(t'-t)^{2\alpha(\frac{1}{2}-\alpha-\delta')}\epsilon^{2-2\alpha(\frac{1}{2}-\alpha-\delta')} \epsilon^{((-2\alpha-1+\xi\zeta)\wedge 0)-2\alpha\delta_2}\notag\\
    &= 2^{2\xi\zeta N_\zeta}(t'-t)^{\alpha(1-2\alpha-2\delta')} \epsilon^{2((-\alpha+\frac{1}{2}+\xi\zeta)\wedge 1)-\alpha\delta_2-\alpha(\frac{1}{2}-\alpha-\delta')}\notag\\
    &= 2^{2\xi\zeta N_\zeta}(t'-t)^{\alpha(1-2\alpha-2\delta')}\epsilon^{2\hat{p}}.\label{estimateD41}
  \end{align}
  For $D_{4,2}$, we use the inequality $|e^{-a}-e^{-b}|\leq |a-b|$ and then the bound $\frac{1}{t-s}-\frac{1}{t'-s}\leq \frac{t'-t}{(t-s)^2}$, which holds as in \eqref{holdsasin}, to get
  \begin{align}
    D_{4,2}&\lesssim 2^{2\xi\zeta N_\zeta}\int_0^{t-\epsilon}(t'-s)^{-2\alpha}\bigg| \frac{|x|^{1/\alpha}}{t-s}-\frac{|x|^{1/\alpha}}{t'-s} \bigg|^2 (t-s)^{2\xi\zeta}\dd s\notag\\
    &\lesssim 2^{2\xi\zeta N_\zeta}|x|^{2/\alpha}\int_0^{t-\epsilon}(t'-s)^{-2\alpha}(t-s)^{-4}(t'-t)^2 (t-s)^{2\xi\zeta}\dd s\notag\\
    &\lesssim 2^{2\xi\zeta N_\zeta}|x|^{2/\alpha}\epsilon^{-3-2\alpha+2\xi\zeta} (t'-t)^2\notag\\
    &\lesssim 2^{2\xi\zeta N_\zeta}|x|^{2/\alpha}\epsilon^{-3-2\alpha+2\xi\zeta} (t'-t)^{2\alpha(\frac{1}{2}-\alpha-\delta')}\epsilon^{2-2\alpha(\frac{1}{2}-\alpha-\delta')}\notag\\
    &=  2^{2\xi\zeta N_\zeta}|x|^{2/\alpha} \epsilon^{2(\frac{1}{2}-\alpha+\xi\zeta-\alpha(\frac{1}{2}-\alpha)+\alpha\delta')}(t'-t)^{\alpha(1-2\alpha-2\delta')}\notag\\
    &=  2^{2\xi\zeta N_\zeta}|x|^{2/\alpha} \epsilon^{2\hat{p}}(t'-t)^{\alpha(1-2\alpha-2\delta')}.\label{estimateD42}
  \end{align}
  Hence, \eqref{bound_Dxt} and plugging \eqref{estimateD3}, \eqref{D41+D42}, \eqref{estimateD41} and \eqref{estimateD42} into \eqref{D3+D4}, we obtain
  \begin{equation}\label{bound_Q4}
    \int_t^{t'}D^{x,t'}(s)\dd s + \int_0^t D^{x,x,t,t'}(s)\dd s\lesssim 2^{2\xi\zeta N_\zeta} |t'-t|^{\alpha(1-2\alpha-2\delta')}\epsilon^{2\hat{p}}.
  \end{equation}

  Combining \eqref{bound_Q3} and \eqref{bound_Q4}, we can denote $C>0$ to be the maximum of the two generic constants occuring in the estimates, to conclude, that if we can secure that
  \begin{equation}
    C2^{2\xi\zeta N_\zeta}\epsilon^{2\hat{p}} < \epsilon^{2p},\label{condition}
  \end{equation}
  then the conditions inside of $Q_3$ and $Q_4$ are never fulfilled and, thus, we get that $Q_3=Q_4=0$. By $\epsilon=2^{-N}$, \eqref{condition} is equivalent to
  \begin{equation*}
    C<2^{2N(\hat{p}-p)-2N_\zeta \xi\zeta},
  \end{equation*}
  and, since $\hat{p}-p>0$, fulfilled for all
  \begin{equation*}
    N>\frac{2\xi \zeta N_\zeta + log_2(C)}{2(\hat{p}-p)}.
  \end{equation*}
  Therefore, we can find a deterministic constant $c_{K,\zeta,\delta,\delta_1,\delta',\delta_2}$ such that, for all
  \begin{equation}\label{N1}
    N\geq N_1(\omega):= c_{K,\zeta,\delta,\delta_1,\delta',\delta_2}N_\zeta(\omega),
  \end{equation}
  $Q_3=Q_4=0$ holds.

  \medskip

  \textit{Step (iii):}
  We discretize $\tilde{X}(t,y)$ for $t\in [0,T_K]$ and $y\in [-\frac{1}{2},\frac{1}{2}]$ as follows:
  \begin{align*}
    M_{n,N,K}:=\max \Big\{ &\Big| \tilde{X}(j2^{-n},(z+1)2^{-\alpha n}) - \tilde{X}(j2^{-n},z2^{-\alpha n})\Big|\\
    &\quad+\Big| \tilde{X}((j+1)2^{-n},z2^{-\alpha n}) - \tilde{X}(j2^{-n},z2^{-\alpha n}) \Big|:\\
    &\quad|z|\leq 2^{\alpha n-1},(j+1)2^{-n}\leq T_K,j\in \mathbb{Z}_+,z\in \mathbb{Z}, \\
    &\quad(j2^{-n},z2^{-\alpha n})\in Z_{K,N,\zeta} \Big\}.
  \end{align*}
  Moreover, we define the event
  \begin{equation*}
    A_N:=\big\lbrace \omega\in\Omega\,:\, \text{for some }n\geq N,\, M_{n,N,K}\geq 2^{-n\alpha(\frac{1}{2}-\alpha-\delta)}2^{-Np},\,N\geq N_1 \big\rbrace.
  \end{equation*}
  Then, we get, by using \eqref{q1q2q3q4}, Step (i) and Step (ii), that for all $N\geq N_1$ as in \eqref{N1}:
  \begin{align*}
    \mathbb{Q}^{X,K}\bigg( \bigcup\limits_{N'\geq N}A_{N'}\bigg)
    &\leq \sum\limits_{N'=N}^{\infty}\sum\limits_{n=N'}^{\infty}\mathbb{Q}^{X,K} (M_{n,N',K}\geq  2\cdot 2^{-n\alpha(\frac{1}{2}-\alpha-\delta)}2^{-Np})\\
    &\lesssim \sum_{N'=N}^\infty \sum_{n=N'}^\infty 2^{(\alpha+1)n}e^{-c'2^{n\delta''\alpha}},
  \end{align*}
  since the total number of partition elements in each $M_{n,N,K}$ is at most $2~\cdot ~2^{\alpha n-1}\cdot ~K\cdot 2^n \lesssim K2^{(\alpha+1)n}$ (if $T_K=T$). Furthermore, we used that $|t-\hat{t}|\leq 2^{-n}$ and $|x-\hat{x}|\leq 2^{-n\alpha}$, which follows by the construction of $M_{n,N,K}$.

  We use the convexity $2^{x+y}\geq 2^x+2^y$ for $x,y\geq 0$ to estimate
  \begin{align*}
    \mathbb{Q}^{X,K}\bigg(\bigcup\limits_{N'\geq N}A_{N'}\bigg)
    &\lesssim \sum\limits_{N'=N}^{\infty}\sum\limits_{n=0}^{\infty} 2^{(\alpha+1)(n+N')}e^{-c'2^{(n+N')\delta''\alpha}}\notag\\
    &\leq \sum\limits_{N'=N}^{\infty}2^{(\alpha+1)N'}\sum\limits_{n=0}^{\infty} 2^{(\alpha+1)n}e^{-c'(2^{n\delta''\alpha}+2^{N'\delta''\alpha})}\notag\\
    &= \sum\limits_{N'=N}^{\infty}2^{(\alpha+1)N'}e^{-c'2^{N'\delta''\alpha}}\sum\limits_{n=0}^{\infty} 2^{(\alpha+1)n}e^{-c'2^{n\delta''\alpha}}\notag\\
    &= 2^{(\alpha+1)N}e^{-c'2^{N\delta''\alpha}}\sum\limits_{N'=0}^{\infty}2^{(\alpha+1)N'}e^{-c'2^{N'\delta''\alpha}}\sum\limits_{n=0}^{\infty} 2^{(\alpha+1)n}e^{-c'2^{n\delta''\alpha}}\notag\\
    &\lesssim e^{(\alpha+1)N}e^{-c'2^{N\delta''\alpha}}\\
    &\lesssim e^{-c_22^{N\delta''\alpha}},
  \end{align*}
  for some constant $c_2>0$, where we used convergence and thus finiteness of the two series in the fourth line by applying the ratio test
  \begin{equation*}
    \lim\limits_{n\to\infty}\Big| 2^{\alpha+1}e^{-c'(2^{(n+1)\delta^{\prime\prime}\alpha}-2^{n\delta^{\prime\prime}\alpha})}\Big|=0.
  \end{equation*}
  Therefore, we get for
  \begin{equation*}
    N_2(\omega):=\min\lbrace N\in\mathbb{N} \,:\,\omega\in A_{N'}^c\,\forall N'\geq N \rbrace,
  \end{equation*}
  where the superscript $c$ denotes the complement of a set, that
  \begin{equation}\label{P_N2}
    \mathbb{Q}^{X,K}(N_2>N)=\mathbb{Q}^{X,K}\bigg(\bigcup\limits_{N'\geq N}A_{N'}\bigg)\lesssim e^{-c_22^{N\delta''\alpha}},
  \end{equation}
  and thus $N_2<\infty$ $\Q^{X,K}$-a.s.

  We fix some $m\in\N$ with $m>3/\alpha$ and choose $N(\omega)\geq (N_2(\omega)+m)\wedge (N_1+m)$, which is finite a.s., such that holds:
  \begin{equation}\label{7}
    \forall n\geq N: M_{n,N,K}< 2^{-n\alpha(\frac{1}{2}-\alpha-\delta)}2^{-Np}\quad \text{a.s.}
  \end{equation}
  and $Q_3=Q_4=0$.

  Furthermore, we choose $(t,x)\in Z_{K,N,\zeta}$ and $(t',y)$ such that
  \begin{align*}
    d((t',y),(t,x)):=|t'-t|^\alpha+|y-x|\leq 2^{-N\alpha},
  \end{align*}
  and we choose points near $(t,x)$ as follows: for $n\geq N$, we denote by $t_n\in 2^{-n}\mathbb{Z}_+$ and $x_n\in 2^{-\alpha n}\mathbb{Z}$ for the unique points such that
  \begin{align*}
    &t_n\leq t<t_n+2^{-n},\\
    &x_n\leq x <x_n+2^{-\alpha n} \text{ for }x\geq 0
    \quad\text{or}\quad
    x_n-2^{-\alpha n}< x \leq x_n \text{ for }x< 0.
  \end{align*}
  We define $t_n',y_n$ analogously. Let $(\hat{t},\hat{x})$ be the points from the definition of $Z_{K,N,\zeta}$ with $|\tilde{X}(\hat{t},\hat{x})|\leq 2^{-N\zeta}$. Then, for $n\geq N$, we observe that
  \begin{align}\label{9}
    d((t_n',y_n),(\hat{t},\hat{x}))&\leq d((t_n',y_n),(t',y)) + d((t',y),(t,x)) + d((t,x),(\hat{t},\hat{x}))\notag\\
    &\leq |t_n'-t|^\alpha +|y-y_n|+2^{-N\alpha}+2\cdot 2^{-N\alpha}\notag\\
    &\leq 6\cdot 2^{-N\alpha}<2^{3-N\alpha}=2^{-\alpha(N-\frac{3}{\alpha})}\notag\\
    &<2^{-\alpha(N-m)},
  \end{align}
  which implies $(t_n',y_n)\in Z_{K,N-m,\zeta}$. We use that to finally formulate our bound. We also use the continuity of $\tilde{X}$ and our construction of the $t_n,x_n$ to get that
  \begin{equation*}
    \lim\limits_{n\to\infty}\tilde{X}(t_n,x_n)=\tilde{X}(t,x)\quad \text{a.s.}
  \end{equation*}
  and the same for $t_n',y_n$. Thus, by the triangle inequality:
  \begin{align*}
    |\tilde{X}(t,x)-\tilde{X}(t',y)|
    &=\bigg| \sum\limits_{n=N}^\infty \bigg( (\tilde{X}(t_{n+1},x_{n+1})-\tilde{X}(t_n,x_n)) + ( \tilde{X}(t'_{n},y_{n})-\tilde{X}(t'_{n+1},y_{n+1}) ) \bigg)\\
    &\quad+\tilde{X}(t_N,x_N)-\tilde{X}(t_N',y_N) \bigg|\\
    &\leq  \sum\limits_{n=N}^\infty | \tilde{X}(t_{n+1},x_{n+1})-\tilde{X}(t_n,x_n)| + | \tilde{X}(t'_{n},y_{n})-\tilde{X}(t'_{n+1},y_{n+1}) |\\
    &\quad+ |\tilde{X}(t_N,x_N)-\tilde{X}(t_N',y_N)|.
  \end{align*}
  Since we choose $t_n,x_n$ and $t'_n,y_n$ to be of the form of the discrete points in $M_{n,N,K}$ and, since we have~\eqref{9}, we can continue to estimate
  \begin{equation*}
    | \tilde{X}(t,x)-\tilde{X}(t',y) |
    \leq \sum\limits_{n=N}^\infty 2M_{n+1,N-m,K}+ |\tilde{X}(t_N,x_N)-\tilde{X}(t_N',y_N)|.
  \end{equation*}
  Because of $|t-t'|\leq 2^{-N}$ and our construction of $t_N,t'_N$, they must be equal or adjacent in $2^{-N}\mathbb{Z}_+$ and analogue for $x_N,y_N$. Thus, we get
  \begin{align*}
    | \tilde{X}(t,x)-\tilde{X}(t',y) |
    &\leq \sum\limits_{n=N}^\infty 2M_{n+1,N-m,K} + M_{N,N-m,K}\\
    &\leq 2\sum\limits_{n=N}^\infty M_{n,N-m,K}\\
    &\lesssim\sum\limits_{n=N}^\infty  2^{-n\alpha (\frac{1}{2}-\alpha-\delta)}2^{-(N-m)p}\\
    &= 2^{-(N-m)p} \sum\limits_{n=0}^\infty 2^{-(n+N)\alpha (\frac{1}{2}-\alpha-\delta)}\\
    &\lesssim 2^{mp}2^{-N(\alpha(\frac{1}{2}-\alpha-\delta)+p)} \\
    &< 2^{-N\zeta_1},
  \end{align*}
  where the last line follows with $\alpha(\frac{1}{2}-\alpha-\delta)+p>\zeta_1$, which holds by \eqref{def:delta1} and \eqref{def:p}, and for all
  \begin{equation}\label{N3}
    N\geq N_3
  \end{equation}
  for some $N_3$ that is large enough such that $2^{mp}$ is dominated and thus depends deterministically on~$p$. Therefore, we have proven Theorem~\ref{thm:thm1} with
  \begin{equation*}
    N_{\zeta_1}(\omega):=\max \lbrace N_2(\omega)+m,N_{\zeta}(\omega)+m,c_{K,\zeta,\delta,\delta_1,\delta',\delta_2}N_{\zeta}(\omega)+m,N_3 \rbrace
  \end{equation*}
  by $N_{\zeta_1}$ chosen in that way due to \eqref{7}, Step (ii), \eqref{N1} and \eqref{N3}. If we denote $R':=1\vee c_{K,\zeta,\delta,\delta_1,\delta',\delta_2}$ and consider some $N\geq 2m\vee N_3$, \eqref{P_N2} implies
  \begin{align*}
    \mathbb{Q}^{X,K}(N_{\zeta_1}\geq N) &\leq \mathbb{Q}^{X,K}(N_{2}\geq N-m)+2\mathbb{Q}^{X,K}\bigg(N_{\zeta}\geq \frac{N-m}{R'}\bigg)\\
    &\leq CKe^{-c_22^{(N-m)\delta''\alpha}}+2\mathbb{Q}^{X,K}(N_\zeta \geq N/R)
  \end{align*}
  for $R=2R'$ and $C>0$ not depending on $K$, which shows the probability bound in \eqref{thm_probbound} by re-defining $\delta:=\delta''\alpha>0$ and thus completes the proof.
\end{proof}

In the following we sometimes only write a.s. when we mean $\P$-a.s. Since $\mathbb{Q}^{X,K}\ll \P$, this implies $\mathbb{Q}^{X,K}$-a.s.

\begin{corollary}\label{cor:Cor5}
  With the hypotheses of Theorem~\ref{thm:thm1} and $\frac{1}{2}-\alpha <\zeta < \frac{\frac{1}{2}-\alpha}{1-\xi}\wedge 1$, there is an a.s. finite positive random variable $C_{\zeta,K}(\omega)$ such that, for any $\epsilon\in (0,1]$, $t\in [0,T_K]$ and $|x|<\epsilon^\alpha$, if $|\tilde{X}(t,\hat{x})|\leq \epsilon^\zeta$ for some $|\hat{x}-x|\leq \epsilon^\alpha$, then
  \begin{equation}\label{cor_ineq}
    |\tilde{X}(t,y)|\leq C_{\zeta,K}\epsilon^\zeta,
  \end{equation}
  whenever $|x-y|\leq \epsilon^\alpha$.

  Moreover, there are constants $\delta, C_1, c_2,\tilde{R}>0$, depending on $\zeta$ (but not on $K$), and $r_0(K)>0$ such that
  \begin{equation}\label{cor_prob_bound}
    \mathbb{Q}^{X,K}(C_{\zeta,K}\geq r)
    \leq C_1 \bigg[ \Q^{X,K}\bigg( N_{\frac{\alpha}{2}(\frac{1}{2}-\alpha)}\geq \frac{1}{\tilde{R}}\log_2\bigg(\frac{r-6}{K+1}\bigg) \bigg) +K e^{ -c_2\big( \frac{r-6}{K+1} \big)^\delta } \bigg]
  \end{equation}
  for all $r\geq r_0(K)> 6+(K+1)$, where $\Q^{X,K}$ is the probability measure from Theorem~\ref{thm:thm1}.
\end{corollary}

\begin{proof}
  We will derive the statement by an appropriate induction. We start by choosing 
  \begin{equation*}
    \zeta_0:=\frac{\alpha}{2}\bigg(\frac{1}{2}-\alpha\bigg),
  \end{equation*}
  to be able to use the regularity result from Proposition~\ref{prop: regularity result}. Indeed, by \ref{prop: regularity result}~(ii) we get the inequality \eqref{5.1} with $\zeta_0$ by Kolmogorov's continuity theorem.
  
  Now, we define 
  \begin{equation*}
    \zeta_{n+1}:=\bigg[ \bigg( \zeta_n \xi+\frac{1}{2}-\alpha \bigg) \wedge 1\bigg]\bigg(1-\frac{1}{n+d}\bigg)
  \end{equation*}
  for some $d\in\R$. We chose that $d$ given $\zeta_0$ big enough such that $\zeta_1>\frac{1}{2}-\alpha$. Moreover, it is clearly $\zeta_{n+1}>\zeta_n$. Thus, we get inductively that $\zeta_n\uparrow \frac{\frac{1}{2}-\alpha}{1-\xi}\wedge 1$ and, for every fixed $\zeta\in \Big( \frac{1}{2}-\alpha,\frac{\frac{1}{2}-\alpha}{1-\xi}\wedge 1 \Big)$ as in the statement, we can find $n_0\in\mathbb{N}$ such that $\zeta_{n_0}\geq \zeta>\zeta_{n_0-1}$. By applying Theorem~\ref{thm:thm1} $n_0$-times, we get \eqref{5.1} for $\zeta_{n_0-1}$ and, hence, \eqref{5.2} for $\zeta_{n_0}$.

  We derive the estimation \eqref{cor_ineq} for all $0<\epsilon\leq 1$. Therefore, we consider first $\epsilon\leq 2^{-N_{\zeta_{n_0}}}$, where we got $N_{\zeta_{n_0}}$ from the application of Theorem~\ref{thm:thm1} to $\zeta_{n_0-1}$. Further, we choose $N\in\mathbb{N}$ such that $2^{-N-1}<\epsilon\leq 2^{-N}$ and, thus, $N\geq N_{\zeta_{n_0}}$. Also, we choose $t\leq T_K$ and $|x|\leq \epsilon^\alpha\leq 2^{-N\alpha}$ such that, by assumption of Theorem~\ref{thm:thm1}, for some $|\hat{x}-x|\leq \epsilon^\alpha\leq 2^{-N\alpha}$,
  \begin{equation*}
    |\tilde{X}(t,\hat{x})|\leq \epsilon^\zeta\leq 2^{-N\zeta}\leq 2^{-N\zeta_{n_0-1}}.
  \end{equation*}
  Hence, $(t,x)\in Z_{K,N,\zeta_{n_0-1}}$. For any $y$ such that $|y-x|\leq \epsilon^\alpha$, we get, by \eqref{5.2},
  \begin{align*}
    |\tilde{X}(t,y)|&\leq |\tilde{X}(t,\hat{x})|+|\tilde{X}(t,\hat{x})-\tilde{X}(t,x)|+|\tilde{X}(t,x)-\tilde{X}(t,y)|\\
    &\leq 2^{-N\zeta}+2^{-N\zeta_{n_0}}+2^{-N\zeta_{n_0}}\leq 3\cdot 2^{-N\zeta}\leq 6\epsilon^\zeta.
  \end{align*}
  Now, we consider $\epsilon\in (2^{-N_{\zeta_{n_0}}},1]$. Then, for $(t,x)$ and $(t,y)$ as in the assumption, we get
  \begin{align*}
    |\tilde{X}(t,y)|&\leq |\tilde{X}(t,x)|+|\tilde{X}(t,y)-\tilde{X}(t,x)|\\
    &\leq K+2^{-N\zeta}\leq (K+1)2^{N_{\zeta_{n_0}}\zeta}\epsilon^\zeta
  \end{align*}
  by $\epsilon2^{N_{\zeta_{n_0}}}>1$ and, therefore, we have shown~\eqref{cor_ineq} with $C_{\zeta,K}=(K+1)2^{N_{\zeta_{n_0}}\zeta}+6$.

  It remains to show the estimate~\eqref{cor_prob_bound}. Therefore, we use \eqref{thm_probbound} to conclude that
  \begin{align*}
    \mathbb{Q}^{X,K}\bigg( C_{\zeta,K}\geq r \bigg)
    &=\mathbb{Q}^{X,K}\bigg( 2^{N_{\zeta_{n_0}}\zeta}\geq \frac{r-6}{K+1}\bigg)
    =\mathbb{Q}^{X,K}\bigg( N_{\zeta_{n_0}}\geq \frac{1}{\zeta} \log_2\bigg(\frac{r-6}{K+1}\bigg)\bigg)\\
    &\leq C\bigg(\mathbb{Q}^{X,K}\bigg( N_{\zeta_{n_0-1}}\geq \frac{1}{R\zeta } \log_2\bigg(\frac{r-6}{K+1} \bigg)\bigg)+ K \textup{exp}\bigg(-c_22^{\frac{\delta}{\zeta}\log_2\big(\frac{r-6}{K+1} \big)}\bigg) \bigg).
  \end{align*}
  Applying \eqref{thm_probbound} $n_0$-times, we end up with
  \begin{align*}
    &\mathbb{Q}^{X,K}v( C_{\zeta,K}\geq r )\\
    &\quad\leq C^{n_0}\mathbb{Q}^{X,K}\bigg( N_{\frac{\alpha}{2}(\frac{1}{2}-\alpha)}
    \geq \frac{1}{\zeta R^{n_0}}\log_2\bigg(\frac{r-6}{K+1}\bigg) \bigg)
    + \sum\limits_{i=0}^{n_0}C^i K e^{-c_22^{R^{-i-1}\frac{\delta}{\zeta}\log_2\big(\frac{r-6}{K+1}\big)}}\\
    &\quad\leq C^{n_0}n_0\bigg(\mathbb{Q}^{X,K}\bigg( N_{\frac{\alpha}{2}(\frac{1}{2}-\alpha)}\geq \frac{1}{\tilde{R}}\log_2\bigg(\frac{r-6}{K+1}\bigg) \bigg) + K e^{-c_2\big(\frac{r-6}{K+1}\big)^{\frac{\delta}{\zeta R^{n_0}}}}\bigg)\\
    &\quad=: C_1\bigg(\mathbb{Q}^{X,K}\bigg( N_{\frac{\alpha}{2}(\frac{1}{2}-\alpha)}\geq \frac{1}{\tilde{R}}\log_2\bigg(\frac{r-6}{K+1}\bigg) \bigg) + K e^{-c_2\big(\frac{r-6}{K+1}\big)^{\tilde{\delta}}}\bigg),
  \end{align*}
  where $C_1,\tilde{\delta},\tilde{R}>0$ depend on $\zeta$ but not on $K$.
\end{proof}

We will handle the event on the right-hand side of \eqref{cor_prob_bound} under the measure~$\P$ again.

\begin{proposition}\label{prop:helpCor}
  In the setup and notation of Corollary~\ref{cor:Cor5}, one has
  \begin{equation*}
    \P\bigg( N_{\frac{\alpha}{2}(\frac{1}{2}-\alpha)}\geq \frac{1}{\tilde{R}}\log_2\bigg(\frac{r-6}{K+1}\bigg) \bigg) \lesssim \bigg(\frac{r-6}{K+1}\bigg)^{-\epsilon},
  \end{equation*}
  for some $\epsilon>0$.
\end{proposition}

\begin{proof}
  We show that, for every $M\in\R_+$,
  \begin{equation*}
    \mathbb{P}\big( N_{\frac{\alpha}{2}(\frac{1}{2}-\alpha)}\geq M \big)
    \lesssim 2^{-M\epsilon}
  \end{equation*}
  for some $\epsilon>0$, which then yields the statement. 
  
  Indeed, from Proposition~\ref{prop: regularity result}~(ii), we have that
  \begin{equation*}
    \mathbb{E}[ |\tilde{X}(t,x)-\tilde{X}(t',x')|^p]\lesssim  |t-t'|^{(\frac{1}{2}-\alpha)p}+|x-x'|^{(\frac{1}{2}-\alpha)p},
  \end{equation*}
  for all $p\geq 2$, $t,t'\in [0,T]$ and $|x|,|x'|\leq 1$. By choosing $(t,x)\in Z_{K,N,\zeta}$, $(t',x')$ from the definition of $Z_{K,N,\zeta}$ and $p>2$ such that $\alpha p(\frac{1}{2}-\alpha)= 1+\beta$ for some $\beta>0$, it holds that
  \begin{equation*}
    \mathbb{E} [|\tilde{X}(t,x)-\tilde{X}(t',x')|^p ]
    \lesssim 2^{-N(1+\beta)} + 2^{-N(1+\beta)}
    \lesssim 2^{-N(1+\beta)}.
  \end{equation*}
  We discretize $[0,T]\times [-1,1]$ on the dyadic rational numbers. For simplicity, we assume $T=1$. First, for some $n\in \N$, we keep some space variable $x\in \lbrace k2^{-n},k\in -2^n,\dots,0,1,\dots,2^n \rbrace$ fixed and apply Markov's inequality to get
  \begin{equation*}
    \mathbb{P}\Big( |\tilde{X}(k2^{-n},x)-\tilde{X}((k-1)2^{-n},x)|\geq 2^{-\zeta n} \Big)
    \lesssim 2^{\zeta n p}2^{-n(1+\beta)}= 2^{-n(1+\beta-\zeta p)}
  \end{equation*}
  for any $k\in 1,\dots,2^n$. Next, we define the following events:
  \begin{align*}
    &A_n=A_n(\zeta):=\bigg\lbrace \max\limits_{k\in\lbrace -2^n+1,\dots,2^n \rbrace} |\tilde{X}(k2^{-n},x)-\tilde{X}((k-1)2^{-n},x)| \geq 2^{-\zeta n-1}\bigg\rbrace,\\
    &B_n:=\bigcup\limits_{m=n}^\infty A_m,\quad N:=\limsup\limits_{n\to\infty}A_n=\bigcap\limits_{n=1}^\infty B_n.
  \end{align*}
  Then, for every $n\in\mathbb{N}$,
  \begin{align}\label{lab}
    \mathbb{P}(A_n)
    &\leq \sum\limits_{k=-2^n+1}^{2^n}\mathbb{P}\Big( |\tilde{X}(k2^{-n},x)-\tilde{X}((k-1)2^{-n},x)|\geq 2^{-\zeta n-1} \Big)\notag\\
    &\lesssim 2^{n+2}2^{-n(1+\beta-\zeta p)+p}=2^{2+p}2^{-n(\beta-\zeta p)}.
  \end{align}
  We choose, for $\zeta=\frac{\alpha}{2}(\frac{1}{2}-\alpha)$,
  \begin{equation*}
    p>\max \bigg\lbrace \frac{1+\beta}{\alpha(\frac{1}{2}-\alpha)},\frac{1}{\frac{\alpha}{2}-\zeta-\alpha^2}  \bigg\rbrace.
  \end{equation*}
  Note that $\frac{\alpha}{2}-\zeta-\alpha^2=\frac{\alpha}{2}-\frac{\alpha}{2}(\frac{1}{2}-\alpha)-\alpha^2=\frac{\alpha}{4}-\frac{\alpha^2}{2}>0$ as $\alpha<\frac{1}{2}$. Then, we have that
  \begin{equation}\label{lab2}
    0< p \bigg(\frac{\alpha}{2}-\zeta-\alpha^2 \bigg)-1=\alpha p \bigg(\frac{1}{2}-\alpha \bigg)-1-\zeta p=\beta -\zeta p
  \end{equation}
  and from \eqref{lab} it follows by the geometric series that
  \begin{equation*}
    \mathbb{P}(B_n)
    \leq \sum\limits_{m=n}^\infty \mathbb{P}(A_m)\lesssim 2^{2+p}\frac{2^{-n(\beta-\zeta p)}}{1-2^{\zeta p-\beta}}\to 0 \quad \text{as } n\to\infty,
  \end{equation*}
  where $2^{\zeta p-\beta}<1$ because of \eqref{lab2}.

  Analogously, we fix some time variable $t$ and get an analogue version of inequality \eqref{lab}. Now, we fix an event $\omega\in\Omega$ and some
  \begin{equation*}
    N\geq N_{\frac{\alpha}{2}(\frac{1}{2}-\alpha)}(\omega),
  \end{equation*}
  where $N_{\frac{\alpha}{2}(\frac{1}{2}-\alpha)}(\omega)$ is such that 
  \begin{equation*}
    \omega \notin \bigcup\limits_{n=N_{\frac{\alpha}{2}(\frac{1}{2}-\alpha)}}^\infty A_n,
  \end{equation*}
  and this should also hold for the union of the analogue sets for fixed $t$, denote those by $A_n^{(2)}$.

  Let $t,t',x,x'\in D_N$ with $|t-t'|\leq 2^{-N}$ and $|x-x'|\leq 2^{-\alpha N}$. Then, we have
  \begin{align*}
    |\tilde{X}(t,x,\omega)-\tilde{X}(t',x',\omega)|
    &\leq |\tilde{X}(t,x,\omega)-\tilde{X}(t',x,\omega)| + |\tilde{X}(t',x,\omega)-\tilde{X}(t',x',\omega)|\\
    &\leq 2\cdot 2^{-\zeta N-1}=2^{-\zeta N}.
  \end{align*}
  Then, we get from \eqref{lab} that 
  \begin{align*}
    \mathbb{P}(N_\zeta\geq M)
    \leq \sum\limits_{m=M}^\infty \mathbb{P}(A_m) + \sum\limits_{m=M}^\infty \mathbb{P}(A_m^{(2)})
    \lesssim \sum\limits_{m=M}^\infty 2^{-m(\beta-\zeta p)}
    =\frac{2^{-M(\beta-\zeta p)}}{1-2^{\zeta p-\beta}}
    \lesssim 2^{-M\epsilon}
  \end{align*}
  with $\epsilon:=\beta-\zeta p$, by the geometric series with $\beta-\zeta p>0$.

  By the density of the dyadic rational numbers in the reals and the continuity of $\tilde{X}$, the regularity extends to the whole $[0,T]\times [-1,1]$ and, thus, the statement holds.
\end{proof}

We want to fix $\zeta\in (0,1)$, that fulfills the requirements of the previous corollary.

\begin{lemma}\label{lem:Lem14}
  With fixed $\alpha\in (0,\frac{1}{2})$ and $\xi\in (\frac{1}{2},1)$ satisfying
  \begin{equation*}
    1>\xi>\frac{1}{2(1-\alpha)}>\frac{1}{2},
  \end{equation*}
  we can choose $\zeta\in (0,1)$ such that
  \begin{equation}\label{xiabschatzung}
    \frac{\alpha}{2\xi-1}<\zeta <\bigg( \frac{\frac{1}{2}-\alpha}{1-\xi}\wedge 1 \bigg).
  \end{equation}
  Especially, we get 
  \begin{equation*}
    \eta:= \frac{\zeta}{\alpha}>\frac{1}{2\xi -1}.
  \end{equation*}
\end{lemma}

\begin{proof}
  First, we consider $\frac{\frac{1}{2}-\alpha}{1-\xi}<1$. In this case, we have that
  \begin{align*}
    \frac{\frac{1}{2}-\alpha}{1-\xi}-\frac{\alpha}{2\xi-1}&=\frac{(\frac{1}{2}-\alpha)(2\xi-1)-\alpha(1-\xi)}{(1-\xi)(2\xi-1)}\\
    &=\frac{\xi-\frac{1}{2}-2\alpha\xi+\alpha-\alpha+\alpha\xi}{(1-\xi)(2\xi-1)}=\frac{\xi(1-\alpha)-\frac{1}{2}}{(1-\xi)(2\xi-1)}>0,
  \end{align*}
  by the assumption on $\xi$.
  
  On the other hand, if $\frac{\frac{1}{2}-\alpha}{1-\xi}\geq 1$, then $\alpha\leq \xi-\frac{1}{2}$, i.e. $\frac{\alpha}{2\xi-1}\leq \frac{1}{2}$, and we can fix $\zeta$ such that~\eqref{xiabschatzung} holds.
\end{proof}

Let us finally introduce the following stopping time, that plays a central role for the following Lemma~\ref{lem:crucial}, and is the reason, why we needed Corollary~\ref{cor:Cor5} and Proposition~\ref{prop:helpCor}:
\begin{align}\label{def:T_xi,K}
  T_{\zeta,K}:=\inf\limits_{t\geq 0} \left\{
   \begin{array}{l}
    t\leq T_K \text{ and there exist }\epsilon\in (0,1],\hat{x},x,y\in \mathbb{R}\text{ with }\\
    |x|\leq \epsilon^\alpha,|\tilde{X}(t,\hat{x})|\leq \epsilon^\zeta,|x-\hat{x}|\leq \epsilon^\alpha,|x-y|\leq \epsilon^\alpha \\
    \text{ such that }|\tilde{X}(t,y)|>c_0(K)\epsilon^\zeta \\
  \end{array}
  \right \}\wedge T_K\wedge T,
\end{align}
where $c_0(K):=r_0(K)\vee K^2>0$ with $r_0(k)$ from Corollary~\ref{cor:Cor5}.

\begin{corollary}\label{Cor:StoppzeitggT}
  The stopping time $T_{\zeta,K}$ fulfills $T_{\zeta,K}\to T$ as $K\to\infty$ a.s.
\end{corollary}

\begin{proof}
  We fix arbitrary $K,\tilde{K}>0$ such that $\tilde{K}\leq K$. We can bound for any $t\in[0,T)$, 
  \begin{align}
    \P\big( T_{\zeta,K}\leq t \big) &\leq \P\Big( \lbrace T_{\zeta,K}\leq t\rbrace\cap \lbrace T_{\tilde{K}}\geq T \rbrace \Big) +\P\big( T_{\tilde{K}}<T \big)\notag\\
    &=: P_1^{K,\tilde{K}}+P_2^{\tilde{K}}.\label{eq:P1P2}
  \end{align}
  We show that $\lim_{K\to\infty}P_1^{K,\tilde{K}}=0$. For this purpose, we consider the probability measure~$\mathbb{Q}^{X,\tilde{K}}$ from Corollary~\ref{cor:Cor5}. By the definition of $T_{\zeta,K}$ and Corollary~\ref{cor:Cor5}, we obtain that
  \begin{align}
    &\mathbb{Q}^{X,\tilde{K}}\Big( \lbrace T_{\zeta,K}\leq t\rbrace\cap \lbrace T_{\tilde{K}}\geq T \rbrace \Big)\notag \\
    &\qquad\leq \mathbb{Q}^{X,\tilde{K}}\big(T_K\leq t\big) + \mathbb{Q}^{X,\tilde{K}}\big( C_{\zeta,K}>c_0(K) \big)\notag\\
    &\qquad\leq \mathbb{Q}^{X,\tilde{K}}\big(T_K\leq t\big) + C_1\bigg[ \mathbb{Q}^{X,\tilde{K}}\Big( N_{\frac{\alpha}{2}(\frac{1}{2}-\alpha)}\geq \frac{1}{\tilde{R}}\log_2\Big( \frac{K^2-6}{\tilde{K}+1} \Big) \Big)+\tilde{K}e^{-c_2\big( \frac{K^2-6}{\tilde{K}+1} \big)	\delta} \bigg].\label{ineq:boundQXK}
  \end{align}
  By Proposition~\ref{prop:helpCor} we know that the respective of the second probability on the right-hand side of \eqref{ineq:boundQXK} with $\P$ instead of $\Q^{X,\tilde{K}}$ tends to zero as $K\to\infty$. Since $\mathbb{Q}^{X,\tilde{K}}\ll\P$ holds on $(\Omega,\mathscr{F})$, $\lim_{K\to\infty} \P(A_K)=0$ implies $\lim_{K\to\infty} \Q^{X,\tilde{K}}(A_K)=0$ for any sequence $(A_K)_{K\in\N}$ of events in $\Omega$ (see e.g. \cite[Theorem~6.11]{Rudin1987}) and, since $T_K\to \infty$ as $K\to \infty$ a.s., by the continuity of the solutions $X^1$ and $X^2$, we conclude that the whole right-hand side of \eqref{ineq:boundQXK} tends to zero as $K\to\infty$. Hence, since $\P\ll\mathbb{Q}^{X,\tilde{K}}$ on $(\Omega,\mathscr{F}^{\tilde{K}})$ and the event inside $P_1^{K,\tilde{K}}$ is trivially in $\mathscr{F}^{\tilde{K}}$, this implies also tending to zero for the respective $\P$-probability and we obtain $\lim\limits_{K\to\infty}P_1^{K,\tilde{K}}=0$.

  Therefore, using the continuity of $X^1$ and $X^2$ again, we can for every $\epsilon>0$ find some $\tilde{K}>0$ such that \eqref{eq:P1P2} yields
  \begin{equation*}
    \lim\limits_{K\to\infty}\P\big( T_{\zeta,K}\leq t \big)\leq \P\big( T_{\tilde{K}}<T \big)<\epsilon
  \end{equation*}
  and we obtain $\lim\limits_{K\to\infty}\P\big( T_{\zeta,K}\leq t \big)=0$, which yields the statement.
\end{proof}

Recall that we have a fixed constant $\eta>\frac{1}{2\xi-1}$, determined by Lemma~\ref{lem:Lem14}. We use this to fix the sequence $(m^{(n)})_{n\in\mathbb{N}}$ by defining
\begin{equation*}
  m^{(n)}:=a_{n-1}^{-\frac{1}{\eta}}>1,
\end{equation*}
where $a_n$ is the Yamada--Watanabe sequence, defined in \eqref{def:a_n_YW}. With this, we get the following crucial lemma, that regularizes $\tilde{X}$ based on regularity of the approximation $|\langle \tilde{X},\Phi^{n} \rangle|$.

\begin{lemma}\label{lem:crucial}
  For all $x\in B(0,\frac{1}{m})$ and $s\in [0,T_{\zeta,K}]$, if $|\langle \tilde{X}_s,\Phi_x^{n} \rangle|\leq a_{n-1}$, then
  \begin{equation*}
    \sup\limits_{y\in B(x,\frac{1}{m})} |\tilde{X}(s,y)|\leq \tilde{C}_Ka_{n-1},
  \end{equation*}
  for some $\tilde{C}_K>0$ only dependent on $K$.
\end{lemma}

\begin{proof}
  By the assumption $|\langle \tilde{X}_s,\Phi_x^{n} \rangle|\leq a_{n-1}$, we can apply Proposition~\ref{prop:my_prop6}~(v) to get that there exists $\hat{x}\in B(x,\frac{1}{m})$ with $|\tilde{X}(s,\hat{x})|\leq C_Ka_{n-1}$.
  
  For fixed $n\geq 1$, we define $\epsilon_n>0$ such that
  \begin{equation*}
    \epsilon_n^\alpha = \frac{1}{m^{(n)}}C_K^{\frac{1}{\eta}}
  \end{equation*}
  holds and, thus, by the choice $\eta=\frac{\zeta}{\alpha}$,
  \begin{equation*}
    C_Ka_{n-1}
    =C_K \bigg(\frac{1}{m} \bigg)^\eta
    =\bigg(\frac{C_K^{\frac{1}{\eta}}}{m}\bigg)^\eta
    =\epsilon_n^\zeta.
  \end{equation*} 
  We use this and the definition of $T_{\zeta,K}$ in \eqref{def:T_xi,K} to get the desired result with $\tilde{C}_K=C_Kc_0(K)$.
\end{proof}

Finally, we can handle the term $I_4^{m,n}$ from \eqref{I1I2I5}.

\begin{lemma}\label{lem:13}
  With $I_4^{m,n}$ from \eqref{I1I2I5} and $T_{\zeta,K}$ defined in \eqref{def:T_xi,K}, one has
  \begin{equation*}
    \lim\limits_{n\to\infty} \mathbb{E} [| I_4^{m,n}(t\wedge T_{\zeta,K})|]=0.
  \end{equation*}
\end{lemma}

\begin{proof}
  We use the H{\"o}lder continuity of $\sigma$ as well as the bounded support of $\psi_n$, the inequality $\psi_n(x)\leq\frac{2}{nx}\mathbbm{1}_{\lbrace a_n\leq x\leq a_{n-1} \rbrace}$, the boundedness of $\Psi$, Lemma~\ref{lem:crucial} and Proposition~\ref{prop:my_prop6}~(ii) to get
  \begin{align}\label{tendsto0}
    |I_4^{m,n}(t\wedge T_{\zeta,K})|
    &\lesssim \bigg| \int_0^{t\wedge T_{\zeta,K}} \int_{\mathbb{R}} \psi_n(|\langle \tilde{X}_s,\Phi_{x}^{n} \rangle|)\Phi_{x}^{n}(0)^2 \Psi_s(x) \dd x |\tilde{X}(s,0)|^{2\xi} \dd s \bigg|\notag\\
    &\lesssim \int_0^{t\wedge T_{\zeta,K}} \int_{\mathbb{R}}\mathbbm{1}_{\lbrace a_n\leq |\langle \tilde{X}_s,\Phi_{x}^{n} \rangle| 
    \leq a_{n-1} \rbrace} \frac{2}{n a_n}\Phi_{x}^{n}(0)^2 \Psi_s(x) \dd x |\tilde{X}(s,0)|^{2\xi} \dd s\notag \\
    &\leq \frac{\|\Psi\|_{\infty}}{n a_n} \int_0^{t\wedge T_{\zeta,K}} \int_{\mathbb{R}} \Phi_{x}^{n}(0)^2 \dd x (\tilde{C}_K a_{n-1})^{2\xi} \dd s \notag\\
    &\lesssim \frac{a_{n-1}^{2\xi}}{n a_n} \int_0^{t\wedge T_{\zeta,K}} \int_{\mathbb{R}} \Phi_{x}^{n}(0)^2 \dd x \dd s \notag\\
    &\lesssim \frac{a_{n-1}^{2\xi}}{n a_n} m^{(n)}\lesssim \frac{a_{n-1}^{2\xi}}{n a_n} a_{n-1}^{-\frac{1}{\eta}}=\frac{1}{n} \frac{a_{n-1}^{2\xi-\frac{1}{\eta}}}{ a_n} .
  \end{align}
  We know that $\frac{a_{n-1}}{a_n}=e^n$, $a_0=1$ and, thus, get inductively that $a_n=e^{-\frac{n(n+1)}{2}}$. Therefore, \eqref{tendsto0} tends to zero as $n\to\infty$ if
  \begin{equation*}
    n(n+1)-(2\xi- \eta^{-1})(n-1)n<0
  \end{equation*}
  for $n$ large, which holds if and only if $1-(2\xi-\eta^{-1})<0$, i.e., $\xi >\frac{1}{2}+\frac{1}{2\eta}$, which holds by Lemma~\ref{lem:Lem14}.
\end{proof}

We summarize the essential findings for the proof of Theorem~\ref{thm:main} in the next proposition.

\begin{proposition}\label{prop:step4}
  With $\Psi$ that fulfills Assumption~\ref{ass:Psi} and $T_{\zeta,K}$ defined in \eqref{def:T_xi,K} for $K>0$, one has, for $t\in[0,T]$, that
  \begin{align}\label{absch_kurzvorende}
    \int_{\mathbb{R}}\mathbb{E}[|\tilde{X}(t\wedge T_{\zeta,K},x)|]\Psi_{t\wedge T_{\zeta,K}}(x)\dd x
    &\lesssim \int_0^{t\wedge T_{\zeta,K}}\int_{\mathbb{R}}\mathbb{E}[|\tilde{X}(s,x)|]| \Delta_\theta \Psi_s(x)+\dot{\Psi}_s(x)|\dd x\dd s \notag\\
    &\quad+ \int_0^{t\wedge T_{\zeta,K}} \Psi_s(0) \mathbb{E}[|\tilde{X}(s,0)|]\dd s.
  \end{align}
\end{proposition}

\begin{proof}
  By Proposition~\ref{prop:Psi_step3}, Lemma~\ref{lem:i1i2i3i5i}, Lemma~\ref{lem:13} and sending $n\to\infty$ after applying Fatou's lemma to exchange limiting and the integral, we get
  \begin{align}\label{eq:final6}
    &\int_{\mathbb{R}}\mathbb{E}[ |\tilde{X}(t\wedge T_{\zeta,K},x)|]\Psi_{t\wedge T_{K,\zeta}}(x)\dd x\\
    &\quad= \int_{\mathbb{R}}\liminf\limits_{n\to\infty} \mathbb{E}  [\phi_n (\langle \tilde{X}_{t\wedge T_{\zeta,K}},\Phi_x^{n}\rangle )]\Psi_{t\wedge T_{K,\zeta}}(x)\dd x\notag\\
    &\quad\leq \liminf\limits_{n\to\infty} \int_{\mathbb{R}} \mathbb{E}[\phi_n (\langle \tilde{X} _{t\wedge T_{\zeta,K}},\Phi_x^{n}\rangle ) ]\Psi_{t\wedge T_{K,\zeta}}(x)\dd x\notag\\
    &\quad\lesssim \mathbb{E}\bigg[ \int_0^{t\wedge T_{\zeta,K}}\int_{\mathbb{R}}|\tilde{X}(s,x)|\big( \Delta_\theta \Psi_s(x)+\dot{\Psi}_s(x) \big)\dd x \dd s \bigg]\notag\\
    &\quad\quad+\mathbb{E}\bigg[ \int_0^{t\wedge T_{\zeta,K}} \Psi_s(0)|\tilde{X}(s,0)|\dd s \bigg]\notag.
  \end{align}
  Applying Fubini's theorem then yields~\eqref{absch_kurzvorende}.
\end{proof}

\section{Step~5: Removing the auxiliary localizations}\label{sec_step5}

We want to construct appropriate test functions $\Psi\in C_0^{\infty}([0,t],\R)$ for some fixed $t\in[0,T]$. They will be of the form 
\begin{equation}\label{def_psi}
  \Psi_{N,M}(s,x):=(S_{t-s}\phi_M(x))g_N(x)
\end{equation}
for $N,M\in\mathbb{N}$, where $(S_u)_{u\in [0,T]}$ denotes the semigroup generated by $\Delta_\theta$ and we specify the sequences of functions $\phi_M,g_N \in C_0^\infty(\mathbb{R})$ in the following.

With the sequence $(\phi_M)_{M\in\mathbb{N}}$ we want to approximate the Dirac distribution around~$0$. To that end, we define
\begin{equation*}
  \phi_M(x):=Me^{-M^2x^2}\mathbbm{1}_{\lbrace |x|\leq \frac{1}{M}\rbrace}+s_M(x),\quad M\geq 2,
\end{equation*}
where the function $s_M(x)$ extends smoothly to zero outside the ball $B(1,\frac{1}{M-1})$ such that  $\lim_{M\to\infty}\phi_M(x)=\delta_0(x)$ pointwise.

Moreover, let $(g_N)_{N\in\mathbb{N}}$ be a sequence of functions in $C_0^\infty(\mathbb{R})$ such that $g_N\colon \mathbb{R}\to [0,1]$,
\begin{equation*}
  B(0,N)\subset \lbrace x\in\mathbb{R}:\, g_N(x)=1\rbrace,\quad B(0,N+1)^C \subset \lbrace x\in\mathbb{R}:\, g_N(x)=0\rbrace,
\end{equation*}
and
\begin{equation}\label{bounddd}
  \sup\limits_{N\in\mathbb{N}}\Big[ \| |x|^{-\theta}g_N'(x) \|_\infty + \| \Delta_\theta g_N(x) \|_\infty \Big] =:C_g<\infty.
\end{equation}

We simplify the term on the right-hand side of~\eqref{eq:final6} in the next corollary.

\begin{corollary}\label{cor:simplify}
  With $\Psi_{N,M}$ constructed in \eqref{def_psi}, one has that
  \begin{align}\label{eq:cor:simplify}
    & \Delta_\theta \Psi_{N,M}(s,x) + \dot{\Psi}_{N,M}(s,x) \notag\\
    &\quad = 4\alpha^2 |x|^{-\theta}\Big( \frac{\partial}{\partial x}S_{t-s}\phi_M(x) \Big)\Big( \frac{\partial}{\partial x}g_N(x) \Big) + S_{t-s}\phi_M(x)\Delta_\theta g_N(x).
  \end{align}
\end{corollary}

\begin{proof}
  Recall, that, by the definition of the semigroup $(S_t)_{t\in[0,T]}$ in \eqref{eq: def_semigroup} and using the fundamental solution of \eqref{eq: evol_eq}, we get
  \begin{equation*}
    \Delta_\theta S_t \phi(x)=\frac{\partial}{\partial t}S_t \phi(x),\quad t\in[0,T],
  \end{equation*}
  for all $\phi\in C_0^\infty(\R)$. Therefore, the second term on the left-hand side of \eqref{eq:cor:simplify} equals
  \begin{align}\label{eq:calc1}
    \dot{\Psi}_{N,M}(s,x) &= g_N(x)\frac{\partial}{\partial s}\big(S_{t-s} \phi_M(x)\big)\notag\\
    &= -g_N(x)\Delta_\theta\big( S_{t-s} \phi_M(x)\big)\notag\\
    &= -2\alpha^2 g_N(x)\frac{\partial}{\partial x}\Big( |x|^{-\theta}\frac{\partial}{\partial x}\big( S_{t-s}\phi_M(x) \big) \Big)\notag\\
    &= -2\alpha^2 g_N(x)\Big(\frac{\partial}{\partial x} |x|^{-\theta}\Big)\Big(\frac{\partial}{\partial x} S_{t-s}\phi_M(x)  \Big) -2\alpha^2 g_N(x) |x|^{-\theta}\Big(\frac{\partial^2}{\partial x^2} S_{t-s}\phi_M(x) \Big).
  \end{align}
  For the first term on the left-hand side of \eqref{eq:cor:simplify}, we calculate
  \begin{align}\label{eq:calc2}
    &\Delta_\theta \Psi_{N,M}(s,x)\notag\\
    &\quad= 2\alpha^2 \frac{\partial}{\partial x}\Big( |x|^{-\theta}\frac{\partial}{\partial x}\Psi_{N,M}(s,x) \Big)\notag\\
    &\quad= 2\alpha^2 |x|^{-\theta}\frac{\partial^2}{\partial x^2}\Big( S_{t-s}\phi_M(x)g_N(x) \Big) + 2\alpha^2 \Big(\frac{\partial}{\partial x} |x|^{-\theta}\Big)\Big(\frac{\partial}{\partial x}S_{t-s}\phi_M(x)g_N(x) \Big)\notag\\
    &\quad= 4\alpha^2 |x|^{-\theta}\Big(\frac{\partial}{\partial x} S_{t-s}\phi_M(x)\Big)\Big(\frac{\partial}{\partial x}g_N(x)\Big)  + 2\alpha^2 |x|^{-\theta}g_N(x)\Big(\frac{\partial^2}{\partial x^2}S_{t-s}\phi_M(x) \Big)\notag\\
    &\quad\quad+ 2\alpha^2 |x|^{-\theta}\Big( S_{t-s}\phi_M(x)\Big)\Big(\frac{\partial^2}{\partial x^2}g_N(x) \Big) \notag\\
    &\quad\quad+ 2\alpha^2 \Big(\frac{\partial}{\partial x} |x|^{-\theta}\Big)\Big(\frac{\partial}{\partial x}S_{t-s}\phi_M(x)\Big)g_N(x)\notag\\
    &\quad\quad+ 2\alpha^2 \Big(\frac{\partial}{\partial x} |x|^{-\theta}\Big)\big(S_{t-s}\phi_M(x)\big)\Big(\frac{\partial}{\partial x}g_N(x)\Big).
  \end{align}
  Hence, adding up \eqref{eq:calc1} and \eqref{eq:calc2}, we obtain
  \begin{align*}
    &\Delta_\theta \Psi_{N,M}(s,x) + \dot{\Psi}_{N,M}(s,x)\\
    &\quad= 4\alpha^2 |x|^{-\theta}\Big(\frac{\partial}{\partial x} S_{t-s}\phi_M(x)\Big)\Big(\frac{\partial}{\partial x}g_N(x)\Big) +  2\alpha^2 |x|^{-\theta}\Big( S_{t-s}\phi_M(x)\Big)\Big(\frac{\partial^2}{\partial x^2}g_N(x) \Big)\\
    &\qquad + 2\alpha^2 \Big(\frac{\partial}{\partial x} |x|^{-\theta}\Big)\big(S_{t-s}\phi_M(x)\big)\Big(\frac{\partial}{\partial x}g_N(x)\Big)\\
    &\quad = 4\alpha^2 |x|^{-\theta}\Big(\frac{\partial}{\partial x} S_{t-s}\phi_M(x)\Big)\Big(\frac{\partial}{\partial x}g_N(x)\Big) + S_{t-s}\phi_M(x)\Delta_\theta g_N(x).
  \end{align*}
\end{proof}

With these observations, we want to show that the semigroup $(S_t)_{t\in[0,T]}$ can be exponentially bounded in the following way.

\begin{lemma}\label{lem:Lemma3}
  For any $\phi \in C_0^\infty(\mathbb{R})$, $t\in [0,T]$ and for any $\lambda>0$, there is a constant $C_{\lambda,\phi,t}>0$ such that
  \begin{equation*}
    \bigg| S_t \phi(x) + \frac{\partial}{\partial x}( S_t \phi(x)) \bigg|\mathbbm{1}_{\lbrace N+1>|x|>N\rbrace} \leq C_{\lambda,\phi,t}e^{-\lambda |x|}\mathbbm{1}_{\lbrace N+1> |x|>N\rbrace}
  \end{equation*}
  for any $N\geq 1$ and $x\in\mathbb{R}$.
\end{lemma}

\begin{proof}
  For $t=0$, the statement is trivial due to $S_0\phi(x)+\frac{\partial}{\partial x}(S_0\phi(x))=\phi(x)+\phi'(x)$, which is bounded with compact support. Thus, we fix $t>0$ and consider the first summand without the derivative. We use the inequality
  \begin{equation}\label{ineq:besselfunction}
    I_\nu(b)< \bigg(\frac{b}{a}\bigg)^\nu e^{b-a}\bigg( \frac{a+\nu+\frac{1}{2}}{b+\nu+\frac{1}{2}} \bigg)^{\nu+\frac{1}{2}} I_\nu(a),\quad 0<a<b,\nu>-1,
  \end{equation}
  from \cite[Theorem~2.1~(ii)]{Ifantis1991}, with $a=\frac{|y|^{1+\frac{\theta}{2}}}{t}$ and $b=\frac{|xy|^{1+\frac{\theta}{2}}}{t}$ such that $b>a$ due to $|x|>N\geq 1$. By the bound on $p_t^\theta(x,y)$ from Corollary~\ref{cor:bound_ptheta}, due to the compact support of $\phi$, which we denote by $S_\phi$, and using \eqref{ineq:besselfunction}, we get
  \begin{align}\label{exponentterm}
    S_t\phi(x)
    &\leq\int_{\mathbb{R}}\frac{(2+\theta)}{2t}|xy|^{\frac{(1+\theta)}{2}}e^{-\frac{|x|^{2+\theta}+|y|^{2+\theta}}{2t}}I_\nu\bigg(\frac{|xy|^{1+\frac{\theta}{2}}}{t}\bigg)\phi(y)\dd y\notag\\
    &\leq C_\phi \int_{S_\phi}\frac{(2+\theta)}{2t}|xy|^{\frac{(1+\theta)}{2}}e^{-\frac{|x|^{2+\theta}+|y|^{2+\theta}}{2t}}|x|^{\nu(1+\frac{\theta}{2})}e^{\frac{|xy|^{1+\frac{\theta}{2}}}{t}-\frac{|y|^{1+\frac{\theta}{2}}}{t}} I_\nu\bigg(\frac{|y|^{1+\frac{\theta}{2}}}{t}\bigg) \dd y\notag\\
    &\leq C_\phi \bigg(\int_{\mathbb{R}}\frac{(2+\theta)}{2t}|y|^{\frac{(1+\theta)}{2}}e^{-\frac{1^{2+\theta}+|y|^{2+\theta}}{2t}}I_\nu\bigg(\frac{|y|^{1+\frac{\theta}{2}}}{t}\bigg) \dd y\bigg)
    |x|^{(\nu+1)(1+\frac{\theta}{2})}e^{-\frac{|x-1|^{2+\theta}}{2t}}\notag\\
    &\quad  \times e^{c_\phi(|x|^{1+\frac{\theta}{2}}-1)} \notag\\
    &=C_\phi \bigg(\int_{\mathbb{R}}p_t^\theta(1,y) \dd y\bigg)
    |x|^{(\nu+1)(1+\frac{\theta}{2})}e^{-\frac{|x-1|^{2+\theta}}{2t}+c_\phi(|x|^{1+\frac{\theta}{2}}-1)+\lambda|x|}e^{-\lambda|x|} \notag\\
    &\leq C_{\lambda,\phi,t}e^{-\lambda|x|},
  \end{align}
  since the function $x\mapsto |x|^{(\nu+1)(1+\frac{\theta}{2})}e^{-\frac{|x-1|^{2+\theta}}{2t}+c_\phi(|x|^{1+\frac{\theta}{2}}-1)+\lambda|x|}$ attains a maximum on $\R$ for all $c_\phi>0$.
  
  For the second summand, we substitute $z=\frac{|xy|^{1+\frac{\theta}{2}}}{t}$ such that $\frac{1}{\partial x}=\frac{1+\frac{\theta}{2}}{t}y|xy|^{\frac{\theta}{2}}\frac{1}{\partial z}$, apply the product rule and $\frac{\partial}{\partial z}I_\nu(z)=\frac{\nu}{z}I_\nu(z)+I_{\nu+1}(z)$ (see \cite[page~67]{Magnus1966}) to get, for $|x|>1$,
  \begin{align}\label{resubstitute}
    \frac{\partial}{\partial x}( S_t\phi(x))
    &\leq \frac{\partial}{\partial x}\int_{\mathbb{R}}\frac{(2+\theta)}{2t}|xy|^{\frac{(1+\theta)}{2}}e^{-\frac{|x|^{2+\theta}+|y|^{2+\theta}}{2t}}I_\nu\bigg(\frac{|xy|^{1+\frac{\theta}{2}}}{t}\bigg)\phi(y)\dd y\notag\\
    &= \frac{(2+\theta)}{2t}\int_{\mathbb{R}}\frac{\partial}{\partial z}\bigg(|xy|^{\frac{(1+\theta)}{2}}\frac{1+\frac{\theta}{2}}{t}y|xy|^{\frac{\theta}{2}}e^{-\frac{|x|^{2+\theta}+|y|^{2+\theta}}{2t}}I_\nu (z)\bigg)\phi(y)\dd y\notag\\
    &= \frac{(2+\theta)}{2t}\int_{\mathbb{R}}\bigg(\frac{\partial}{\partial z}\bigg(|xy|^{\frac{(1+\theta)}{2}}\frac{1+\frac{\theta}{2}}{t}y|xy|^{\frac{\theta}{2}}e^{-\frac{|x|^{2+\theta}+|y|^{2+\theta}}{2t}}\bigg) I_\nu(z) \notag\\
    &\quad\quad\quad\quad\quad\quad+ |xy|^{\frac{(1+\theta)}{2}}\frac{1+\frac{\theta}{2}}{t}y|xy|^{\frac{\theta}{2}}e^{-\frac{|x|^{2+\theta}+|y|^{2+\theta}}{2t}}\frac{\partial}{\partial z}(I_\nu(z))\bigg)\phi(y)\dd y\notag\\
    &= \frac{(2+\theta)}{2t}\int_{\mathbb{R}}\bigg(\frac{1+\theta}{2}y|xy|^{\frac{(\theta-1)}{2}}e^{-\frac{|x|^{2+\theta}+|y|^{2+\theta}}{2t}} \notag\\
    &\qquad\qquad\qquad- \frac{2+\theta}{2t}|x|^{1+\theta}|xy|^{\frac{1+\theta}{2}}e^{-\frac{|x|^{2+\theta}+|y|^{2+\theta}}{2t}}\bigg) I_\nu\bigg( \frac{|xy|^{1+\frac{\theta}{2}}}{t}\bigg)\phi(y)\dd y \notag\\
    &\quad+ \frac{(2+\theta)}{2t}\int_{\mathbb{R}}\bigg(|xy|^{\frac{(1+\theta)}{2}}\frac{1+\frac{\theta}{2}}{t}y|xy|^{\frac{\theta}{2}}e^{-\frac{|x|^{2+\theta}+|y|^{2+\theta}}{2t}} \notag\\
    &\qquad\qquad\qquad\bigg( \nu\frac{t}{|xy|^{1+\frac{\theta}{2}}} I_\nu\bigg( \frac{|xy|^{1+\frac{\theta}{2}}}{t}\bigg)
    +I_{\nu+1}\bigg( \frac{|xy|^{1+\frac{\theta}{2}}}{t} \bigg) \bigg)\bigg)\phi(y)\dd y \notag\\
    &\leq C_{t,\phi}\int_{S_\phi}\bigg(|xy|^{\frac{(1+\theta)}{2}}e^{-\frac{|x|^{2+\theta}+|y|^{2+\theta}}{2t}} \notag\\
    &\qquad\qquad\qquad+ |x|^{1+\theta}|xy|^{\frac{(1+\theta)}{2}}e^{-\frac{|x|^{2+\theta}+|y|^{2+\theta}}{2t}}\bigg) I_\nu\bigg( \frac{|xy|^{1+\frac{\theta}{2}}}{t}\bigg)\dd y \notag\\
    &\quad+ \int_{S_\phi}\bigg(|xy|^{\frac{(1+\theta)}{2}}e^{-\frac{|x|^{2+\theta}+|y|^{2+\theta}}{2t}} \bigg( I_\nu\bigg( \frac{|xy|^{1+\frac{\theta}{2}}}{t}\bigg)
    +I_{\nu+1}\bigg( \frac{|xy|^{1+\frac{\theta}{2}}}{t} \bigg) \bigg)\bigg)\dd y \notag\\
    &\leq C_{t,\phi}\int_{S_\phi} |x|^{1+\theta}|xy|^{\frac{(1+\theta)}{2}}e^{-\frac{|x|^{2+\theta}+|y|^{2+\theta}}{2t}}\bigg( I_\nu\bigg( \frac{|xy|^{1+\frac{\theta}{2}}}{t}\bigg)
    +I_{\nu+1}\bigg( \frac{|xy|^{1+\frac{\theta}{2}}}{t} \bigg) \bigg)\bigg)\dd y,
  \end{align}
  where $S_\phi:=\lbrace y\in\R\colon \phi(y)\neq 0 \rbrace$. The integrands in \eqref{resubstitute} vanish for $y=0$ by the definition of $I_\nu$ in \eqref{def:Besselfunction} with $\nu=\frac{1}{2+\theta}-1<\frac{1+\theta}{2}$. If we thus show that, for any $\nu>-1$, there is a constant $C_\nu>0$ such that
  \begin{equation}\label{bound_bessel}
    I_\nu(z)+I_{\nu+1}(z)\leq C_\nu\big( z^{\nu+1}+z^{\nu+2} \big)e^z
  \end{equation}
  holds for all $z>0$, then the statement will follow, since, similar as in \eqref{exponentterm}, all the $x$-polynomials in \eqref{resubstitute} and the Bessel function terms are dominated by the term $e^{-\frac{|x|^{2+\theta}}{2t}}$ and the $y$ terms can be bounded using the compact support of~$\phi$.

  To get \eqref{bound_bessel}, we use the equality (see \cite[(5.7.9), page~110]{Lebedev1972})
  \begin{equation}\label{Besselbound1}
    I_\nu(z)=2(\nu+1)I_{\nu+1}(z)+I_{\nu+2}(z),
  \end{equation}
  and, since $\nu+1,\nu+2>-\frac{1}{2}$, we can then apply the following inequality from \cite[(6.25), page 63]{Luke1972}, for $x>0$:
  \begin{equation}\label{Besselbound}
    I_\nu(x)<\frac{e^x+e^{-x}}{2\Gamma(\nu+1)}\bigg( \frac{x}{2} \bigg)^\nu<\frac{e^x}{\Gamma(\nu+1)}\bigg( \frac{x}{2} \bigg)^\nu.
  \end{equation}
  \eqref{Besselbound1} and \eqref{Besselbound} yield, as $\Gamma(x)>0$ for $x>0$, that
  \begin{align*}
	I_\nu(z)+I_{\nu+1}(z)
	&=2\bigg(\nu+\frac{3}{2}\bigg)I_{\nu+1}(z)+I_{\nu+2}(z)\\
	&< 2\bigg(\nu+\frac{3}{2}\bigg)\frac{e^z}{\Gamma(\nu+2)}\bigg( \frac{z}{2} \bigg)^{\nu+1} +\frac{e^z}{\Gamma(\nu+3)}\bigg( \frac{z}{2} \bigg)^{\nu+2} \\
	&\leq C_\nu\big( z^{\nu+1}+z^{\nu+2} \big)e^z,
  \end{align*}
  which proves \eqref{bound_bessel}.
\end{proof}

\begin{proposition}\label{prop:step5}
  It holds that
  \begin{equation*}
    \E[|\tilde{X}(t,0)|] \lesssim \int_0^t (t-s)^{-\alpha}\E[|\tilde{X}(s,0)|]\dd s, \qquad t\in[0,T].
  \end{equation*}
\end{proposition}

\begin{proof}
  First, to apply Proposition~\ref{prop:step4}, we need to show that $\Psi_{N,M}$ defined in \eqref{def_psi} fulfills Assumption~\ref{ass:Psi}. $\Psi_{N,M}\in C^2([0,T]\times\R)$ and the conditions $\Psi_{N,M}(s,0)>0$ and $\Gamma(t)\in B(0,J(t))$ for some $J(t)>0$ follow by construction. Moreover, Lemma~\ref{lem:Lemma3} directly yields that the last property holds:
  \begin{align*}
    \sup_{s\leq t}\bigg| \int_{\mathbb{R}}|x|^{-\theta}\bigg(\frac{\partial}{\partial x}\Psi_{N,M}(s,x)\bigg)^2\dd x \bigg|
    &\leq C \int_{\mathbb{R}}|x|^{-\theta}e^{-2\lambda |x|}\dd x,
  \end{align*}
  which is clearly finite as $\theta<1$. Hence, Assumption~\ref{ass:Psi} holds.

  Thus, Proposition~\ref{prop:step4} holds and plugging \eqref{def_psi} into \eqref{absch_kurzvorende}, sending $K\to\infty$ such that $T_{\zeta,K}\to T$ by Corollary~\ref{Cor:StoppzeitggT} and using Corollary~\ref{cor:simplify}, \eqref{bounddd} and Lemma~\ref{lem:Lemma3}, we get
  \begin{align}\label{eqfastfertig}
    &\int_{\mathbb{R}}\mathbb{E}[ |\tilde{X}(t,x)|]\phi_M(x)g_N(x)\dd x\notag\\
    &\quad\lesssim \int_0^{t}\int_{\mathbb{R}}\mathbb{E}[|\tilde{X}(s,x)|]
    \bigg| 4\alpha^2 |x|^{-\theta}\Big( \frac{\partial}{\partial x}S_{t-s}\phi_M(x) \Big)\Big( \frac{\partial}{\partial x}g_N(x) \Big) + S_{t-s}\phi_M(x)\Delta_\theta g_N(x)\bigg| \dd x\dd s\notag\\
    &\qquad+ \int_0^t \Psi_{N,M}(s,0) \mathbb{E}[|\tilde{X}(s,0)| ]\dd s\notag\\
    &\quad\lesssim \int_0^{t}\int_{\mathbb{R}}\mathbb{E}[|\tilde{X}(s,x)|] \Big(\frac{\partial}{\partial x} S_{t-s} \phi_M(x) \Big) + S_{t-s} \phi_M(x) |\mathbbm{1}_{\lbrace N+1>|x|>N\rbrace} \dd x\dd s\notag\\
    &\qquad+ \int_0^t \Psi_{N,M}(s,0) \mathbb{E}[|\tilde{X}(s,0)|]\dd s\notag\\
    &\quad\lesssim\int_0^{t}\int_{\mathbb{R}}\mathbb{E}[|\tilde{X}(s,x)|] e^{-\lambda|x|}\mathbbm{1}_{\lbrace N+1>|x|>N \rbrace} \dd x\dd s + \int_0^t \Psi_{N,M}(s,0) \mathbb{E}[|\tilde{X}(s,0)|]\dd s.
  \end{align}
  We want to send $N,M\to\infty$. By Proposition~\ref{prop: regularity result} (i) we get that
  \begin{align*}
    &\int_0^{t}\int_{\mathbb{R}}\mathbb{E}[|\tilde{X}(s,x)|] e^{-\lambda|x|}\mathbbm{1}_{\lbrace N+1>|x|>N \rbrace} \dd x\dd s
    \lesssim t \int_{N}^{N+1}e^{-\lambda x} \dd x\to 0\quad \text{as }N\to \infty.
  \end{align*}
  Moreover, we get
  \begin{align*}
    \int_0^t \Psi_{N,M}(s,0) \mathbb{E}[|\tilde{X}(s,0)| ]\dd s
    &=\int_0^t ( S_{t-s}\phi_M(0))g_N(x) \mathbb{E}[|\tilde{X}(s,0)|]\dd s\\
    &= \int_0^t \bigg( \int_\R p_{t-s}^\theta (y,0)\phi_M(y)\dd y \bigg) \mathbb{E}[|\tilde{X}(s,0)| ]\dd s\\
    &\stackrel{M\to\infty}{\to} \int_0^t p_{t-s}^\theta(0) \mathbb{E}[|\tilde{X}(s,0)| ]\dd s
    \quad \text{as }M\to \infty,
  \end{align*}
  which gives
  \begin{equation*}
    \int_0^t \Psi_{N,M}(s,0) \mathbb{E}[|\tilde{X}(s,0)| ]\dd s
    = c_\theta \int_0^t (t-s)^{-\alpha} \mathbb{E}[|\tilde{X}(s,0)|]\dd s.
  \end{equation*}
  Hence, sending $N,M\to\infty$ in \eqref{eqfastfertig} yields
  \begin{align*}
    \mathbb{E}[ |\tilde{X}(t,0)| \lesssim \int_0^t (t-s)^{-\alpha}\mathbb{E}[|\tilde{X}(s,0)|]\dd s.
  \end{align*}
\end{proof}

\bibliography{literature}{}
\bibliographystyle{amsalpha}

\end{document}